\documentclass[11pt]{article}
\usepackage{amssymb}
\usepackage{amsfonts}
\usepackage{amsmath}
\usepackage{graphicx}
\usepackage{accents}

\newtheorem{theorem}{Theorem}[section]

\newtheorem{corollary}[theorem]{Corollary}

\newtheorem{lemma}[theorem]{Lemma}

\newtheorem{proposition}[theorem]{Proposition}
\newtheorem{remark}[theorem]{Remark}

\newenvironment{proof}[1][Proof]{\textbf{#1.} }{\hfill\rule{0.5em}{0.5em}}
{\catcode`\@=11\global\let\AddToReset=\@addtoreset
\AddToReset{equation}{section}

\AddToReset{theorem}{section}

\begin{document}

\title{A quasilinear elliptic equation with absorption term and Hardy
potential}
\author{Marie-Fran\c{c}oise Bidaut-V\'{e}ron\thanks{%
veronmf@univ-tours.fr- Institut Denis Poisson, CNRS-UMR 7013-Universit\'{e}
de Tours, 37200 Tours,France} \and Huyuan Chen\thanks{%
chenhuyuan@yeah.net-Department of Mathematics, Jiangxi Normal
University-Nanchang, Jianxi 330022, PR China}}
\date{}
\maketitle

\begin{abstract}
Here we study the positive solutions of the equation 
\begin{equation*}
-\Delta _{p}u+\mu \frac{u^{p-1}}{\left\vert x\right\vert ^{p}}+\left\vert
x\right\vert ^{\theta }u^{q}=0,\qquad x\in \mathbb{R}^{N}\backslash \left\{
0\right\} ,
\end{equation*}%
where $\Delta _{p}u= {div}(\left\vert \nabla u\right\vert ^{p-2}\nabla
u) $ and $1<p<N,q>p-1,\mu ,\theta \in \mathbb{R}.$ We give a complete
description of the existence and the asymptotic behaviour of the solutions
near the singularity $0,$ or in an exterior domain. We show that the global
solutions $\mathbb{R}^{N}\backslash \left\{ 0\right\} $ are radial and give
their expression according to the position of the Hardy coefficient $\mu $
with respect to the critical exponent $\mu _{0}=-(\frac{N-p}{p})^{p}.$ Our
method consists into proving that any nonradial solution can be compared to
a radial one, then making exhaustive radial study by phase-plane techniques.
Our results are optimal, extending the known results when $\mu =0$ or $p=2$,
with new simpler proofs. They make in evidence interesting phenomena of
nonuniqueness when $\theta +p=0$, and of existence of locally constant
solutions when moreover $p>2$ .
\end{abstract}

\tableofcontents

\section{Introduction}

Here we study the positive solutions $u$ in a domain $\Omega $ of $\mathbb{R}%
^{N}$ of the quasilinear equation with a Hardy term 
\begin{equation}
-\Delta _{p}u+\mu \frac{u^{p-1}}{\left\vert x\right\vert ^{p}}+\left\vert
x\right\vert ^{\theta }u^{q}=0,  \label{pq}
\end{equation}%
where $\Delta _{p}$ is the $p$-Laplace operator $u\longmapsto \Delta _{p}u=%
 {div}(\left\vert \nabla u\right\vert ^{p-2}\nabla u)$ and 
\begin{equation*}
1<p<N,\qquad q>p-1,\qquad \mu ,\theta \in \mathbb{R}.
\end{equation*}%
We consider the problem of the isolated singularities at $0$ with $\Omega
=B_{r_{0}}\backslash \left\{ 0\right\} $ (a priori estimates, description of
the asymptotic behaviour of the solutions, existence of local solutions of
any possible type), the problem in an exterior domain with $\Omega =\mathbb{R%
}^{N}\backslash \overline{B_{r_{0}}}$ (same questions), and the global
problem in $\mathbb{R}^{N}\backslash \left\{ 0\right\} $ (existence of
solutions, radiality or nonradiality of the possible solutions).\bigskip

The case $p=2$ and $q>1,\theta =0$ was studied by Guerch and Veron \cite%
{GuVe}, who gave the precise behaviour near $0$ of the solutions of any sign
of equation 
\begin{equation}
-\Delta u+\mu \frac{u}{\left\vert x\right\vert ^{2}}+g(u)=0  \label{2g}
\end{equation}%
and isotropy results when $g(u)$ has a power-like growth. They extended many
properties the solutions of the classical problem 
\begin{equation*}
-\Delta u+u^{q}=0
\end{equation*}%
object of an impressive number of articles, starting to the pionneer papers 
\cite{BrVe},\cite{Ve81}, and also \cite{BrOs}.It is clear that problem (\ref%
{2g}) is deeply linked to the Hardy operator 
\begin{equation*}
v\longmapsto \mathcal{L}_{2,\mu }v=-\Delta v+\mu \frac{v}{\left\vert
x\right\vert ^{2}}
\end{equation*}%
for which the value $\mu =\mu _{0}=-(\frac{N-2}{2})^{2}$ plays an essential
role, due to the Hardy inequality in bounded $\Omega $ 
\begin{equation*}
\int_{\Omega }\left\vert \nabla u\right\vert ^{2}dx\geq (\frac{N-2%
}{2})^{2}\int_{\Omega }\frac{u^{2}}{\left\vert x\right\vert ^{2}}%
dx,\qquad \forall u\in W_{0}^{1,2}(\Omega ).
\end{equation*}%
In \cite{Ci}, Cirstea who considered the positive solutions of 
\begin{equation*}
-\Delta u+\mu \frac{u}{\left\vert x\right\vert ^{2}}+b(x)g(u)=0,
\end{equation*}%
where $g(u)/u^{q}$ and $b(x)/\left\vert x\right\vert ^{\theta }$ may have a
logarithmic behaviour, giving a precise behaviour near $0$, in the case $\mu
\geq -(\frac{N-2}{2})^{2}$, and near $\infty $ by Kelvin transform. In the
case of equation 
\begin{equation}
-\Delta u+\mu \frac{u}{\left\vert x\right\vert ^{2}}+\left\vert x\right\vert
^{\theta }u^{q}=0  \label{casp2}
\end{equation}%
with $\mu <-(\frac{N-2}{2})^{2}$ and $\theta +2>0$, Wei and Du \cite{WeDu}
gave a precise behaviour near $0$, and a uniqueness result of global
solutions in $\mathbb{R}^{N}\backslash \left\{ 0\right\} $. The problem of
existence and radiality of global solutions was then solved in \cite{CiFa}
for any value of $\theta $, by using thin techniques of supersolutions and
subsolutions.\bigskip

The case $p=2$ and $q<1$ was treated in \cite{BiGr}, showing in particular
the existence of many nonradial solutions, with possible dead cores.\bigskip

In case $p>1$, the quasilinear equation where $\mu =0,\theta =0,q>p-1$, 
\begin{equation}
-\Delta _{p}u+u^{q}=0  \label{sim}
\end{equation}%
was studied by \cite{FrVe} and \cite{VaVe}. Concerning the behaviour of any
solution $u$ in $B_{r_{0}}\backslash \left\{ 0\right\} $, it was shown that
three eventualities occur when $q<q_{c}=\frac{N(p-1)}{N-p}$: either $%
\lim_{\left\vert x\right\vert \longrightarrow 0}\left\vert x\right\vert ^{%
\frac{p}{q+1-p}}$ $u=C_{N,p,q}>0,$ or $\lim_{\left\vert x\right\vert
\longrightarrow 0}\left\vert x\right\vert ^{N-p}$ $u=k>0$, or $u$ can be
extended as a solution in $B_{r_{0}};$ and the singularity is removable if
and only if $q\geq q_{c}.$\bigskip

In case of quasilinear equation (\ref{pq}), the Hardy operator 
\begin{equation}
v\longmapsto \mathcal{L}_{p,\mu }v=-\Delta _{p}v+\mu \frac{v^{p-1}}{%
\left\vert x\right\vert ^{p}}  \label{lpmu}
\end{equation}%
plays an essential role, and the critical value is 
\begin{equation}
\mu _{0}=-(\frac{N-p}{p})^{p},  \label{muo}
\end{equation}%
coming from the extended Hardy inequality 
\begin{equation*}
\int_{\Omega }\left\vert \nabla u\right\vert ^{p}dx\geq (\frac{N-p%
}{p})^{p}\int_{\Omega }\frac{\left\vert u\right\vert ^{p}}{%
\left\vert x\right\vert ^{p}}dx,\qquad \forall u\in W_{0}^{1,p}(\Omega ).
\end{equation*}%
For a precise study of the Hardy operator, we refer for exemple to \cite%
{FraPi}. \medskip

Quasilinear equations involving this operator with a source term, of type 
\begin{equation*}
-\Delta _{p}u+\mu \frac{u^{p-1}}{\left\vert x\right\vert ^{p}}=u^{q},
\end{equation*}%
have been the object if an intensive study, starting from the pioneer
articles of Serrin \cite{Se64},\cite{Se65}, and \cite{GiSp}, \cite{CaGiSp}, 
\cite{SeZo}, when $\mu =0$; in particular in the case of critical growth $%
q=q_{s}=\frac{N(p-1)+p}{N-p.}$, for solutions with finite energy: the
problem of existence, uniqueness and behaviour of the radial solutions were
established in \cite{AbFePe}, sharp asymptotic estimates of the solutions
were given in \cite{Xi15}, \cite{Xi17}, and their radiality was obtained by
moving planes technique in \cite{OlScVa}.\bigskip

To our knowledge, equation (\ref{pq}) has been much less studied. We can
mention an article of \cite{FeTaWe} which extends the results of \cite{WeDu}
in restrictive conditions on the parameters. The Dirichlet problem in
bounded $\Omega $ for the nonhomogeneous equation $\mathcal{L}_{p,\mu
}v+v^{q}=h$ where $h\in L^{1}(\Omega )$, $h\neq 0$ was considered in \cite%
{Mi}. \bigskip

Here our purpose is to extend to equation (\ref{pq}) for any $p>1$ and $%
q>p-1 $, the study of \cite{Ci}, \cite{CiFa} relative to the case $p=2$, and
simplify the proofs, by a \textbf{quite different approach}. Our line of
attack is new, even in the case of equation (\ref{sim}). It is built on a
complete study of the radial case by phase-plane techniques. Indeed,
equation (\ref{pq}) is invariant by the scaling $T_{k}$ defined for any $k>0$
and $x\in \mathbb{R}^{N}\backslash \left\{ 0\right\} $ by 
\begin{equation}
T_{k}u(x)=k^{\gamma }u(kx),  \label{scal}
\end{equation}%
where 
\begin{equation}
\gamma =\frac{p+\theta }{q+1-p}.  \label{gam}
\end{equation}%
Then its radial formulation can be reduced to an autonomous system of order
2. A precise analysis of the phase-plane first implies many properties of
existence and possible uniqueness of local and global radial solutions. Then
our \textbf{key point} is to show that any nonradial solution can be
compared from above and below with a radial one, as shown at Theorem \ref%
{clef}. In that way we avoid a construction of supersolutions and
subsolutions, in general hard with a quasilinear equation such as (\ref{pq}%
). In particular we show that \textbf{all the global solutions in }$\mathbb{R%
}^{N}\backslash \left\{ 0\right\} $\textbf{\ are radial. }Our results are
optimal.

\section{Main results}

\subsection{Parameters of the study}

$\bullet $ We first consider the equation 
\begin{equation}
\mathcal{L}_{p,\mu }u=0  \label{Hardy}
\end{equation}%
and search of the form $u(x)=C\left\vert x\right\vert ^{-S},$ with $C>0;$
then 
\begin{equation}
\mathcal{L}_{p,\mu }u=0\Longleftrightarrow \varphi (S)=0;  \label{eqS}
\end{equation}%
where 
\begin{equation}
\varphi (S)=(p-1)\left\vert S\right\vert ^{p}-(N-p)\left\vert S\right\vert
^{p-2}S-\mu ;  \label{fau}
\end{equation}%
the function $\varphi $ admits a minimum value $\varphi (\frac{N-p}{p})=-(%
\frac{N-p}{p})^{p}-\mu =\mu _{0}-\mu $, thus such solutions exist if and
only if $\mu \geq \mu _{0}.$ \ In any case 
\begin{equation}
S_{1}>0,\quad \quad S_{2}\leq \frac{N-p}{p}\leq S_{1}.  \label{sel}
\end{equation}%
If $\mu >0$ we obtain two roots $S_{2}<0<S_{1}.$ If $\mu _{0}<\mu <0$ we
obtain two roots $0<S_{2}<S_{1}.$ If $\mu =0$, $S_{2}=0<S_{1}=\frac{N-p}{p-1}
$. If $\mu =\mu _{0}$, we find only one root $S_{1}=S_{2}=\frac{N-p}{p}$,
corresponding to solutions $u(x)=C\left\vert x\right\vert ^{-\frac{N-p}{p}}.$
Note that for $\mu =\mu _{0},$ equation (\ref{Hardy}) admits also radial
local positive solutions with a logarithmic behaviour :\textbf{\ }%
\begin{equation*}
\lim_{x\longrightarrow 0}\left\vert x\right\vert ^{\frac{N-p}{p}}\left\vert
\ln \left\vert x\right\vert \right\vert ^{-\frac{2}{p}}u(x)=C_{1}>0,\qquad 
\text{or }\mathbf{\ }\lim_{\left\vert x\right\vert \longrightarrow \infty
}\left\vert x\right\vert ^{\frac{N-p}{p}}\left\vert \ln \left\vert
x\right\vert \right\vert ^{-\frac{2}{p}}u(x)=C_{2}>0,
\end{equation*}%
(explicit and well known when\textbf{\ }$p=2$) see \textbf{\ \cite{ItaTa} }%
and a short proof at Lemma \ref{critic} below.\bigskip

\noindent $\bullet $ Next we consider equation (\ref{pq}), which takes the
form%
\begin{equation}
\mathcal{L}_{p,\mu }u+\left\vert x\right\vert ^{\theta }u^{q}=0  \label{pql}
\end{equation}%
and search a particular power solution of the form $u^{\ast }(x)=a^{\ast
}\left\vert x\right\vert ^{-\lambda }.$ If it exists, then necessarily $%
\lambda =\gamma $, defined at (\ref{gam}) and 
\begin{equation}
u^{\ast }(x)=a^{\ast }\left\vert x\right\vert ^{-\gamma }\text{ and }%
(a^{\ast })^{q+1-p}=\varphi (\gamma )>0,  \label{condi}
\end{equation}%
and the existence of $u$ only depends on the position of $\gamma $ with
respect to $S_{1},S_{2}.\bigskip $

\noindent $\bullet $ In the sequel we divide our analysis into 5 assumptions:%
\begin{equation}
\left\{ 
\begin{array}{cc}
\text{(}\mathcal{H}_{1}\text{):} & \mu \geq \mu _{0}\quad \text{and }\gamma
>S_{1}, \\ 
\text{(}\mathcal{H}_{2}\text{)}\text{: } & \mu \geq \mu _{0}\quad \text{and }%
\gamma <S_{2}, \\ 
\text{(}\mathcal{H}_{3}\text{)}\text{:} & \mu >\mu _{0}\quad \text{and }%
S_{2}\leq \gamma \leq S_{1}, \\ 
\text{(}\mathcal{H}_{4}\text{):} & \mu =\mu _{0}\quad \text{and }\gamma =%
\frac{N-p}{p}, \\ 
\text{(}\mathcal{H}_{5}\text{)}\text{:} & \mu <\mu _{0}.%
\end{array}%
\right.  \label{cases}
\end{equation}%
Then the solution $u^{\ast }$ exists exclusively in cases ($\mathcal{H}_{1}$%
),($\mathcal{H}_{2}$) ($\mathcal{H}_{5}$).\medskip

\ We note that $\gamma $\textbf{\ }has the sign of\textbf{\ }$p+\theta .$\
These different cases do not involve the sign of $\gamma .$ but a main
difference is the behaviour of the solution $u^{\ast }$, when it exists:
when $p+\theta >0$, then $u^{\ast }$ is singular, decreasing from $\infty $
as $r\rightarrow 0$ to $0$ as $r\rightarrow \infty ;$ when $p+\theta <0$,
then $u^{\ast }\in C(\left[ 0,\infty \right) $, increasing from $0$ to $%
\infty $.

When $p+\theta =0$, that means $\gamma =0$, equation (\ref{pq}) takes the
form 
\begin{equation}
-\Delta _{p}u+\frac{u^{p-1}(u^{q+1-p}+\mu )}{\left\vert x\right\vert ^{p}}=0,
\label{pqzero}
\end{equation}%
so that for $\mu <0,\mu \neq \mu _{0}$, the function $u^{\ast }\equiv
\left\vert \mu \right\vert ^{\frac{1}{q+1-p}}$ is a \textbf{constant solution%
}. This case has a specific interest. Indeed there hold phenomena of \textbf{%
nonuniqueness} of solutions in a neighborhood of $0$, such that $%
\lim_{r\longrightarrow 0}u=\left\vert \mu \right\vert ^{\frac{1}{q+1-p}}$,
possibly global, see Remark \ref{nonuniq}. Moreover when $p>2$, we show the
existence of \textbf{locally constant solutions}, but not identically
constant, near $0$ or $\infty $, and possibly global. It can happen under ($%
\mathcal{H}_{2}$) when $\mu _{0}\leq \mu <0$, or ($\mathcal{H}_{5}$).

\subsection{Local and global behaviour of the solutions}

Next we give our main results on the behaviour near $0$ or $\infty $ of the
solutions, and on the existence and possible uniqueness of the global
solutions.\ \textbf{\ }All the following theorems extend to the case $p>1$
the results of \cite{Ci} and \cite{CiFa} relative to equation (\ref{casp2}).
Note that the behaviour of the solutions in $\mathbb{R}^{N}\backslash 
\overline{B_{r_{0}}}$ cannot be obtained from the ones in $%
B_{r_{0}}\backslash \left\{ 0\right\} $ by Kelvin transform, see also Remark %
\ref{Kelvin} below.

We say that $u\geq 0$ is a solution of equation (\ref{pq}) in $\Omega
=B_{r_{0}}\backslash \left\{ 0\right\} $ (resp. $\Omega =\mathbb{R}%
^{N}\backslash \overline{B_{r_{0}}}$) if $u\in L_{loc}^{\infty }(\Omega )$, $%
\left\vert \nabla u\right\vert \in L_{loc}^{p}(\Omega )$ and 
\begin{equation*}
\int_{\Omega }\left\vert \nabla u\right\vert ^{p-2}\nabla u.\nabla
\varphi dx+\int_{\Omega }(\mu \frac{u^{p-1}}{\left\vert
x\right\vert ^{p}}+\left\vert x\right\vert ^{\theta }u^{q})\varphi
dx=0,\qquad \forall \varphi \in C_{c}^{1}(\Omega ).
\end{equation*}%
It follows from \cite[Theorem 1]{To} that $u\in C^{1}(\Omega )$.

Our results are given following the different assumptions (\ref{cases}):

\begin{theorem}
\label{H1nonrad}\textbf{Case} ($\mathcal{H}_{1}$) Let $\mu \geq \mu _{0}$
and $\gamma >S_{1}.$

\noindent (i) Let $u$ be any positive solution in $B_{r_{0}}\backslash
\left\{ 0\right\} .$ Then

$\bullet $ either 
\begin{equation}
\lim_{x\rightarrow 0}\left\vert x\right\vert ^{\gamma }u=a^{\ast },
\label{ra}
\end{equation}

$\bullet $ or 
\begin{equation}
\left\{ 
\begin{array}{ccc}
\lim_{x\rightarrow 0}\left\vert x\right\vert ^{S_{1}}u=k_{1}>0, &  & \text{%
if }\mu >\mu _{0}, \\ 
\lim_{x\rightarrow 0}\left\vert x\right\vert ^{\frac{N-p}{p}}(\left\vert \ln
\left\vert x\right\vert \right\vert ^{-\frac{2}{p}})u=\ell >0, &  & \text{if 
}\mu =\mu _{0},%
\end{array}%
\right.  \label{rb}
\end{equation}%
\noindent

$\bullet $ or 
\begin{equation}
\lim_{x\rightarrow 0}\left\vert x\right\vert ^{S_{2}}u=k_{2}>0.  \label{rc}
\end{equation}

\noindent (ii) Let $u$ be any positive solution in $\mathbb{R}^{N}\backslash 
\overline{B_{r_{0}}}$. Then

\begin{equation}
\lim_{\left\vert x\right\vert \rightarrow \infty }\left\vert x\right\vert
^{\gamma }u=a^{\ast }.  \label{rd}
\end{equation}%
There exist solutions of each type.\medskip

\noindent (iii) There exist global solutions in $\mathbb{R}^{N}\backslash
\left\{ 0\right\} $, all of them are radial:

$\bullet $ $u^{\ast }=a^{\ast }\left\vert x\right\vert ^{-\gamma },$

$\bullet $ if $\mu >\mu _{0},$ for any $k_{1}>0$, there exists a unique
global solution such that 
\begin{equation}
\lim_{r\rightarrow 0}\left\vert x\right\vert ^{S_{1}}u=k_{1}>0,\qquad
\lim_{r\rightarrow \infty }\left\vert x\right\vert ^{\gamma }u=a^{\ast },
\label{re}
\end{equation}

$\bullet $ if $\mu =\mu _{0},$ for any $\ell >0,$ there exists a unique
global solution such that%
\begin{equation}
\lim_{x\rightarrow 0}\left\vert x\right\vert ^{\frac{N-p}{p}}\left\vert \ln
\left\vert x\right\vert \right\vert ^{-\frac{2}{p}}u=\ell ,\qquad
\lim_{\left\vert x\right\vert \rightarrow \infty }\left\vert x\right\vert
^{\gamma }u=a^{\ast }.  \label{rf}
\end{equation}
\end{theorem}

\begin{theorem}
\label{H2nonrad}\textbf{Case} ($\mathcal{H}_{2}$): Let $\mu \geq \mu _{0}$
and $\gamma <S_{2}$

\noindent (i) Let $u$ be any positive solution in $B_{r_{0}}\backslash
\left\{ 0\right\} $. Then%
\begin{equation}
\lim_{x\rightarrow 0}\left\vert x\right\vert ^{\gamma }u=a^{\ast }.
\label{aut1}
\end{equation}%
(ii) Let $u$ be any positive solution in $\mathbb{R}^{N}\backslash \overline{%
B_{r_{0}}}$. Then

\noindent $\bullet $ either 
\begin{equation}
\lim_{\left\vert x\right\vert \rightarrow \infty }\left\vert x\right\vert
^{\gamma }u=a^{\ast },  \label{aut2}
\end{equation}%
$\bullet $ or 
\begin{equation}
\left\{ 
\begin{array}{ccc}
\lim_{\left\vert x\right\vert \rightarrow \infty }\left\vert x\right\vert
^{S_{2}}u=k_{2}>0, &  & \text{if }\mu >\mu _{0}, \\ 
\lim_{\left\vert x\right\vert \rightarrow \infty }\left\vert x\right\vert ^{%
\frac{N-p}{p}}\left\vert \ln \left\vert x\right\vert \right\vert ^{-\frac{2}{%
p}}u=\ell >0, &  & \text{if }\mu =\mu _{0},%
\end{array}%
\right.  \label{autnew}
\end{equation}%
$\bullet $ or 
\begin{equation*}
\lim_{\left\vert x\right\vert \rightarrow \infty }\left\vert x\right\vert
^{S_{1}}u=k_{1}>0.
\end{equation*}%
there exist solutions of each type.

\noindent (iii) There exist global solutions in $\mathbb{R}^{N}\backslash
\left\{ 0\right\} $, given by:

$\bullet $ $u^{\ast }=a^{\ast }\left\vert x\right\vert ^{-\gamma }.$

$\bullet $ if $\mu >\mu _{0}$, for any $k_{2}>0$, there exist a unique
global solution (then radial), such that 
\begin{equation}
\lim_{x\longrightarrow 0}\left\vert x\right\vert ^{\gamma }u=a^{\ast
}.\qquad \lim_{\left\vert x\right\vert \rightarrow \infty }\left\vert
x\right\vert ^{S_{2}}u=k_{2}>0,  \label{aut4}
\end{equation}

$\bullet $ if $\mu =\mu _{0},$ for any $\ell >0,$ there exists a unique
global solution such that%
\begin{equation}
\lim_{x\longrightarrow 0}\left\vert x\right\vert ^{\gamma }u=a^{\ast
},\qquad \lim_{\left\vert x\right\vert \rightarrow \infty }\left\vert
x\right\vert ^{\frac{N-p}{p}}(\left\vert \ln \left\vert x\right\vert
\right\vert ^{\frac{2}{p}})u(x)=\ell >0.  \label{aut3}
\end{equation}
\end{theorem}

\begin{theorem}
\label{H3nonrad} \textbf{Case} ($\mathcal{H}_{3}$) Let $\mu >\mu _{0}$ and $%
S_{2}\leq \gamma \leq S_{1}.$Then

\noindent (i) Let $u$ be any positive solution in $B_{r_{0}}\backslash
\left\{ 0\right\} $. Then $\gamma <S_{1}$ and 
\begin{eqnarray*}
\lim_{x\longrightarrow 0}\left\vert x\right\vert ^{S_{2}}u &=&k_{2}>0\qquad 
\text{if }S_{2}<\gamma <S_{1}, \\
\lim_{x\longrightarrow 0}\left\vert x\right\vert ^{S_{2}}(\left\vert \ln
\left\vert x\right\vert \right\vert )^{-\frac{1}{q+1-p}}u &=&\alpha
_{N,p,q}>0\qquad \text{if }\gamma =S_{2}\neq 0, \\
\lim_{x\longrightarrow 0}\left\vert \ln \left\vert x\right\vert \right\vert
^{-\frac{p-1}{q+1-p}}u &=&\delta _{N,p,q}>0\qquad \text{if }\gamma =S_{2}=0.
\end{eqnarray*}%
(ii) Let $u$ be any positive solution in $\mathbb{R}^{N}\backslash \overline{%
B_{r_{0}}}$. Then $S_{2}<\gamma $ and 
\begin{eqnarray*}
\lim_{\left\vert x\right\vert \longrightarrow \infty }\left\vert
x\right\vert ^{S_{1}}u &=&k_{1}>0\qquad \text{if }S_{2}<\gamma <S_{1}, \\
\lim_{\left\vert x\right\vert \longrightarrow \infty }\left\vert
x\right\vert ^{S_{1}}(\left\vert \ln \left\vert x\right\vert \right\vert )^{-%
\frac{1}{q+1-p}}u &=&\beta _{N,p,q}\qquad \text{if }\gamma =S_{1}.
\end{eqnarray*}%
There exist solutions of each type.

\noindent (iii) There is no global positive solution.
\end{theorem}

\begin{theorem}
\label{H4nonrad}\textbf{Case} ($\mathcal{H}_{4}$) Let $\mu =\mu _{0}$ and $%
\gamma =\frac{N-p}{p}$. Then

\noindent (i) there is no global positive solution.

\noindent (ii) All the solutions in $B_{r_{0}}\backslash \left\{ 0\right\} .$
(resp. $\mathbb{R}^{N}\backslash \overline{B_{r_{0}}}$) satisfy 
\begin{equation*}
\lim_{x\longrightarrow 0}\left\vert x\right\vert ^{\frac{N-p}{p}}\left\vert
\ln \left\vert x\right\vert \right\vert ^{\frac{2}{q+1-p}}u=c_{N,p,q}>0,%
\text{ (resp. }\lim_{\left\vert x\right\vert \longrightarrow \infty
}\left\vert x\right\vert ^{\frac{N-p}{p}}\left\vert \ln \left\vert
x\right\vert \right\vert ^{\frac{2}{q+1-p}}u=c_{N,p,q}).
\end{equation*}%
There exist solutions of each type for any $p>1,$moreover \textbf{explicit}
if $p=2$.
\end{theorem}

The last result concerns the case $(\mathcal{H}_{5}$) where $\mu <\mu _{0}.$
It extends the main result of \cite[Theorem 1.1]{WeDu} and the one of \cite[%
Theorem 1.1]{CiFa} to the case of quasilinear equation (\ref{pq}), and the
proof is quite shorter.

\begin{theorem}
\label{H5nonrad}\textbf{Case} ($\mathcal{H}_{5}$) Let $\mu <\mu _{0}$. Then

\noindent (i) there is the unique global solution in $\mathbb{R}%
^{N}\backslash \left\{ 0\right\} $, radial: 
\begin{equation}
u^{\ast }(x)=a^{\ast }\left\vert x\right\vert ^{-\gamma };  \label{ne}
\end{equation}%
(ii) any solution in $B_{r_{0}}\backslash \left\{ 0\right\} .$ (resp. $%
\mathbb{R}^{N}\backslash \overline{B_{r_{0}}}$) satisfies%
\begin{equation}
\lim_{\left\vert x\right\vert \longrightarrow 0}\left\vert x\right\vert
^{\gamma }u=a^{\ast }\text{ (resp. }\lim_{\left\vert x\right\vert
\longrightarrow \infty }\left\vert x\right\vert ^{\gamma }u=a^{\ast }\text{).%
}  \label{nou}
\end{equation}%
Moreover if $\gamma =0$ and $p>2$, then $u$ is constant for small $%
\left\vert x\right\vert $ small enough (resp. for large $\left\vert
x\right\vert $). There exist local solutions of each type.
\end{theorem}

\subsection{Other formulation of the classification (\protect\ref{cases})%
\protect\bigskip}

We can formulate the different assumptions of (\ref{cases}) in another way,
distinguish according to the sign of $\gamma $: when $\mu \geq \mu _{0}$, we
define 
\begin{equation*}
\mathbf{q}_{1}=p-1+\frac{p+\theta }{S_{1}},\qquad \mathbf{q}_{2}=p-1+\frac{%
p+\theta }{S_{2}},\quad \text{if }\mu \neq 0.
\end{equation*}%
These two critical values of $q$ can be involved, according to the value of $%
\theta $ and $\mu $. For the case $\mu =\mu _{0}$, 
\begin{equation*}
\mathbf{q}_{1}=\mathbf{q}_{2}=\mathbf{q}_{s}:=\frac{N(p-1)+p+p\theta }{N-p}
\end{equation*}%
is the Sobolev radial exponent.\medskip

We get the following equivalences:

\noindent $\bullet $ when $p+\theta >0$, then $\mathbf{q}_{1}\leq \mathbf{q}%
_{2},$ and there holds $\mathbf{q}_{1}>p-1$, and $\mathbf{q}%
_{2}>p-1\Longleftrightarrow \mu <0$, and 
\begin{align*}
\text{(}\mathcal{H}_{1}\text{)}& \Longleftrightarrow \text{ }\mu \geq \mu
_{0}\qquad \text{and\quad\ }1<q<\mathbf{q}_{1}, \\
\text{(}\mathcal{H}_{2}\text{)}& \Longleftrightarrow 0>\mu \geq \mu
_{0}\qquad \text{and\quad\ }\mathbf{q}_{2}<q, \\
\text{(}\mathcal{H}_{3}\text{)}& \Longleftrightarrow \text{ }(\mu \geq 0%
\text{ and }\mathbf{q}_{1}\leq q)\quad \text{or\quad }(0>\mu >\mu _{0}\text{
and }\mathbf{q}_{1}\leq q\leq \mathbf{q}_{2}), \\
\text{(}\mathcal{H}_{4}\text{)}& \Longleftrightarrow \mu =\mu _{0}\qquad 
\text{and}\quad q=\mathbf{q}_{s};
\end{align*}

\noindent $\bullet $ when $p+\theta <0$, there holds $\mathbf{q}_{1}<p-1$,
and $\mathbf{q}_{2}>p-1\Longleftrightarrow \mu >0$, then ($\mathcal{H}_{1}$)
and ($\mathcal{H}_{4}$) are empty,%
\begin{eqnarray*}
\text{(}\mathcal{H}_{2}\text{)} &\Longleftrightarrow &\mu \leq 0\qquad \text{%
or\quad (}\mu >0\text{ and }1<q\leq \mathbf{q}_{2}\text{)}, \\
\text{(}\mathcal{H}_{3}\text{)} &\Longleftrightarrow &\text{ }\mu >0\qquad 
\text{and\quad }\mathbf{q}_{2}\leq q;
\end{eqnarray*}%
$\bullet $ when $p+\theta =0$, then $\gamma =0$, so ($\mathcal{H}_{1}$) and (%
$\mathcal{H}_{4}$) are still empty, and 
\begin{equation*}
\text{ (}\mathcal{H}_{2}\text{)}\Longleftrightarrow 0>\mu \geq \mu
_{0},\qquad \text{(}\mathcal{H}_{3}\text{) }\Longleftrightarrow \mu \geq 0,
\end{equation*}%
and $u^{\ast }$ exists for any $\mu <0$ and is constant: $u^{\ast }\equiv
a^{\ast }=(-\mu )^{\frac{1}{q+1-p}}.$\bigskip

Let us add a simple corollary relative to the case $\mu =0$, still studied
in \cite{FrVe} for $\theta =0$, and \cite{CiDu} for $\theta +p>0.$ We find
again and complete their results, in particular when $\theta +p\leq 0$,
which to our knowledge, is new: the existence of solutions in an exterior
domain such that $\lim_{\left\vert x\right\vert \rightarrow \infty
}u=k_{2}>0 $, first obtained in \cite{BiGa} is a \textbf{surprising}
significant result. Here $S_{2}=0$, $S_{1}=\frac{N-p}{p-1}$, and there is
only one critical value 
\begin{equation}
\mathbf{q}_{1}=\mathbf{q}_{c}=\frac{(p-1)(N+\theta )}{N-p}.  \label{qc}
\end{equation}%
When $\theta +p>0$, ($\mathcal{H}_{1}$)$\Longleftrightarrow $ $1<q<\mathbf{q}%
_{1}$, ($\mathcal{H}_{3}$)$\Longleftrightarrow $ $\mathbf{q}_{1}\leq q$, and
the other cases do not hold. When $\theta +p<0$, then ($\mathcal{H}_{2}$)
always hold.

\begin{corollary}
\label{muzero}Consider the Henon type equation with $\theta \in \mathbb{R}:$%
\begin{equation}
-\Delta _{p}u+\left\vert x\right\vert ^{\theta }u^{q}=0.  \label{ep}
\end{equation}

\noindent (i) Suppose $\theta +p>0$ and $q<\mathbf{q}_{c}$, where $\mathbf{q}%
_{c}$ is defined at (\ref{qc}).Then any solution $u\not\equiv 0$ in $%
B_{r_{0}}\backslash \left\{ 0\right\} $ satisfies either $\lim_{\left\vert
x\right\vert \rightarrow 0}\left\vert x\right\vert ^{\gamma }u=a^{\ast }$,
or $\lim_{\left\vert x\right\vert \rightarrow 0}\left\vert x\right\vert
^{N-p}u=k_{1}>0$, or $u$ extends as a solution in $B_{r_{0}}.$ Any positive
solution in $\mathbb{R}^{N}\backslash \overline{B_{r_{0}}}$ satisfies $%
\lim_{\left\vert x\right\vert \rightarrow \infty }\left\vert x\right\vert
^{\gamma }u=a^{\ast }.$ The global solutions in $\mathbb{R}^{N}\backslash
\left\{ 0\right\} $ are radial, given by

$\bullet $ $u^{\ast }=a^{\ast }\left\vert x\right\vert ^{-\gamma }.$

$\bullet $ for any $k_{1}>0$,there exist a unique global solution such that 
\begin{equation*}
\lim_{r\rightarrow 0}\left\vert x\right\vert ^{N-p}u=k_{1}>0,\qquad
\lim_{r\rightarrow \infty }\left\vert x\right\vert ^{\gamma }u=a^{\ast }.
\end{equation*}

\noindent (ii) Suppose $\theta +p>0$ and $q\geq \mathbf{q}_{c}$. Any
solution $u\not\equiv 0$ in $B_{r_{0}}\backslash \left\{ 0\right\} $ extends
a solution in $B_{r_{0}}.$ Any positive solution in $\mathbb{R}%
^{N}\backslash \overline{B_{r_{0}}}$ satisfies $\lim_{\left\vert
x\right\vert \rightarrow \infty }\left\vert x\right\vert ^{N-p}u=k_{1}>0.$
There is no positive solution in $\mathbb{R}^{N}\backslash \left\{ 0\right\}
.\medskip $

(iii\noindent ) Suppose $\theta +p<0$, hence $\gamma <0$. Then any solution $%
u\not\equiv 0$ in $B_{r_{0}}\backslash \left\{ 0\right\} $ satisfies $%
\lim_{\left\vert x\right\vert \rightarrow 0}\left\vert x\right\vert ^{\gamma
}u=a^{\ast }.$ Any solution in $\mathbb{R}^{N}\backslash \overline{B_{r_{0}}}
$ satisfies either $\lim_{\left\vert x\right\vert \rightarrow \infty
}\left\vert x\right\vert ^{\gamma }u=a^{\ast }$, or $\lim_{\left\vert
x\right\vert \rightarrow \infty }u=k_{2}>0$, or $\lim_{\left\vert
x\right\vert \rightarrow \infty }\left\vert x\right\vert ^{N-p}u=k_{1}>0.$
The global solutions are radial:

$\bullet $ $u^{\ast }=a^{\ast }\left\vert x\right\vert ^{-\gamma }.$

$\bullet $ For any $k_{1}>0$,there exist a unique global (increasing)
solution such that 
\begin{equation*}
\lim_{x\longrightarrow 0}\left\vert x\right\vert ^{\gamma }u=a^{\ast
},\qquad \lim_{\left\vert x\right\vert \rightarrow \infty }u=k_{1}>0.
\end{equation*}%
(iv) Suppose $\theta +p=0.$ Then there is no global solution $u\not\equiv 0$%
. Any solution in $B_{r_{0}}\backslash \left\{ 0\right\} $ satisfies%
\begin{equation*}
\lim_{x\longrightarrow 0}\left\vert \ln \left\vert x\right\vert \right\vert
^{-\frac{p-1}{q+1-p}}u=\delta _{N,p,q}>0.
\end{equation*}%
Any solution $u\not\equiv 0$ in $\mathbb{R}^{N}\backslash \overline{B_{r_{0}}%
}$ satisfies 
\begin{equation*}
\lim_{\left\vert x\right\vert \rightarrow \infty }\left\vert x\right\vert
^{N-p}u=k_{1}>0,
\end{equation*}%
and there exist such solutions for any $k_{1}>0$.\medskip
\end{corollary}

\begin{remark}
\label{expli}Moreover, when $\theta =-\frac{p(N-1)}{p-1}$, thus $\theta +p<0$%
, there exist \textbf{explicit} of equation (\ref{ep}), given by 
\begin{equation*}
u(r)=(c+\epsilon d_{p,q,N}r^{\frac{p-N}{p-1}})^{-\frac{p}{q+1-p}},\qquad
\epsilon =\pm 1,\quad c\in \mathbb{R}\backslash \left\{ 0\right\} ,
\end{equation*}%
where $d_{p,q,N}=\frac{(p-1)(q-p+1)}{p(N-p)}(\frac{p}{(p-1)(q+1)})^{\frac{1}{%
p}},$ global if $c>0$ and $\epsilon =1,$ local near $0$ if $c<0$ and $%
\epsilon =1,$ or near $\infty $ if $c>0$ and $\epsilon =-1,$ see also \cite[%
Theorem 4.17]{BiGa}.
\end{remark}

\section{First specific radial cases}

Here we only indicate the method of the proofs, since the results will be
covered by the description of the general case.

\subsection{The radial case $p>1,\protect\mu =0,\protect\theta \in \mathbb{R}
$}

In the case $\mu =0$, equation (\ref{pq}) reduces to the equation (\ref{ep}%
). The radial solutions of Henon type equations with the two signs 
\begin{equation}
-\Delta _{p}u-\varepsilon \left\vert x\right\vert ^{\theta }u^{q}=0,\qquad
\varepsilon =\pm 1  \label{epsi}
\end{equation}%
have been exhaustively studied for $q>p-1$ in \cite{BiGa}, for any value of $%
N,p,\theta $, by a phase-plane technique. Here we are concerned by the case $%
p<N$, see \cite[Theorems 4.9,4.14,4.15]{BiGa} for noncritical cases, and 
\cite[Theorems 4.10,4.18,4.19]{BiGa} for critical ones.

\subsection{The radial case $p=2,\protect\mu ,\protect\theta \in \mathbb{R}$}

When $p=2<N$, equation (\ref{pq}) reduces to equation (\ref{casp2}), studied
in \cite{Ci},\cite{WeDu},\cite{CiFa}. In the radial case, the study can also
be reduced completely to the case $\mu =0$, by an elementary proof:

\begin{lemma}
\label{prod}Let $p=2<N$ and $\mu \geq \mu _{0}=-\left( \frac{N-2}{2}\right)
^{2}.$ Let 
\begin{equation*}
u=r^{-S_{i}}w,\quad \text{with }i=1\text{ or }2,
\end{equation*}%
where $S_{1},S_{2}$ are the roots of (\ref{fau}), given by 
\begin{eqnarray*}
S_{1} &=&\frac{N-2+\sqrt{(N-2)^{2}+4\mu }}{2}=\frac{N-2}{2}+\sqrt{\mu -\mu
_{0}}, \\
S_{2} &=&\frac{N-2-\sqrt{(N-2)^{2}+4\mu }}{2}=\frac{N-2}{2}-\sqrt{\mu -\mu
_{0}}.
\end{eqnarray*}%
Then $u$ is a radial solution of equation (\ref{casp2}) if and only if $w$
satisfies a Henon type equation in another dimension (possibly not an
integer) with a Henon-factor $\left\vert x\right\vert ^{\sigma }$, where $%
\sigma $ depends on $\theta $,$\mu $ and also $q$: 
\begin{equation}
-\Delta ^{(\mathbf{N)}}w+r^{\sigma }w^{q}=0,\qquad \mathbf{N}%
=N-2S_{i},\qquad \sigma =\theta -S_{i}(q-1).  \label{newN}
\end{equation}%
In particular, for $i=2,$ then $\mathbf{N}>2$.
\end{lemma}

\begin{proof}
By definition, $S^{2}-(N-2)S-\mu =0.$ Setting $u=r^{-S}w$, we find 
\begin{equation*}
u^{\prime }=r^{-S}w^{\prime }-Sr^{-S-1}w,\qquad u^{\prime \prime
}=r^{-S}w^{\prime \prime }-2Sr^{-S-1}w^{\prime }+S(S+1)r^{-S-2}w,
\end{equation*}%
then $u$ satisfies equation (\ref{casp2}) if and only if 
\begin{align*}
0& =-r^{-S}w^{\prime \prime }+2Sr^{-S-1}w^{\prime
}+S(S+1)r^{-S-2}w-(N-1)(r^{-S-1}w^{\prime }-Sr^{-S-2}w) \\
& +\mu r^{-S-2}w-r^{-Sq+\theta }w^{q} \\
& =r^{-S}(-w^{\prime \prime }-(N-2S-1)\frac{w^{\prime }}{r}+\frac{w}{r^{2}}%
(\mu -S^{2}+(N-2)S)-r^{\theta -S(q-1)}w^{q}) \\
& =r^{-S}(-w^{\prime \prime }-(N-2S-1)\frac{w^{\prime }}{r}-r^{\theta
-S(q-1)}w^{q})=r^{-S}(-\Delta ^{(\mathbf{N)}}w-r^{\sigma }w^{q}),
\end{align*}%
hence the conclusion. If $S=S_{2}$, then $\mathbf{N}-2=N-2S_{2}-2=\frac{1}{2}%
\sqrt{(N-2)^{2}+4\mu }>0.\medskip $
\end{proof}

As a consequence, all the radial study of equation (\ref{casp2}) can be
obtained by making for example the change $u=r^{-S_{2}}w$, and then applying
the results of \cite[Theorems 4.9,4.14;4.15]{BiGa} and \cite[Theorems
4.10,4.18,4.19]{BiGa} with $p=2.$ We find again the study of \cite[Chapter 7]%
{Ci} concerning the radial solutions of equation (\ref{casp2}).

\begin{remark}
The introduction of equation \ref{newN} for the radial solutions is new,
even if the change of unknown, first introduced in \cite{BiGr} for $q<1$,
was used in \cite{Ci} and \cite{CiFa}. It gives an interesting explanation
of the link between the case $\mu =0$ and the general case. For example: by
the change $u=r^{-S_{2}}w$, the radial solutions $u$ in $B_{r_{0}}\backslash
\left\{ 0\right\} $ such that the solutions $\lim_{r\longrightarrow
0}r^{S_{1}}u=k_{1}>0$ correspond to solutions $w$ such that $%
\lim_{r\longrightarrow 0}r^{\mathbf{N}-2}w=k_{1}$, and the solutions $u$
such that $\lim_{r\longrightarrow 0}r^{S_{2}}u=k_{2}>0$ correspond to the
solutions solutions $w$ such that $\lim_{r\longrightarrow 0}w=k_{2}>0$ and
then $w$ extend to a solution in $B_{r_{0}}.\medskip $
\end{remark}

\begin{remark}
\label{expli2}From this change of unknown, and Remark \ref{expli} we also
deduce explicit solutions of equation (\ref{casp2}):

$u(r)=r^{-S_{i}}(c\pm dr^{2-\mathbf{N}})^{-\frac{2}{q-1}}=r^{-S_{i}}(c\pm
dr^{2-N+2S_{i}})^{-\frac{2}{q-1}}$ when $\sigma =-2(\mathbf{N}-1)$, that
means%
\begin{equation*}
u_{1}(r)=r^{-S_{1}}(c\pm dr^{2\sqrt{\mu -\mu _{0}}})^{-\frac{2}{q-1}}\text{%
when }\theta =S_{1}(q-1)-2(N-1-2S_{1}),
\end{equation*}%
\begin{equation*}
u_{2}(r)=r^{-S_{2}}(c\pm dr^{-2\sqrt{\mu -\mu _{0}}})^{-\frac{2}{q-1}}\text{
when }\theta =S_{2}(q-1)-2(N-1-2S_{2}),
\end{equation*}%
where $c>0$ and $d=d_{N,q}>0.$ In particular we find again the global
solutions given in \cite[p.470]{CiFa}.
\end{remark}

\section{General radial case $p>1,\protect\mu ,\protect\theta \in \mathbb{R}$
\label{mainrad}}

Here we make a full description of the radial solutions of equation (\ref{pq}%
), where we find again in particular the specific cases above. \textbf{In
case} $p\neq 2$, \textbf{we cannot reduce the study to the case }$\mu =0$,
as we did at Lemma \ref{prod} in case $p=2.\bigskip $

The radial study by phase-plane techniques reducing the problem to an
autonomous system of two equations has the advantage that it gives a
generally \textbf{exhaustive description of the solutions}, and in a very
precise way. And it is the keypoint for obtaining the behaviours of all the
possibly nonradial solutions, see Theorem \ref{clef}. We consider the radial
form of equation (\ref{pq}):

\begin{equation}
\mathcal{L}_{p,\mu }^{rad}u+r^{\theta }u^{q}=0,\text{ where }\mathcal{L}%
_{p,\mu }^{rad}u=-\frac{d}{dr}(\left\vert u^{\prime }\right\vert
^{p-2}u^{\prime })-\frac{N-1}{r}\left\vert u^{\prime }\right\vert
^{p-2}u^{\prime }+\mu \frac{u^{p-1}}{r^{p}}.  \label{onep}
\end{equation}

\subsection{A common formulation of the problem}

For studying the l solutions of (\ref{onep}), and more generally of 
\begin{equation}
\mathcal{L}_{p,\mu }^{rad}u=\varepsilon r^{\theta }u^{q},\qquad \varepsilon
=\pm 1,  \label{gepsi}
\end{equation}%
the first attempt is to make a classical transformation, introduced in \cite%
{Bi} for the equation with a source term $-\Delta _{p}u=u^{q}$, 
\begin{equation}
X=r^{\gamma }u=e^{\gamma t}U,\qquad Y=-r^{(\gamma +1)(p-1)}\left\vert
u^{\prime }\right\vert ^{p-2}u^{\prime }=-e^{\gamma (p-1)t}\left\vert
U_{t}\right\vert ^{p-2}U_{t},\qquad t=\ln r,  \label{X}
\end{equation}%
where $\gamma $ has been defined at (\ref{gam}). For equation (\ref{gepsi}),
it leads to the autonomous system 
\begin{equation}
\left\{ 
\begin{array}{ccc}
X_{t} & = & \gamma X-\left\vert Y\right\vert ^{\frac{2-p}{p-1}}Y, \\ 
Y_{t} & = & -(N-p-\gamma (p-1))Y-\varepsilon X^{q}-\mu X^{p-1}.%
\end{array}%
\right.  \label{SXY}
\end{equation}%
A system of this type was still used in \cite{AbFePe} for equation (\ref%
{gepsi}) with $\varepsilon =1,\theta =0$, when $q=\frac{N(p-1)+p}{N-p}$ is
the critical exponent (which was the object of many works). In particular
system (\ref{SXY}) admits a Pohozaev type energy function, above all used
for $\varepsilon =1:$

\begin{lemma}
\label{ener}Consider the system (\ref{SXY}). Let 
\begin{equation*}
D=N-p-p\gamma =\frac{(N-p)(q-\mathbf{q}_{s})}{q+1-p},\text{ with }\mathbf{q}%
_{s}=\frac{N(p-1)+p+p\theta }{N-p},
\end{equation*}%
\begin{equation}
\mathcal{E=}\frac{p-1}{p}\left\vert Y\right\vert ^{\frac{p}{p-1}}-\gamma
XY+(D\left\vert \gamma \right\vert ^{p-2}\gamma -\mu )\frac{X^{p}}{p}%
+\varepsilon \frac{X^{q+1}}{q+1}.  \label{fv}
\end{equation}%
Then 
\begin{equation}
\mathcal{E}_{t}=-D(\gamma X-\left\vert Y\right\vert ^{\frac{2-p}{p-1}%
}Y)(\left\vert \gamma X\right\vert ^{p-2}(\gamma X)-Y)  \label{deri}
\end{equation}%
has the sign of $-D$, that means $\mathcal{E}$ is increasing $q<q_{s}$,
decreasing for $q>\mathbf{q}_{s}$ and constant for $q=\mathbf{q}_{s}.$
\end{lemma}

\begin{proof}
By simple computation: we set $k=N-p-\gamma (p-1)$ and $D=k-\gamma
=N-p-p\gamma ;$ then 
\begin{align*}
\mathcal{E}_{t}& =\left\vert Y\right\vert ^{\frac{1}{p-1}}Y_{t}-\gamma
XY_{t}-\gamma X_{t}Y+(D\left\vert \gamma \right\vert ^{p-2}\gamma -\mu
)X^{p-1}X_{t}+\varepsilon X^{q}X_{t} \\
& =(\left\vert Y\right\vert ^{\frac{2-p}{p-1}}Y-\gamma X)(-kY+\varepsilon
X^{q}-\mu X^{p-1}) \\
& +(\gamma X-\left\vert Y\right\vert ^{\frac{2-p}{p-1}}Y)(\varepsilon
X^{q}-\gamma Y+(D\left\vert \gamma \right\vert ^{p-2}\gamma -\mu )X^{p-1}) \\
& =(\gamma X-\left\vert Y\right\vert ^{\frac{2-p}{p-1}}Y)(kY-\varepsilon
X^{q}+\mu X^{p-1}+\varepsilon X^{q}-\gamma Y+(D\left\vert \gamma \right\vert
^{p-2}\gamma -\mu )X^{p-1}) \\
& =(\gamma X-\left\vert Y\right\vert ^{\frac{2-p}{p-1}}Y)(A\left\vert \gamma
\right\vert ^{p-2}\gamma X^{p-1}+DY).
\end{align*}
\end{proof}

\subsection{New formulation as an autonomous system}

The system (\ref{SXY}) has the drawback to be singular at $X=0$ or $Y=0$,
according to the value of $p.$ But the main lack is that it has at most two
fixed points; and the study at $(0,0)$ is difficult, and does not show
easily the multiple possibilities of behaviour of the solutions. In \cite%
{BiGi} we have shown that an introduction of the function slope $\frac{%
ru^{\prime }}{u}$in such kind of problems allows to give much more
informations on the solutions, that is what we do in the sequel.
Nevertheless, the change of unknown that we had introduced in the case $\mu
=0$ of equation (\ref{epsi}) with $\varepsilon =-1$, defined by 
\begin{equation*}
S=-\frac{ru_{r}}{u},\qquad Z=r^{\theta +1}\frac{u^{q}}{\left\vert
u_{r}\right\vert ^{p-2}u_{r}},\qquad t=\ln r,
\end{equation*}%
leading to the system 
\begin{equation*}
\left\{ 
\begin{array}{ccc}
S_{t} & = & S(S-\frac{N-p}{p-1}+\frac{Z}{p-1}), \\ 
Z_{t} & = & Z(N+\theta -qS-Z),%
\end{array}%
\right.
\end{equation*}%
cannot be adapted to the case $\mu \neq 0.$ That is why we introduce another
form allowing to treat equation (\ref{pq}) \textbf{for any } $\mu \in 
\mathbb{R}$.

\begin{lemma}
Let $u$ be a radial positive solution of (\ref{pq}). For any $r>0$, let 
\begin{equation}
G(t)=\left\vert S\right\vert ^{p-2}S(r),\text{ where }S=-\frac{ru^{\prime }}{%
u},\qquad V=r^{\theta +p}u^{q+1-p},\qquad t=\ln r.  \label{gv}
\end{equation}%
Then $(G,V)$ satisfies the system 
\begin{equation}
\left\{ 
\begin{array}{ccc}
G_{t} & = & (p-1)\left\vert G\right\vert ^{\frac{p}{p-1}}-(N-p)G-\mu -V, \\ 
V_{t} & = & (q+1-p)V(\gamma -\left\vert G\right\vert ^{\frac{2-p}{p-1}}G).%
\end{array}%
\right.  \label{SGV}
\end{equation}%
\bigskip
\end{lemma}

\begin{proof}
We set $t=\ln r$ and $u(r)=U(t)$, then $S=-\frac{U_{t}}{U},$ 
\begin{equation*}
u^{\prime }(r)=\frac{U_{t}}{r},\quad \left\vert u^{\prime }\right\vert
^{p-2}u^{\prime }=\frac{\left\vert U_{t}\right\vert ^{p-2}U_{t}}{r^{p-1}}%
,\quad \frac{d}{dr}(\left\vert u^{\prime }\right\vert ^{p-2}u^{\prime })=%
\frac{(\left\vert U_{t}\right\vert ^{p-2}U_{t})_{t}-(p-1)\left\vert
U_{t}\right\vert ^{p-2}U_{t}}{r^{p}}
\end{equation*}%
\begin{equation*}
-(\left\vert U_{t}\right\vert ^{p-2}U_{t})_{t}+(p-N)\left\vert
U_{t}\right\vert ^{p-2}U_{t}+\mu U^{p-1}+r^{\theta +p}U^{q}=0.
\end{equation*}%
Since $U$ is positive, dividing by $U^{p-1}$ we get 
\begin{equation*}
\frac{-(\left\vert U_{t}\right\vert ^{p-2}U_{t})_{t}}{U^{p-1}}+(N-p)\frac{%
\left\vert U_{t}\right\vert ^{p-2}U_{t}}{U^{p-1}}+\mu +e^{(\theta +p)t}U^{q%
\mathbf{-}p+1}=0.
\end{equation*}%
Defining $G$ and $V$ by (\ref{gv}) we obtain the system (\ref{SGV}).\bigskip
\end{proof}

\begin{remark}
Note the relations between systems (\ref{SXY}) and (\ref{SGV}): there holds 
\begin{equation}
V=X^{q+1-p},\qquad Y=GX^{p-1},  \label{xygv}
\end{equation}%
and the energy function is expressed in terms of $G,V$ by 
\begin{equation}
\mathcal{E}=V^{\frac{p}{q+1-p}}(\frac{p-1}{p}\left\vert G\right\vert ^{\frac{%
p}{p-1}}-\gamma G+\frac{D\left\vert \gamma \right\vert ^{p-2}\gamma -\mu }{p}%
-\frac{V}{q+1})=V^{\frac{p}{q+1-p}}(\frac{F(G)+D\left\vert \gamma
\right\vert ^{p-2}\gamma }{p}-\frac{V}{q+1}).  \label{derigv}
\end{equation}
\end{remark}

\begin{remark}
\label{rema}Note that the relation (\ref{eqS}) giving the possible roots $%
S_{1}$,$S_{2}$ can be written in terms of $G$ by 
\begin{equation*}
\varphi \mathcal{(S)}=0\Longleftrightarrow (p-1)\left\vert G\right\vert ^{%
\frac{p}{p-1}}-(N-p)G-\mu =0,
\end{equation*}%
and the difference $(p-1)\left\vert G\right\vert ^{\frac{p}{p-1}}-(N-p)G-\mu 
$ clearly appears in system (\ref{SGV}). The choice of the variable $G$
compared to $S$ is motivated by the fact that it leads to a simpler system:
indeed the system relative to $S,V$ is 
\begin{equation*}
\left\{ 
\begin{array}{ccc}
S_{t} & = & S^{2}-\frac{(N-p)}{p-1}S-\frac{1}{p-1}\left\vert S\right\vert
^{2-p}(\mu +V), \\ 
V_{t} & = & (q+1-p)V(\gamma -S),%
\end{array}%
\right.
\end{equation*}%
where the term $\left\vert S\right\vert ^{2-p}$ does not make the study
easy.\bigskip
\end{remark}

Next we apply the logarithmic transformation to the radial solutions $u$ of
equation (\ref{Hardy}). We consider in particular the case $\mu =\mu _{0},$
and find again a result of \cite[Theorem 2.1]{ItaTa} with a simple proof :

\begin{lemma}
\label{critic} Let $p>1$ and $\mu =\mu _{0}.$ Then for any $\ell >0$ there
exists a radial solution $u$ of equation (\ref{Hardy}) defined near $r=0$
(resp. near $\infty )$ such that 
\begin{equation}
\lim_{r\longrightarrow 0}r^{\frac{N-p}{p}}\left\vert \ln r\right\vert ^{-%
\frac{2}{p}}u(r)=\ell \qquad (\text{resp. }\lim_{r\longrightarrow \infty }r^{%
\frac{N-p}{p}}\left\vert \ln r\right\vert ^{-\frac{2}{p}}u(r)=\ell ).
\label{the}
\end{equation}
\end{lemma}

\begin{proof}
Let again $G(t)=\left\vert S\right\vert ^{p-2}S(r),$ where $S=-\frac{%
ru^{\prime }}{u}$ and $t=\ln r.$ We get 
\begin{equation}
\mathcal{L}_{p,\mu }u=0\Longleftrightarrow G_{t}=F(G)=(p-1)\left\vert
G\right\vert ^{\frac{p}{p-1}}-(N-p)G-\mu  \label{equi}
\end{equation}%
$\bullet $ In the case $p=2$ then $G=S$ and $S_{t}=(S-\frac{N-2}{2})^{2},$
and by integration we obtain explicitely all the solutions of any sign of
the equation by $u(r)=r^{-\frac{N-2}{2}}(C_{1}+C_{2}\ln r),$ depending on
two parameters $C_{1},C_{2}\in \mathbb{R}.$

\noindent $\bullet $ In the general case $p>1$, the equation still admits
the solution $G\equiv G_{0},$ that means $-\frac{U_{t}}{U}=\frac{N-p}{p},$
which corresponds to the solutions $u(r)=Cr^{-\frac{N-p}{p}.}.$ The other
solutions are obtained by quadrature. Setting $G_{0}=(\frac{N-p}{p})^{p-1}$
and $G=G_{0}+\overline{G},$ there holds 
\begin{equation*}
F(G_{0}+\overline{G})=\frac{1}{2}F^{\prime \prime }(G_{0}+\tau \overline{G})%
\overline{G}^{2}=\frac{p}{2(p-1)}(G_{1}+\tau \overline{G})^{\frac{2-p}{p-1}}%
\overline{G}^{2}
\end{equation*}%
for some $\tau \in \left( 0,1\right) ,$ hence $F(G)\sim _{G\longrightarrow
G_{0}}c\overline{G}^{2}$ with $c=\frac{p}{2}(\frac{N-p}{p})^{2-p};$ so we
obtain a family with one parameter of solutions such that 
\begin{equation*}
t=C+\int_{G_{0}}^{G}\frac{dg}{F(g)}.
\end{equation*}%
defined on $(\infty ,C),$ or in $(C,\infty )$ and then $\overline{G}\sim
_{t\longrightarrow \pm \infty }\frac{1}{ct};$ then $G_{t}=\frac{p}{2(p-1)}%
G_{0}^{\frac{2-p}{p-1}}(\overline{G}^{2}+O(\overline{G}^{3})=\frac{p}{2(p-1)}%
G_{0}^{\frac{2-p}{p-1}}\overline{G}^{2}(1+O(\frac{1}{\left\vert t\right\vert 
});$ by integration $\overline{G}(t)=-\frac{1}{ct}+O(\frac{\left\vert \ln
\left\vert t\right\vert \right\vert }{\left\vert t\right\vert ^{2}})$; then 
\begin{equation*}
S=-\frac{U_{t}}{U}=\frac{N-p}{p}-\frac{(\frac{N-p}{p})^{2-p}}{(p-1)c}\frac{1%
}{t}+O(\frac{\ln \left\vert t\right\vert }{\left\vert t\right\vert ^{2}})=%
\frac{N-p}{p}-\frac{2}{p}\frac{1}{t}+O(\frac{\left\vert \ln \left\vert
t\right\vert \right\vert }{\left\vert t\right\vert ^{2}}),
\end{equation*}%
thus $\ln (Ue^{\frac{N-p}{p}t}\left\vert t\right\vert ^{-\frac{2}{p}})=O(%
\frac{\ln \left\vert t\right\vert }{\left\vert t\right\vert ^{2}})$, and $%
\frac{\ln \left\vert t\right\vert }{\left\vert t\right\vert ^{2}}$ is
integrable, then we obtain, $\lim_{t\longrightarrow \pm \infty }Ue^{\frac{N-p%
}{p}t}\left\vert t\right\vert ^{-\frac{2}{p}}=\ell _{\pm }\neq 0$. By a
scaling $u(r)=v(ar)$ letting the equation invariant, we obtain the existence
for any $\ell >0.$
\end{proof}

\subsection{Main radial results}

The proofs of the following theorems are postponed at Section \ref{App}.

\begin{theorem}
\label{H1rad}\textbf{Case }($\mathcal{H}_{1}$): Let $\mu \geq \mu _{0}$ and $%
\gamma >S_{1}.$ Then

$\bullet $ There exists a global particular solution $u^{\ast }=a^{\ast
}r^{-\gamma }$ of (\ref{onep}).

$\bullet $ If $\mu >\mu _{0}$, for any $k_{1}>0$, there exist a unique
global solution of (\ref{onep}) such that 
\begin{equation}
\lim_{r\rightarrow 0}r^{S_{1}}u=k_{1},\qquad \lim_{r\rightarrow \infty
}r^{\gamma }u=a^{\ast }.  \label{aa}
\end{equation}

$\bullet $ If $\mu =\mu _{0}$, for any $\ell >0$ there exists a unique
global solution such that 
\begin{equation}
\lim_{r\rightarrow 0}r^{\frac{N-p}{p}}(\left\vert \ln (r)\right\vert ^{-%
\frac{2}{p}})u(r)=\ell ,\qquad \lim_{r\rightarrow \infty }r^{\gamma
}u=a^{\ast },\qquad \text{if }\mu =\mu _{0}.  \label{aaa}
\end{equation}%
All the other solutions are local.

$\bullet $ There exist solutions in a finite interval $(0,R)$, such that
respectively 
\begin{equation}
\lim_{r\rightarrow 0}r^{\gamma }u=a^{\ast },\qquad \text{ }%
\lim_{r\rightarrow R}u=\infty ,  \label{bb}
\end{equation}%
\begin{equation}
\lim_{r\rightarrow 0}r^{\gamma }u=a^{\ast },\text{ \qquad\ }u(R)=0,
\label{cc}
\end{equation}%
\begin{equation}
\left\{ 
\begin{array}{c}
\lim_{r\rightarrow 0}r^{S_{1}}u=k_{1}>0\text{ if }\mu >\mu _{0} \\ 
\lim_{r\rightarrow 0}r^{\frac{N-p}{p}}(\left\vert \ln (r)\right\vert ^{-%
\frac{2}{p}})u(r)=\ell >0,\text{ if }\mu =\mu _{0}%
\end{array}%
\right. ,\quad \quad \lim_{r\rightarrow R}u=\infty ,  \label{dd}
\end{equation}%
\begin{equation}
\left\{ 
\begin{array}{c}
\lim_{r\rightarrow 0}r^{S_{1}}u=k_{1}>0\text{ if }\mu >\mu _{0} \\ 
\lim_{r\rightarrow 0}r^{\frac{N-p}{p}}(\left\vert \ln (r)\right\vert ^{-%
\frac{2}{p}})u(r)=\ell >0,\text{ if }\mu =\mu _{0}%
\end{array}%
\right. \text{\qquad }\lim_{r\rightarrow R}u=0,  \label{ee}
\end{equation}%
\begin{equation}
\lim_{r\rightarrow 0}r^{S_{2}}u=k_{2}>0,\quad \text{\quad\ }%
\lim_{r\rightarrow R}u=\infty .  \label{ff}
\end{equation}

$\bullet $ There exist solutions in an interval $(R,\infty )$ such that 
\begin{equation}
\lim_{r\rightarrow R}u=\infty ,\qquad \text{ }\lim_{r\rightarrow \infty
}r^{\gamma }u=a^{\ast }.  \label{gg}
\end{equation}

Any radial solution, local near $0$ or $\infty $ or global, has one of these
types.
\end{theorem}

\begin{theorem}
\label{H2rad}\textbf{Case }($\mathcal{H}_{2}$): Let $\mu \geq \mu _{0}$ and $%
\gamma <S_{2}.$ Then

$\bullet $ There exists a global particular solution $u^{\ast }=a^{\ast
}r^{-\gamma }$ of (\ref{onep}).

$\bullet $ If $\mu >\mu _{0}$, for any $k_{2}>0$, there exist a unique
global solution of (\ref{onep})such that 
\begin{equation}
\lim_{r\rightarrow 0}r^{\gamma }u=a^{\ast },\qquad \lim_{r\rightarrow \infty
}r^{S_{2}}u=k_{2}.  \label{AA}
\end{equation}

$\bullet $ If $\mu =\mu _{0}$, for any $\ell >0$ there exists a unique
global solution such that 
\begin{equation}
\lim_{r\rightarrow 0}r^{\gamma }u=a^{\ast },\qquad \lim_{r\rightarrow \infty
}r^{\frac{N-p}{p}}\left\vert \ln (r)\right\vert ^{\frac{2}{p}}u(r)=\ell .
\label{AAA}
\end{equation}%
All the other solutions are local, each of them has one of the following
types.

$\bullet $ There exists solutions in a finite interval $(R,\infty )$, such
that respectively 
\begin{equation}
\text{ }\lim_{r\rightarrow R}u=\infty ,\qquad \lim_{r\rightarrow \infty
}r^{\gamma }u=a^{\ast },  \label{BB}
\end{equation}%
\begin{equation}
\lim_{r\rightarrow R}u=0,\qquad \lim_{r\rightarrow \infty }r^{\gamma
}u=a^{\ast },  \label{CC}
\end{equation}%
\begin{equation}
\text{ }\lim_{r\rightarrow R}u=\infty ,\qquad \left\{ 
\begin{array}{c}
\lim_{r\rightarrow \infty }r^{S_{2}}u=k_{2}>0\text{ if }\mu >\mu _{0} \\ 
\lim_{r\rightarrow \infty }r^{\frac{N-p}{p}}(\left\vert \ln (r)\right\vert
-^{\frac{2}{p}})u(r)=\ell >0,\text{ if }\mu =\mu _{0}%
\end{array}%
\right. ,  \label{DD}
\end{equation}%
\begin{equation}
\lim_{r\rightarrow R}u=0,\qquad \left\{ 
\begin{array}{c}
\lim_{r\rightarrow \infty }r^{S_{2}}u=k_{2}>0\text{ if }\mu >\mu _{0} \\ 
\lim_{r\rightarrow \infty }r^{\frac{N-p}{p}}(\left\vert \ln (r)\right\vert
^{-\frac{2}{p}})u(r)=\ell >0,\text{ if }\mu =\mu _{0}%
\end{array}%
\right. ,  \label{EE}
\end{equation}%
\begin{equation}
\lim_{r\rightarrow R}u=\infty ,\qquad \lim_{r\rightarrow
0}r^{S_{1}}u=k_{1}>0,  \label{FF}
\end{equation}

$\bullet $ There exist solutions in an interval $(0,R)$ such that 
\begin{equation}
\lim_{r\rightarrow 0}r^{\gamma }u=a^{\ast },\qquad \lim_{r\rightarrow
R}u=\infty .  \label{GG}
\end{equation}

\noindent Moreover, if $p>2$ and $\gamma =0$, the solutions satisfying (\ref%
{AA}),(\ref{AAA}),(\ref{GG}) are constant near $0,$ the solutions satisfying
(\ref{BB}),(\ref{CC}) are constant near $\infty $.\bigskip
\end{theorem}

\begin{theorem}
\label{H3rad}\textbf{\ Case }($\mathcal{H}_{3}$) Let $\mu >\mu _{0}$ and $%
S_{2}\leq \gamma \leq S_{1}.$ Then there is no global solution of (\ref{onep}%
). There exists solutions respectively on $(0,R)$ such that 
\begin{align}
\lim_{r\longrightarrow 0}r^{S_{2}}u& =k_{2}>0,\qquad \lim_{r\rightarrow
R}u=\infty ,\qquad \text{if }S_{2}<\gamma <S_{1},  \label{nn} \\
\lim_{r\longrightarrow 0}r^{S_{2}}(\left\vert \ln r\right\vert )^{\frac{1}{%
q+1-p}}u& =\alpha _{N,p,q},\qquad \lim_{r\rightarrow R}u=\infty ,\qquad 
\text{if }\gamma =S_{2}\neq 0,  \label{nnn} \\
\lim_{r\longrightarrow 0}\left\vert \ln r\right\vert ^{-\frac{p-1}{q+1-p}}u&
=\delta _{N,p,q},\qquad \lim_{r\rightarrow R}u=\infty ,\qquad \text{if }%
\gamma =S_{2}=0,  \label{nno}
\end{align}%
with $\alpha _{N,p,q}=(\frac{(p-1)(N-p-pS_{2})S_{2}^{p-2}}{q+1-p})^{\frac{1}{%
q+1-p}}$, $\delta _{N,p,q}=(\frac{(q+1-p)(N-p)^{-\frac{1}{p-1}}}{p-1})^{%
\frac{p-1}{q+1-p}}$; and on $(R,\infty )$ such that 
\begin{align}
\lim_{r\rightarrow R}u& =\infty ,\qquad \lim_{r\longrightarrow \infty
}r^{S_{1}}u=k_{2}>0,\qquad \text{if }S_{1}<\gamma <S_{2},  \label{rr} \\
\lim_{r\rightarrow R}u& =\infty ,\qquad \lim_{r\longrightarrow \infty
}r^{S_{1}}(\left\vert \ln r\right\vert )^{\frac{1}{q+1-p}}u=\beta
_{N,p,q},\qquad \text{if }\gamma =S_{1},  \label{rrr}
\end{align}%
with $\beta _{N,p,q}=(\frac{(p-1)(pS_{1}-N+p)S_{2}^{p-2}}{q+1-p})^{\frac{1}{%
q+1-p}}.$ Any local solution near $0$ or $\infty $ has one of these types;
and it is unique when $S_{2}<\gamma <S_{1}.$
\end{theorem}

\begin{theorem}
\label{H4rad}\textbf{Case} ($\mathcal{H}_{4}$) Suppose $\mu =\mu _{0}$ and $%
\gamma =\frac{N-p}{p}.$ Then there is no global solution of (\ref{onep}).
For any $R>0$,there exists solutions such that 
\begin{equation}
\lim_{r\longrightarrow R}u=\infty ,\qquad \lim_{r\longrightarrow \infty }r^{%
\frac{N-p}{p}}(\ln r)^{\frac{2}{q+1-p}}u=c_{N,p,q},\qquad R>0,  \label{ao}
\end{equation}%
\begin{equation}
\lim_{r\longrightarrow 0}r^{\frac{N-p}{p}}\left\vert \ln r\right\vert ^{%
\frac{2}{q+1-p}}u=c_{N,p,q},\qquad \lim_{r\longrightarrow R}u=\infty ,\qquad
R>0,  \label{bo}
\end{equation}%
with $c_{N,p,q}=(\frac{2(q+1)}{(q+1-p)^{2}}(\frac{N-p}{p})^{p-2})^{\frac{1}{%
q+1-p}}.$ All of the solutions have one of these types, and can be obtained
by quadratures. If $p=2,$ the solutions are \textbf{explicit, }given by%
\textbf{\ } 
\begin{equation}
u(r)=(\frac{2(q+1)}{(q-1)^{2}})^{\frac{1}{q-1}}r^{\frac{2-N}{2}}(\ln r\pm
\ln R)^{-\frac{2}{q-1}},\qquad R>0.  \label{co}
\end{equation}
\end{theorem}

\begin{theorem}
\label{H5rad}\textbf{Case} ($\mathcal{H}_{5}$). Suppose $\mu <\mu _{0}.$
Then there exists a unique global solution: $u^{\ast }=a^{\ast }r^{-\gamma }$
of (\ref{onep}).

Moreover, there exist solutions satisfying any of the behaviours 
\begin{equation}
\lim_{r\rightarrow R}u=\infty ,\text{ \qquad }\lim_{r\rightarrow \infty
}r^{\gamma }u=a^{\ast },  \label{t1}
\end{equation}%
\begin{equation}
\lim_{r\rightarrow R}u=0,\qquad \lim_{r\rightarrow \infty }r^{\gamma
}u=a^{\ast },  \label{t2}
\end{equation}%
\begin{equation}
\lim_{r\rightarrow 0}r^{\gamma }u=a^{\ast },\qquad \lim_{r\rightarrow
R}u=\infty ,  \label{t3}
\end{equation}%
\begin{equation}
\lim_{r\rightarrow 0}r^{\gamma }u=a^{\ast },\text{\qquad\ }u(R)=0.
\label{t4}
\end{equation}

Any local solution near $0$ or $\infty $ has one of these types. Moreover,
if $p>2$ and $\gamma =0$, the solutions are constant near $\infty $, or near 
$0$.
\end{theorem}

\begin{remark}
\label{nonuniq}When $\gamma =0$, Theorem \ref{H2rad} in case $\mu _{0}<\mu
<0 $ (resp. Theorem \ref{H5rad} in case $\mu <\mu _{0}$) shows in evidence a
phenomena of \textbf{nonuniquess} of solutions $u$ in $C^{1}(B_{r_{0}}%
\backslash \left\{ 0\right\} )\cap C(B_{r_{0}})$ of equation (\ref{pqzero}), 
\textbf{valid for any} $p>1$: besides the constant solution $u^{\ast
}(r)\equiv \left\vert \mu \right\vert ^{\frac{1}{q+1-p}}$, it admits
solutions satisfying (\ref{AA}),(\ref{AAA}),(\ref{GG}), (resp. (\ref{t3}),(%
\ref{t4})). Considering again $u=\left\vert \mu \right\vert ^{\frac{1}{q-1}%
}(1+\overline{u})$, this result can be compared to the nonuniqueness result
of \cite[Remark 5.1]{GuedVe} relative to the equation $-\Delta _{p}w+w=0$
for $p>2$.
\end{remark}

\section{Basic arguments for the nonradial case}

\subsection{The strong maximum principle}

\begin{theorem}
\label{SMP}(Strong Maximum Principle) Let $u$ be any nonnegative $C^{1}$%
solution of (\ref{pq}) in a domain $\omega $ such that $\overline{\omega }%
\subset \mathbb{R}^{N}\backslash \left\{ 0\right\} .$ Then either $u$ is
positive in $\omega $, or $u\equiv 0.$ As a consequence, any nonnegative
solution in $B_{r_{0}}\backslash \left\{ 0\right\} $ (resp. in $\mathbb{R}%
^{N}\backslash \overline{B_{r_{0}}}$) is positive, or $u\equiv 0$.
\end{theorem}

\begin{proof}
If $\mu \leq 0$, then setting $C_{1}=\max_{x\in \overline{\omega }%
}\left\vert x\right\vert ^{\theta }$, and $\beta _{1}(u)=C_{1}u^{q}$, 
\begin{equation*}
-\Delta _{p}u+\beta _{1}(u)\geq -\Delta _{p}u+\left\vert x\right\vert
^{\theta }u^{q}=\left\vert \mu \right\vert \frac{u^{p-1}}{\left\vert
x\right\vert ^{p}}\geq 0.
\end{equation*}%
If $\mu >0$ then setting $C_{2}=\max_{x\in \overline{\omega }}\left\vert
x\right\vert ^{-p}$ and $\beta _{2}(u)=\mu C_{2}u^{p-1}+u^{q}$, 
\begin{equation*}
-\Delta _{p}u+\beta _{2}(u)\geq 0.
\end{equation*}%
Then the result comes from the Strong Maximum Principle of \cite[Theorem 5]%
{Va}, since in any case $\int_{0}^{1}(s\beta _{i}(s))^{-\frac{1}{p}%
}ds=\infty $.
\end{proof}

\subsection{The Weak Comparison Principle}

We recall a main argument, due to \cite[Proposition 2.2]{DuGu}, leading to a
Weak Comparison Principle:

\begin{proposition}
\label{DuGu}\cite[Proposition 2.2]{DuGu} Let $\omega $ be a bounded domain
in $\mathbb{R}^{N}$, $A\in C_{b}\left( \omega \right) $ and $B\in
C^{+}\left( \omega \right) $, $B\not\equiv 0.$ Let $u$,$v\in C^{1}\left(
\omega \right) $ be two positive functions such that 
\begin{equation*}
-\Delta _{p}v+Av^{p-1}+Bg(v)\geq 0\geq -\Delta _{p}u+Au^{p-1}+Bg(u)
\end{equation*}%
in $\mathcal{D}^{\prime }(\omega )$, where $g\in C(\left[ 0,\infty \right) ,$
and $s\longmapsto g(s)/s^{p-1}$ is increasing on $(\inf_{\omega
}(u,v),\sup_{\omega }(u,v).$ If $\lim \sup_{d(x,\partial \omega
)\longrightarrow 0}(u-v)(x)\leq 0,$ then $u\leq v$ in $\omega $.
\end{proposition}

We apply this theorem to problem (\ref{pq}), with $A(x)=\mu \left\vert
x\right\vert ^{-p}$ and $g(s)=s^{q}$, when $q>p-1$, with $B(x)=\left\vert
x\right\vert ^{\theta }$ and $\omega $ is a domain such that $\overline{%
\omega }\subset \Omega \backslash \left\{ 0\right\} :$

\begin{corollary}
\label{DG}Let $\omega =B_{r_{2}}\backslash \overline{B_{r_{1}}}$, with $%
0<r_{1}<r_{2}$. Let $q>p-1$,and $\mathcal{L}_{p,\mu }$ be defined by (\ref%
{lpmu}) for any $\mu $,$\theta \in \mathbb{R}.$ Let $u$,$v\in C^{1}(\omega
)\cap C(\overline{\omega })$ such that 
\begin{equation*}
\mathcal{L}_{p,\mu }v+\left\vert x\right\vert ^{\theta }v^{q}\geq 0\geq 
\mathcal{L}_{p,\mu }u+\left\vert x\right\vert ^{\theta }u^{q}
\end{equation*}%
$a.e.$in $\omega $, and $u\leq v$ on $\partial \omega .$ Then $u\leq v$ in $%
\omega $.
\end{corollary}

\subsection{A priori Osserman's estimate near $0$ or $\infty $}

\begin{proposition}
\label{Osserman} Let $q>p-1$, and $\mu ,\theta \in \mathbb{R}$. Then there
exists a constant $C=C_{N,p,q,\theta ,\mu }>0$ such that, for any solution $%
u $ of (\ref{pq}) in $B_{2r_{0}}\backslash \left\{ 0\right\} $, (resp. in $%
\mathbb{R}^{N}\backslash \overline{B_{\frac{r_{0}}{2}}}$, resp. in $\mathbb{R%
}^{N}\backslash \left\{ 0\right\} $) and for any $x\in B_{r_{0}}\backslash
\left\{ 0\right\} $, (resp. $x\in $ $\mathbb{R}^{N}\backslash \overline{%
B_{r_{0}}}$, resp. $x\in \mathbb{R}^{N}\backslash \left\{ 0\right\} $),%
\begin{equation}
u(x)\leq C_{N,p,q,\theta ,\mu }\left\vert x\right\vert ^{-\frac{p+\theta }{%
q+1-p}},\qquad \forall x\in B_{r_{0}}\backslash \left\{ 0\right\} \text{ }(%
\text{resp. }x\in \mathbb{R}^{N}\backslash \overline{B_{r_{0}}},\text{ resp. 
}x\in \mathbb{R}^{N}\backslash \left\{ 0\right\} ).  \label{Oss}
\end{equation}
\end{proposition}

\begin{proof}
The estimate is classical in case $\mu >0.$ In the general case, we write,
since $q>p-1$, 
\begin{eqnarray*}
-\Delta _{p}u+\left\vert x\right\vert ^{\theta }u^{q} &\leq &\left\vert \mu
\right\vert \frac{u^{p-1}}{\left\vert x\right\vert ^{p}}=\left\vert \mu
\right\vert (\left\vert x\right\vert ^{\theta }u^{q})^{\frac{p-1}{q}%
}\left\vert x\right\vert ^{-\frac{pq+\theta (p-1)}{q}} \\
&\leq &\left\vert \mu \right\vert (\varepsilon ^{\frac{q}{p-1}}\left\vert
x\right\vert ^{\theta }u^{q}+\varepsilon ^{-\frac{q}{q-p+1}}(\left\vert
x\right\vert ^{-\frac{pq+\theta (p-1)}{q}})^{\frac{q}{q-p+1}}\leq \frac{%
\left\vert x\right\vert ^{\theta }u^{q}}{2}+c_{p,q,\mu }\left\vert
x\right\vert ^{-\frac{pq+\theta (p-1)}{q-p+1}},
\end{eqnarray*}%
where $c_{p,q,\mu }$ only depends on $p$,$q$,$\mu $ from H\"{o}lder
inequality with $\varepsilon =(\frac{1}{2\left\vert \mu \right\vert })^{%
\frac{p-1}{q}}$, then 
\begin{equation*}
-\Delta _{p}u+\frac{1}{2}\left\vert x\right\vert ^{\theta }u^{q}\leq
C_{p,q,\mu }\left\vert x\right\vert ^{-\frac{pq+\theta (p-1)}{q-p+1}}.
\end{equation*}%
Next we apply the following estimate, see for example \cite{VaVe}: if $\psi $
is a function such that in a ball $B_{R}$ 
\begin{equation*}
-\Delta _{p}\psi +a\psi ^{q}\leq b
\end{equation*}%
for some $a,b>0$, and $q>p-1$, then there exists $C_{N,p,q}>0$ depending on $%
N,p,q$ such that%
\begin{equation}
\psi (0)\leq C_{N,p,q}\left( \frac{1}{aR^{p}}\right) ^{\frac{1}{q+1-p}}+(%
\frac{b}{a})^{\frac{1}{q}}  \label{esti}
\end{equation}%
Applying (\ref{esti}) to the function $\psi (x)=u(x-x_{0})$ for any $%
x_{0}\in B_{\frac{r_{0}}{2}}\backslash \left\{ 0\right\} $, and $\left\vert
x-x_{0}\right\vert <R=\frac{\left\vert x_{0}\right\vert }{2}$, thus $\frac{%
\left\vert x_{0}\right\vert }{2}\leq \left\vert x\right\vert \leq \frac{%
3\left\vert x_{0}\right\vert }{2},$ and 
\begin{equation*}
a=2^{-(\theta +1)}\min \left\{ 1,3^{\theta }\right\} \left\vert
x_{0}\right\vert ^{\theta }\qquad b=\left\vert x_{0}\right\vert ^{\frac{%
pq+\theta (p-1)}{q-p+1}}\max \left\{ 2^{-\frac{pq+\theta (p-1)}{q-p+1}},(%
\frac{3}{2})^{\frac{pq+\theta (p-1)}{q-p+1}}\right\} ,
\end{equation*}%
we get 
\begin{eqnarray*}
u(x_{0}) &\leq &C_{N,p,q,\mu }\left( \frac{1}{a(\frac{\left\vert
x_{0}\right\vert }{2})^{p}}\right) ^{\frac{1}{q+1-p}}+(\frac{b}{a})^{\frac{1%
}{q}} \\
&\leq &C_{N,p,q,\theta ,\mu }(\left\vert x_{0}\right\vert ^{-\frac{p+\theta 
}{q+1-p}}+C_{N,p,q,\theta ,\mu }^{\prime }\left\vert x_{0}\right\vert ^{-(%
\frac{pq+\theta (p-1)}{q-p+1}+\theta )\frac{1}{q}}=C_{N,p,q,\theta ,\mu
}\left\vert x_{0}\right\vert ^{-\frac{p+\theta }{q+1-p}},
\end{eqnarray*}%
where $C_{N,p,q,\theta ,\mu }$ only depends on the parameters. Then we get (%
\ref{Oss}).The estimate is also valid for the exterior problem in $\mathbb{R}%
^{N}\backslash \overline{B_{\frac{r_{0}}{2}}},$ and the global problem in $%
\mathbb{R}^{N}\backslash \left\{ 0\right\} .$ \bigskip
\end{proof}

\begin{remark}
Note that Osserman's estimate holds without any restriction assumption: it
is valid for any $\mu \in \mathbb{R}$, and any $\theta \in \mathbb{R};$ in
particular $u$ is bounded when $p+\theta =0$, and $u$ tends to $0$ as $%
\left\vert x\right\vert \rightarrow 0$ when $p+\theta <0$.
\end{remark}

\subsection{Regularity result}

Followinq the method of \cite{FrVe}, we check that estimates of the function 
$u$ imply estimates of the gradient:

\begin{proposition}
\label{regu}Let $q>p-1.$ Let $u$ be a solution of (\ref{pq}) in $%
B_{r_{0}}\backslash \left\{ 0\right\} $(resp. $\mathbb{R}^{N}\backslash 
\overline{B_{r_{0}}})$ such that 
\begin{equation*}
u(x)\leq C_{1}\left\vert x\right\vert ^{-\delta }\qquad \text{in }%
B_{r_{0}}\backslash \left\{ 0\right\} \quad (\text{resp.}\mathbb{R}%
^{N}\backslash \overline{B_{r_{0}}})
\end{equation*}%
for some $\delta \leq \gamma $ (resp. $\delta \geq \gamma ).$ Then 
\begin{equation}
\left\vert \nabla u(x)\right\vert \leq C_{2}\left\vert x\right\vert
^{-(\delta +1)}\qquad \text{in }B_{\frac{r_{0}}{2}}\backslash \left\{
0\right\} ,  \label{hou}
\end{equation}%
and there exists $\alpha \in (0,1)$ only depending of $N,p,q,\delta $ such
that, for any $x,x^{\prime }\in $ $B_{\frac{r_{0}}{2}}\backslash \left\{
0\right\} $ $($resp. $\mathbb{R}^{N}\backslash \overline{B_{2r_{0}}})$ 
\begin{equation}
\left\vert \nabla u(x)-\nabla u(x^{\prime })\right\vert \leq
C_{3}(\left\vert x^{\prime }\right\vert ^{-\delta }+\left\vert x\right\vert
^{-\delta })\left\vert x\right\vert ^{-(1+\alpha )}\left\vert x-x^{\prime
}\right\vert ^{\alpha }.  \label{gou}
\end{equation}%
\medskip
\end{proposition}

\begin{proof}
Let $u$ be a solution of (\ref{pq}) in $B_{r_{0}}\backslash \left\{
0\right\} .$ By the scaling (\ref{scal}) we can assume $r_{0}=1.$ Let $%
\Gamma =B_{7}\backslash \overline{B_{1}}$ and $\Gamma ^{\prime
}=B_{6}\backslash \overline{B_{2}}$. For any $R\in (0,\frac{1}{7}),$ we set $%
u_{R}(\xi )=R^{\delta }u(R\xi )$ for any $\xi \in \Gamma $. Then 
\begin{equation*}
-\Delta _{p}u_{R}(\xi )+\frac{\mu u_{R}^{p-1}(\xi )}{\left\vert \xi
\right\vert ^{p}}+R^{(q+1-p)(\gamma -\delta )}u_{R}^{q}(\xi )=0
\end{equation*}%
for $\xi \in \Gamma $. Then $u_{R}$ is bounded in $\Gamma :$ 
\begin{equation*}
u_{R}(\xi )\leq C_{1}\left\vert \xi \right\vert ^{\delta }\leq C_{1}\max
\left\{ 1,7^{\delta }\right\} .
\end{equation*}%
Moreover $R^{(q+1-p)(\gamma -\delta )}$ stays bounded on $(0,\frac{1}{7})$,
thanks to the assumptions on $\delta $ and $q,$ thus $\Delta _{p}u_{R}$ is
bounded in $\Gamma $. Then from Tolksdorf results of \cite[Theorem 1]{To}, $%
\left\vert \nabla u_{R}\right\vert $ is bounded in $C^{0,\alpha }(\Gamma
^{\prime })$ independently of $R$ for some $\alpha \in (0,1)$. For $x\in $ $%
B_{\frac{1}{2}}\backslash \left\{ 0\right\} $ there exists $R\in (0,\frac{1}{%
7})$ such that $2R<\left\vert x\right\vert <6R$.Then (\ref{hou}) holds.
Moreover, let $x,x^{\prime }\in B_{\frac{1}{2}}\backslash \left\{ 0\right\} $%
. First assume $\left\vert x\right\vert \leq \left\vert x^{\prime
}\right\vert \leq 2\left\vert x\right\vert $. Then $\xi =\frac{x}{R}$ and $%
\xi ^{\prime }=\frac{x^{\prime }}{R}\in \Gamma ^{\prime }$, and 
\begin{equation*}
\left\vert \nabla u(x)-\nabla u(x^{\prime })\right\vert =R^{-1-\delta
}\left\vert \nabla u_{R}(\xi )-\nabla u_{R}(\xi ^{\prime })\right\vert \leq
CR^{-1-\delta }\left\vert \xi -\xi ^{\prime }\right\vert ^{\alpha }\leq
C_{2}\left\vert x\right\vert ^{-(\delta +1+\alpha )}\left\vert x-x^{\prime
}\right\vert ^{\alpha }.
\end{equation*}%
Next assume $\left\vert x^{\prime }\right\vert \geq 2\left\vert x\right\vert 
$. Then 
\begin{equation*}
\left\vert \nabla u(x)-\nabla u(x^{\prime })\right\vert \leq C(\frac{%
\left\vert x\right\vert ^{-\delta }}{\left\vert x\right\vert }+\frac{%
\left\vert x^{\prime }\right\vert ^{-\delta }}{\left\vert x\right\vert }%
)\leq 2C\frac{\left\vert x^{\prime }\right\vert ^{-\delta }+\left\vert
x\right\vert ^{-\delta }}{\left\vert x\right\vert ^{\alpha +1}}\left\vert
x^{\prime }-x\right\vert ^{\alpha },
\end{equation*}%
implying (\ref{gou}). We get analogous results hold in $\mathbb{R}%
^{N}\backslash \overline{B_{r_{0}}}$ when $\delta \geq \gamma .\medskip $
\end{proof}

Moreover in the case of a logarithmic type estimate, we obtain the following:

\begin{proposition}
\label{loga} Let $q>p-1$ and $s>0.$ Let $u$ be a solution of (\ref{pq}) in $%
B_{r_{0}}\backslash \left\{ 0\right\} $ with $r_{0}\leq 1,$ (resp. in $%
\mathbb{R}^{N}\backslash \overline{B_{r_{0}}}$ with $r_{0}\geq 1$) such that 
\begin{equation*}
u(x)\leq C_{1}\left\vert x\right\vert ^{-\delta }(\left\vert \ln \left\vert
x\right\vert \right\vert +1)^{s}\qquad \text{in }B_{r_{0}}\backslash \left\{
0\right\}
\end{equation*}%
for some $\delta <\gamma $ (resp. $\delta >\gamma ).$ Then for any $%
x,x^{\prime }\in $ $B_{\frac{r_{0}}{2}}\backslash \left\{ 0\right\} $ (resp. 
$\mathbb{R}^{N}\backslash \overline{B_{2r_{0}}}$), 
\begin{equation*}
\left\vert \nabla u(x)\right\vert \leq C_{2}\left\vert x\right\vert
^{-(\delta +1)}(\left\vert \ln \left\vert x\right\vert \right\vert +1)^{s}
\end{equation*}%
\begin{equation*}
\left\vert \nabla u(x)-\nabla u(x^{\prime })\right\vert \leq
C_{3}(\left\vert x^{\prime }\right\vert ^{-\delta }+\left\vert x\right\vert
^{-\delta })\left\vert x\right\vert ^{-(1+\alpha )}(\left\vert \ln
\left\vert x\right\vert \right\vert +1)^{s}\left\vert x-x^{\prime
}\right\vert ^{\alpha }.
\end{equation*}
\end{proposition}

\begin{proof}
As above we can assume $r_{0}=1.$ Let $u$ be a solution of (\ref{pq}) in $%
B_{1}\backslash \left\{ 0\right\} .$Taking $u_{R}(\xi )=R^{\delta
}\left\vert \ln R\right\vert ^{-s}u(R\xi )$ for any $\xi \in \Gamma ,$ we
obtain 
\begin{equation*}
-\Delta _{p}u_{R}(\xi )+\frac{\mu u_{R}^{p-1}(\xi )}{\left\vert \xi
\right\vert ^{p}}+R^{(q+1-p)(\gamma -\delta )}\left\vert \ln R\right\vert
^{s(q-p+1)}u_{R}^{q}(\xi )=0.
\end{equation*}%
The function $u_{R}$ is still bounded in $\Gamma $ because $s\geq 0:$ 
\begin{equation*}
u_{R}(\xi )\leq C_{1}\left\vert \xi \right\vert ^{\delta }(\frac{\left\vert
\ln (R\xi )\right\vert +1}{\left\vert \ln R\right\vert })^{s}\leq
C_{1}\left\vert \xi \right\vert ^{\delta }(1+\frac{\ln \xi +1}{\left\vert
\ln R\right\vert })^{s}\leq C_{1}\max \left\{ 1,7^{\delta }\right\} (1+\frac{%
\ln 7+1}{\ln 7})^{s}.
\end{equation*}%
And $R^{(q+1-p)(\gamma -\delta )}\left\vert \ln R\right\vert ^{s(q-p+1)}$ is
bounded on $(0,\frac{1}{7}),$ because $\delta <\gamma .$ Then we conclude as
before, and similarly in $\mathbb{R}^{N}\backslash \overline{B_{r_{0}}}.$
\end{proof}

\subsection{Harnack inequality}

This argument is fundamental, valid for $q>p-1:$

\begin{proposition}
\label{Harnack}.There exists a constant $c=c_{N,p,q,\mu ,\theta }>0$ such
that for any $r_{0}>0$ and any solution $u$ of (\ref{pq}) in $%
B_{r_{0}}\backslash \left\{ 0\right\} $ (resp. $\mathbb{R}^{N}\backslash 
\overline{B_{r_{0}}}$, resp. $\mathbb{R}^{N}\backslash \left\{ 0\right\} $ )%
\begin{equation}
\sup_{\left\vert x\right\vert =r}u(x)\leq c\inf_{\left\vert x\right\vert
=r}u(x),\qquad \forall r\in (0,\frac{r_{0}}{2})\text{ }(\text{resp. }%
(2r_{0},\infty ),\text{resp. }(0,\infty ))\text{.}  \label{har}
\end{equation}%
\medskip
\end{proposition}

\begin{proof}
It is analogous to the proof of \cite[Lemma 2.2]{FrVe} given in case $\mu
=\theta =0$ and is a consequence of the Osserman's estimate: for any $%
x_{0}\in B_{\frac{r_{0}}{2}}\backslash \left\{ 0\right\} ,$ we write the
equation (\ref{pq}) in a ball $B(x_{0},\frac{\left\vert x_{0}\right\vert }{2}%
)$ under the form 
\begin{equation*}
-\Delta _{p}u+\psi ^{p}u^{p-1}=0,
\end{equation*}%
where 
\begin{equation*}
\psi ^{p}=\frac{\mu }{\left\vert x\right\vert ^{p}}+\left\vert x\right\vert
^{\theta }u^{q+1-p}.
\end{equation*}%
From Trudinger \cite[p.724]{Tr}, there exists a constant $C_{1}$ depending
on $N,p,\left\vert x_{0}\right\vert \sup_{B(x_{0},\frac{\left\vert
x_{0}\right\vert }{2})}\left\vert \varphi \right\vert $ such that 
\begin{equation*}
\sup_{x\in B(x_{0},\frac{\left\vert x_{0}\right\vert }{6})}u(x)\leq
C\inf_{x\in B(x_{0},\frac{\left\vert x_{0}\right\vert }{6})}u(x).
\end{equation*}%
From (\ref{Oss}), $\left\vert x_{0}\right\vert \sup_{B(x_{0},\frac{%
\left\vert x_{0}\right\vert }{2})}\left\vert \psi \right\vert $ is bounded
by a constant only depending on $N,p,q,\theta ;$ then $C$ only depends on $%
N,p,q,\theta ,\mu .$ Then (\ref{har}) holds by connecting two points $%
x_{1},x_{2}$ such that $\left\vert x_{1}\right\vert =\left\vert
x_{2}\right\vert =r<\frac{r_{0}}{2}$ by 10 connected balls of radius $\frac{r%
}{6}.$
\end{proof}

\subsection{Existence of radial solutions in $\protect\omega %
=B_{r_{2}}\backslash \overline{B_{r_{1}}}$}

\begin{proposition}
\label{exir}$\ $Let $\mu \in \mathbb{R}.$ Then for any $0<R_{1}<R_{2}$ and
any $\ell _{1}\geq 0,\ell _{2}\geq 0,$ with $\ell _{1}+\ell _{2}\neq 0$,
there exists a unique positive function $v=v_{\ell _{1},\ell _{2}}$ in $%
\omega =B_{R_{2}}\backslash \overline{B_{R_{1}}}$ such that $v\in C(%
\overline{\omega })$ and 
\begin{equation}
\left\{ 
\begin{array}{c}
-\Delta _{p}v+\mu \frac{v^{p-1}}{\left\vert x\right\vert ^{p}}+\left\vert
x\right\vert ^{\theta }v^{q}=0\text{ in }\omega =B_{R_{2}}\backslash 
\overline{B_{R_{1}}}, \\ 
v(x)=\ell _{1}\text{ for }\left\vert x\right\vert =R_{1}, \\ 
v(x)=\ell _{2}\text{ for }\left\vert x\right\vert =R_{2},%
\end{array}%
\right.  \label{pro}
\end{equation}%
and it is radial.
\end{proposition}

\begin{proof}
From Corollary \ref{DG}, if such a positive solution exists, it is unique,
then it is radial. We proceed by minimisation in a space of radial
functions. Let $\varphi (x)=\ell _{1}+(\ell _{2}-\ell _{1})\frac{\left\vert
x\right\vert -r_{1}}{r_{2}-r_{1}},$ be a smooth radial function in $%
\overline{\omega }$ satisfying the boundary conditions. We define a function 
$J$ on $W_{\varphi }=\left\{ v\in W_{rad}^{1,p}(\omega )\mid v=\varphi \text{
on }\partial \omega \right\} $ by 
\begin{equation*}
J(v)=\int_{\omega }(\left\vert \nabla v\right\vert ^{p}+\mu \frac{%
(v^{+})^{p}}{\left\vert x\right\vert ^{p}}+\frac{p}{q+1}\left\vert
x\right\vert ^{\theta }\left\vert v\right\vert ^{q+1})dx.
\end{equation*}%
Then 
\begin{equation*}
J(v)\geq \int_{\omega }(\left\vert \nabla (v)\right\vert
^{p}-\left\vert \mu \right\vert c_{1}^{p}\left\vert v\right\vert
^{p}+c_{2}\left\vert v\right\vert ^{q})dx,
\end{equation*}%
with two constants $c_{1},c_{2}$ only depending on $r_{1},r_{2}$ and $\theta
,p,q.$ For $q>p-1,$ from the H\"{o}lder inequality, with $c_{3}$ depending
on $c_{1},c_{2}$ and $\left\vert \mu \right\vert $ 
\begin{equation*}
J(v)\geq \int_{\omega }(\left\vert \nabla v\right\vert ^{p}+\frac{%
c_{2}}{2}\left\vert v\right\vert ^{q}-c_{3})dx.
\end{equation*}%
Then $\lim_{v\in W_{\varphi },\left\Vert v\right\Vert _{W^{1,p}(\omega
)}\longrightarrow \infty }J(v)=\infty ,$ and by compacity, $\inf_{w\in
W_{\varphi }}J$ is attained at least at some $\widetilde{v}$ $\in W_{\varphi
}.$ When $\mu \geq 0,$ it is clear that the problem of minimisation admits a
unique solution, but not when $\mu <0.$ Concerning the question of
positivity, we observe that the function $\widetilde{v}$ $^{+}\in W_{\varphi
},$ thus $\widetilde{v}$ $^{+}\in C(\overline{\omega }),$ and satisfies $J(%
\widetilde{v})\geq J(\widetilde{v}_{+})$; then $J(\widetilde{v}%
_{+})=\inf_{v\in W_{\varphi }}$ $J(v),$ and $\widetilde{v}$ $^{+}$ satisfies
the equation (\ref{pq}) in $\mathcal{D}^{\prime }(\omega ),$ and the
conditions on $\partial \omega .$ From \cite[Theorem 1]{To}, $\widetilde{v}$ 
$^{+}\in C^{1}(\omega ).$ Hence $\widetilde{v}_{+}$ is a nonnegative
solution of problem (\ref{pro}). Moreover, from Theorem \ref{SMP}, either $%
\widetilde{v}^{+}\equiv 0,$ which contradicts the assumption $\ell _{1}+\ell
_{2}\neq 0;$ then $\widetilde{v}$ $^{+}>0.$ Then $\widetilde{v}$ $^{+}$ is
the unique solution of (\ref{pro}).\bigskip
\end{proof}

Next we give \textbf{the key-point of our study}, consequence of Osserman's
estimate and Harnack inequality. It is is a comparison from above and below
with radial solutions with the same behaviour

\begin{theorem}
\label{clef}Let $u$ be any positive solution of (\ref{pq}) in $%
B_{r_{0}}\backslash \left\{ 0\right\} $ (resp. $\mathbb{R}^{N}\backslash 
\overline{B_{r_{0}}}$ ) (resp. $\mathbb{R}^{N}\backslash \left\{ 0\right\} $%
). Then there exist radial solutions $v$ and $w$ such that 
\begin{equation*}
v\leq u\leq w\leq cv\text{ \qquad in }B_{\frac{r_{0}}{2}}\backslash \left\{
0\right\} \quad \text{(resp.}\mathbb{R}^{N}\backslash \overline{B_{2r_{0}}}%
\text{),(resp}.\mathbb{R}^{N}\backslash \left\{ 0\right\} \text{)},
\end{equation*}%
where $c=c_{N,p,q,\mu ,\theta }$ is the Harnack constant defined at
Proposition \ref{Harnack}.
\end{theorem}

\begin{proof}
Let $u$ be any positive solution of (\ref{pq}) in $B_{r_{0}}\backslash
\left\{ 0\right\} $. From Proposition \ref{exir}$,$ for any $ {integer}$ 
$n\geq 1,$ there exist radial positive solutions of (\ref{pq}) in $%
B_{r_{0}}\backslash \overline{B_{r_{n}}}$ such that 
\begin{eqnarray}
v_{n}(x) &=&\min_{\left\vert y\right\vert =\frac{1}{n}}u(y)\text{ for }%
\left\vert x\right\vert =\frac{1}{n},\quad \quad v_{n}(x)=\min_{\left\vert
y\right\vert =\frac{r_{0}}{2}}u(y)\text{ for }\left\vert x\right\vert =\frac{%
r_{0}}{2},  \label{vn} \\
w_{n}(x) &=&\max_{\left\vert y\right\vert =\frac{1}{n}}u(y)\text{ for }%
\left\vert x\right\vert =\frac{1}{n},\quad \quad w_{n}(x)=\max_{\left\vert
y\right\vert =\frac{r_{0}}{2}}u(y)\text{ for }\left\vert x\right\vert =\frac{%
r_{0}}{2},  \label{wn}
\end{eqnarray}%
Then from Corollary \ref{DG} we get $v_{n}\leq u\leq w_{n}.$ Moreover there
holds from Proposition \ref{Harnack} 
\begin{equation*}
w_{n}(x)\leq c_{N,p,q,\mu ,\theta }v_{n},\qquad \text{for }\left\vert
x\right\vert =\frac{1}{n}\text{and for }\left\vert x\right\vert =\frac{r_{0}%
}{2},
\end{equation*}%
We consider the function $y_{n}=c_{N,p,q,\mu ,\theta }v_{n},$ which
satisfies 
\begin{equation*}
-\Delta _{p}y_{n}+\mu \frac{y_{n}^{p-1}}{\left\vert x\right\vert ^{p}}%
+c_{N,p,q,\mu ,\theta }^{p-1-q}\left\vert x\right\vert ^{\theta }y_{n}^{q}=0.
\end{equation*}%
Since $c_{N,p,q,\mu ,\theta }\geq 1$ and $q>p-1$, the function $y_{n}$ is a 
\textbf{supersolution} of the equation, and greater than $w_{n}$ for $%
\left\vert x\right\vert =\frac{1}{n}$and for $\left\vert x\right\vert =r_{0}$%
. From Corollary \ref{DG}, we get $w_{n}\leq y_{n}$, then 
\begin{equation*}
v_{n}\leq u\leq w_{n}\leq c_{N,p,q,\mu ,\theta }v_{n}\leq c_{N,p,q,\mu
,\theta }u
\end{equation*}%
in $\overline{B}_{\frac{r_{0}}{2}}\backslash B_{\frac{1}{n}}.$ Then from
Theorem \ref{Osserman}, 
\begin{equation*}
v_{n}\leq C_{N,p,q,\mu ,\theta }\left\vert x\right\vert ^{-\gamma },\qquad
w_{n}\leq c_{N,p,q,\mu ,\theta }C_{N,p,q,\mu ,\theta }\left\vert
x\right\vert ^{-\gamma }.
\end{equation*}%
For any fixed $\varepsilon \in (0,\frac{r_{0}}{8}),$ and $n>\frac{1}{%
\varepsilon },$ the sequences $v_{n}$ and $w_{n}$ are uniformly bounded in $%
C(\overline{B}_{\frac{r_{0}}{2}}\backslash B_{\varepsilon })$ and in $%
W^{1,p}(\overline{B}_{\frac{r_{0}}{2}}\backslash B_{\varepsilon }).$ As a
consequence, $\Delta _{p}v_{n}$ and $\Delta _{p}w_{n}$ are uniformly bounded
in $\overline{B}_{\frac{r_{0}}{2}-\varepsilon }\backslash B_{\varepsilon },$
thus from \cite[Theorem 1]{To}, $v_{n}$ and $w_{n}$ are uniformly bounded in 
$C^{1,\alpha }(\overline{B}_{\frac{r_{0}}{2}-\varepsilon }\backslash
B_{\varepsilon })$ for some $\alpha >0.$ By a diagonal process, there exists
subsequences $v_{\nu }$ and $w_{\nu }$ converging strongly in $%
C_{loc}^{1}(B_{\frac{r_{0}}{2}}\backslash \left\{ 0\right\} )$ and weakly in 
$W_{loc}^{1,p}(B_{\frac{r_{0}}{2}}\backslash \left\{ 0\right\} ))$ to radial
functions $v$ and $w,$ solutions of (\ref{pq}), such that $v\leq u\leq w\leq
cv$ in $B_{\frac{r_{0}}{2}}\backslash \left\{ 0\right\} .$

We get the same conclusions in $\mathbb{R}^{N}\backslash \overline{B_{r_{0}}}
$ (resp. $\mathbb{R}^{N}\backslash \left\{ 0\right\} )$ by appying
Proposition \ref{exir} in $B_{n}\backslash \overline{B_{2r_{0}}}$ (resp. in $%
B_{n}\backslash \overline{B_{\frac{1}{n}}}$ ).\bigskip
\end{proof}

\subsection{Precise convergence results}

Finally give a more precise behaviour of the possibly nonradial solutions of
Hardy type, extending some results of \cite{FrVe}.$\medskip $

\textbf{Notation:} $u\asymp _{x\longrightarrow 0}\widetilde{u}$ means that
there exists constants $C_{1},C_{2}>0$ such that $C_{1}\widetilde{u}\leq
u\leq C_{2}\widetilde{u}$ near $0.\medskip $

\begin{proposition}
\label{exact}Suppose that $u$ is a solution in $B_{r_{0}}\backslash \left\{
0\right\} $ of (\ref{pq}) such that $u\asymp _{x\longrightarrow 0}\left\vert
x\right\vert ^{-S_{i}}$ $(i=1$ or $2,S_{i}\neq 0$), where $S_{i}$ $<\gamma $
(resp. $u\asymp _{x\longrightarrow \infty }\left\vert x\right\vert ^{-S_{i}}$
where $S_{i}$ $>\gamma $ ).Then there exists $k_{i}>0$ such that 
\begin{equation*}
\lim_{x\longrightarrow 0}\left\vert x\right\vert ^{S_{i}}u=k_{i}\text{
(resp. }\lim_{x\longrightarrow \infty }\left\vert x\right\vert
^{S_{i}}u=k_{i}\text{)}
\end{equation*}
\end{proposition}

\begin{proof}
Here we proceed as in \cite{FrVe} for the case $\mu =0,$ $i=1$. We consider
for example the case where $u\asymp _{x\longrightarrow 0}\left\vert
x\right\vert ^{-S_{i}}$, with $S_{i}$ $<\gamma $. Let $v_{i}=\left\vert
x\right\vert ^{-S_{i}},$ and 
\begin{equation*}
k=\lim \sup_{x\longrightarrow 0}\frac{u}{v_{i}}(x)\text{ \qquad and\qquad\ }%
\widetilde{k}(r)=\sup_{\left\vert x\right\vert =r}\frac{u}{v_{i}}(x),\quad
\forall r\in \left( 0,r_{0}\right) .
\end{equation*}

(i) We first check that 
\begin{equation}
\lim_{r\longrightarrow 0}\widetilde{k}(r)=k.  \label{tri}
\end{equation}%
Indeed there holds $\lim \sup_{r\longrightarrow 0}\widetilde{k}(r)=k$.
Suppose that (\ref{tri}) is false, then $\lim \inf_{r\longrightarrow 0}%
\widetilde{k}(r)=k_{0}<k$. There exists $r^{\star }<r_{0}$ such that $%
\widetilde{k}(r^{\ast })>\frac{k_{0}+k}{2}.$ There exists a decreasing
sequence $(r_{n})_{n\geq 1}\rightarrow 0$ such that $\widetilde{k}%
(r_{n})\rightarrow k_{0},$ then $\widetilde{k}(r_{n})<\frac{k_{0}+k}{2}$ and 
$r_{n}<r^{\star }<r_{0}$ for $n\geq n_{0}$ if $n_{0}$ is large enough. Then $%
M=\sup_{\overline{B_{r_{n_{0}}}}\backslash B_{r_{n}}}\frac{u}{v_{i}}$ is not
attained on the boundary, but in an interior point $x_{0}\in $ $%
B_{r_{n_{0}}}\backslash \overline{B_{r_{n}}}.$ But there holds 
\begin{equation*}
\mathcal{L}_{p,\mu }(Mv_{i})=-\Delta _{p}(Mv_{i})+\mu \frac{(Mv_{i})^{p-1}}{%
\left\vert x\right\vert ^{p}}=0\geq \mathcal{L}_{p,\mu }u,
\end{equation*}%
and $\left\vert \nabla v_{i}\right\vert $ does not vanish in $%
B_{r_{n_{0}}}\backslash \overline{B_{r_{n}}},$ and $u\leq Mv_{i}$ in $%
B_{r_{n_{0}}}\backslash \overline{B_{r_{n}}}$ with $u(x_{0})=Mv_{i}(x_{0}).$
Then%
\begin{equation*}
- {div}(\left\vert \nabla Mv_{i}\right\vert ^{p-2}\nabla
Mv_{i}-\left\vert \nabla u\right\vert ^{p-2}\nabla u+\Phi (Mv_{i}-u)\geq 0,%
\text{ }
\end{equation*}%
where $\Phi =\mu \frac{(Mv_{i})-u^{p-1}}{\left\vert x\right\vert
^{p}(Mv_{i}-u)}$ is bounded in $B_{r_{n_{0}}}\backslash \overline{B_{r_{n}}}$%
, since $u$ and $v$ are positive. As in \cite[Lemma 1.3]{FrVe}, we get $%
u\equiv Mv_{i}$ in $B_{r_{n_{0}}}\backslash \overline{B_{r_{n}}},$ which is
contradictory. \medskip

(ii) Next we consider for some $r_{0}\in (0,1)$ the scaled function for
given $r>0$ 
\begin{equation*}
u_{r}(\xi )=\frac{u(r\xi )}{v_{i}(r)}\qquad \text{for }0<\left\vert \xi
\right\vert <\frac{r_{0}}{r}
\end{equation*}%
and choose $\xi _{r}\in S^{N-1}$ such that $\widetilde{k}(r)=r^{S_{i}}u(r\xi
_{r})$. By computation, $u_{r}$ satisfy the equation 
\begin{equation*}
-\Delta _{p}u_{r}(\xi )+\mu \frac{u_{r}^{p-1}}{\left\vert \xi \right\vert
^{p}}+r^{d_{i}}u_{r}^{q}=0
\end{equation*}%
with $d_{i}=(q-p+1)(\gamma -S_{i})>0$, then $\lim_{r\longrightarrow
0}r^{d_{i}}=0$ as $r\longrightarrow 0$. Moreover, by assumption there holds $%
u(x)\leq C\left\vert x\right\vert ^{-S_{i}}$ in $B_{\frac{r_{0}}{2}%
}\backslash \left\{ 0\right\} $, implying precise estimates of the gradient
from Proposition \ref{regu} with $\delta =S_{i}$: for any $\left\vert \xi
\right\vert <\frac{r_{0}}{2r},\left\vert \xi ^{\prime }\right\vert <\frac{%
r_{0}}{2r}$, 
\begin{eqnarray*}
u_{r}(\xi ) &\leq &C\left\vert \xi \right\vert ^{-S_{i}},\qquad \left\vert
\nabla u_{r}(\xi )\right\vert \leq C\left\vert x\right\vert ^{-(S_{i}+1)}, \\
\left\vert \nabla u_{r}(\xi )-\nabla u_{r}(\xi ^{\prime })\right\vert &\leq
&C\frac{\left\vert \xi \right\vert ^{-S_{i}}+\left\vert \xi ^{\prime
}\right\vert ^{-S_{i}}}{\left\vert \xi \right\vert ^{\alpha +1}}\left\vert
\xi -\xi ^{\prime }\right\vert ^{\alpha }.
\end{eqnarray*}%
Otherwise, by definition, 
\begin{equation*}
\frac{u_{r}(\xi )}{v_{i}(\xi )}=(r\xi )^{S_{i}}u(r\xi )\leq \widetilde{k}%
(r\left\vert \xi \right\vert ),\qquad \frac{u_{r}(\xi _{r})}{v_{i}(\xi _{r})}%
=r^{S_{i}}u(r\xi _{r})=\widetilde{k}(r).
\end{equation*}%
Thus for any sequence $\widetilde{r_{n}}\rightarrow 0$ we can extract a
subsequence $r_{n}$ such that $u_{r_{n}}$ converges in $C^{1,loc}(\mathbb{R}%
^{N}\backslash \left\{ 0\right\} $ to a function $w\geq 0$, and $\xi
_{r_{n}} $ converges to $\widetilde{\xi }\in S^{N-1}.$ Then 
\begin{equation*}
-\Delta _{p}w+\mu \frac{w^{p-1}}{\left\vert \xi \right\vert ^{p}}=0\text{ in 
}\mathbb{R}^{N}\backslash \left\{ 0\right\} ,
\end{equation*}%
and $w(\xi )\leq kv_{i}(\xi )$ for any $\xi \in \mathbb{R}^{N}\backslash
\left\{ 0\right\} ,$ with $w(\widetilde{\xi })=k$ from (\ref{tri}). Note
that $kv_{i}$ is also a solution of this equation, and $\left\vert \nabla
v_{i}\right\vert $ does not vanishes in $\mathbb{R}^{N}\backslash \left\{
0\right\} ,$ then as above $w\equiv kv_{i}$ from \cite[Lemma 1.3]{FrVe}.
Since $\widetilde{r_{n}}$ is arbitrary, $\lim_{r\longrightarrow 0}u_{r}(\xi
)=k\left\vert \xi \right\vert ^{-S_{i}},$ then taking $\left\vert \xi
\right\vert =1,$ we obtain 
\begin{equation*}
\lim_{x\longrightarrow 0}\left\vert x\right\vert ^{S_{i}}u(x)=k.
\end{equation*}%
We get analogous results near $\infty ,$ by using Proposition \ref{regu} in $%
\mathbb{R}^{N}\backslash \overline{B_{r_{0}}}$ with $\delta =S_{i}>\gamma $%
.\medskip \medskip $\varpi $
\end{proof}

In the particular case $\mu =\mu _{0}$ where $S_{1}=S_{2}=\frac{N-p}{p}%
<\gamma $ (resp. $\frac{N-p}{p}>\gamma )$ there exists two types of radial
singular solutions near $0$ (resp. near $\infty $ ) there still exists the
functions $r\longmapsto \ell r^{-\frac{N-p}{p}},$ $\ell >0$ for which the
proposition is still valid, but also some functions $v$ such that $%
\lim_{r\longrightarrow 0}r^{\frac{N-p}{p}}\left\vert \ln r\right\vert
^{-1}v(r)=\ell ,$ see Lemma \ref{critic}. We get the following:

\begin{proposition}
\label{excrit} Let $\mu =\mu _{0}.$ Suppose that $u$ is a solution in $%
B_{r_{0}}\backslash \left\{ 0\right\} $ (resp $\mathbb{R}^{N}\backslash 
\overline{B_{r_{0}}}$) of (\ref{pq}) such that $u\asymp _{x\longrightarrow
0}\left\vert x\right\vert ^{-\frac{N-p}{p}}\left\vert \ln \left\vert
x\right\vert \right\vert ^{\frac{2}{p}}$ and $\frac{N-p}{p}<\gamma $ (resp. $%
u\asymp _{x\longrightarrow \infty }\left\vert x\right\vert ^{-\frac{N-p}{p}%
}(\ln \left\vert x\right\vert )^{\frac{2}{p}}$ and $\frac{N-p}{p}>\gamma $
).Then there exists $k>0$ such that 
\begin{equation*}
\lim_{x\longrightarrow 0}\left\vert x\right\vert ^{-\frac{N-p}{p}}\left\vert
\ln \left\vert x\right\vert \right\vert ^{-\frac{2}{p}}u=k\text{ (resp. }%
\lim_{\left\vert x\right\vert \longrightarrow \infty }\left\vert
x\right\vert ^{-\frac{N-p}{p}}(\ln \left\vert x\right\vert )^{-\frac{2}{p}%
}u=k\text{)}
\end{equation*}
\end{proposition}

\begin{proof}
Consider for example a solution $u$ in $B_{r_{0}}\backslash \left\{
0\right\} $ with such a behaviour, let $v_{0}$ be the unique radial solution
of the equation 
\begin{equation*}
\mathcal{L}_{p,\mu }(v_{0})=-\Delta _{p}v_{0}+\mu _{0}\frac{v_{0}^{p-1}}{%
\left\vert x\right\vert ^{p}}=0
\end{equation*}%
such that $\lim_{r\longrightarrow 0}r^{\frac{N-p}{p}}\left\vert \ln
r\right\vert ^{-\frac{2}{p}}v_{0}(r)=1,$ given at Lemma \ref{critic}. Let 
\begin{equation*}
k=\lim \sup_{x\longrightarrow 0}\frac{u}{v_{0}}(x)>0\text{ \qquad and\qquad\ 
}\widetilde{k}(r)=\sup_{\left\vert x\right\vert =r}\frac{u}{v_{0}}(x),\quad
\forall r\in \left( 0,r_{0}\right) .
\end{equation*}%
(i) We first check that $\lim_{r\longrightarrow 0}\widetilde{k}(r)=k,$
exactly as before, where $v_{i}$ is replaced by $v_{0}.\medskip $

\noindent (ii) Next consider for some $r_{0}\in (0,1)$ and given $r>0$ 
\begin{equation*}
u_{r}(\xi )=\frac{u(r\xi )}{v_{0}(r)}\qquad \text{for }0<\left\vert \xi
\right\vert <\frac{r_{0}}{r}.
\end{equation*}%
and we choose $\xi _{r}\in S^{N-1}$ such that $\widetilde{k}%
(r)=v_{0}(r)u(r\xi _{r})$. Then $u_{r}$ satisfies the equation 
\begin{equation*}
-\Delta _{p}u_{r}(\xi )+\mu \frac{u_{r}^{p-1}}{\left\vert \xi \right\vert
^{p}}+(r^{\gamma }v_{0})^{q+1-p}u_{r}^{q}=0
\end{equation*}%
Since $\frac{N-p}{p}<\gamma $ there holds $\lim_{r\longrightarrow
0}(r^{\gamma }v_{0})^{q+1-p}=0$. By hypothesis, for $u(x)\leq Cv_{0}\left(
\left\vert x\right\vert \right) \leq C_{1}\left\vert x\right\vert ^{-\frac{%
N-p}{p}}\left\vert \ln \left\vert x\right\vert \right\vert ^{\frac{2}{p}}$
in $B_{\frac{r_{0}}{2}}\backslash \left\{ 0\right\} .$ From Proposition \ref%
{loga}, for any $x,x^{\prime }\in $ $B_{\frac{r_{0}}{2}}\backslash \left\{
0\right\} ,$ 
\begin{equation*}
\left\vert \nabla u(x)\right\vert \leq C_{2}\left\vert x\right\vert ^{-\frac{%
N}{p}}(\left\vert \ln \left\vert x\right\vert \right\vert +1)^{\frac{2}{p}%
}\leq 2^{\frac{2}{p}}C_{2}\left\vert x\right\vert ^{-\frac{N}{p}}\left\vert
\ln \left\vert x\right\vert \right\vert ^{\frac{2}{p}}
\end{equation*}%
\begin{equation*}
\left\vert \nabla u(x)-\nabla u(x^{\prime })\right\vert \leq 2^{\frac{2}{p}%
}C_{3}(\left\vert x^{\prime }\right\vert ^{-\frac{N-p}{p}}+\left\vert
x\right\vert ^{-\frac{N-p}{p}})\left\vert x\right\vert ^{-(1+\alpha
)}\left\vert \ln \left\vert x\right\vert \right\vert ^{\frac{2}{p}%
}\left\vert x-x^{\prime }\right\vert ^{\alpha }.
\end{equation*}%
Then for $\left\vert \xi \right\vert ,\left\vert \xi ^{\prime }\right\vert <%
\frac{r_{0}}{2r}<\frac{1}{2r},$ 
\begin{equation*}
u_{r}(\xi )\leq C\frac{v_{0}(r\xi )}{v_{0}(r)}\leq C_{1}\left\vert \xi
\right\vert ^{-\frac{N-p}{p}}(\frac{\left\vert \ln r\xi \right\vert }{%
\left\vert \ln r\right\vert })^{\frac{2}{p}}\leq C_{1}\left\vert \xi
\right\vert ^{-\frac{N-p}{p}}(1+\frac{\left\vert \ln \xi \right\vert }{%
\left\vert \ln r\right\vert })^{\frac{2}{p}}
\end{equation*}%
\begin{equation*}
\left\vert \nabla u_{r}(\xi )\right\vert =\frac{r}{v_{0}(r)}\left\vert
\nabla u(r\xi )\right\vert \leq C_{4}r\left\vert r\xi \right\vert ^{-\frac{N%
}{p}}\frac{\left\vert \ln r\xi \right\vert ^{\frac{2}{p}}}{r^{-\frac{N-p}{p}%
}\left\vert \ln r\right\vert ^{\frac{2}{p}}}=C_{4}\left\vert \xi \right\vert
^{-\frac{N}{p}}(1+\frac{\left\vert \ln \xi \right\vert }{\left\vert \ln
r\right\vert })^{\frac{2}{p}}
\end{equation*}%
\begin{equation*}
\left\vert \nabla u_{r}(\xi )-\nabla u_{r}(\xi ^{\prime })\right\vert \leq
C_{3}r^{\alpha }(\left\vert \xi ^{\prime }\right\vert ^{-\frac{N-p}{p}%
}+\left\vert \xi \right\vert ^{-\frac{N-p}{p}})\left\vert \xi \right\vert
^{-(1+\alpha )}(1+\frac{\left\vert \ln \xi \right\vert }{\left\vert \ln
r\right\vert })^{\frac{2}{p}}\left\vert \xi -\xi ^{\prime }\right\vert
^{\alpha }.
\end{equation*}%
Then we still have that for any sequence $\widetilde{r_{n}}\longrightarrow
0, $ there exists a subsequence $r_{n}$ such that $u_{r_{\nu }}$ converges
in $C_{loc}^{1}(\mathbb{R}^{N}\backslash \left\{ 0\right\} )$ to a function $%
w,$ such that 
\begin{equation*}
\mathcal{L}_{p,\mu }(w)=-\Delta _{p}w(\xi )+\mu _{0}\frac{w^{p-1}}{%
\left\vert \xi \right\vert ^{p}}=0
\end{equation*}%
in $\mathbb{R}^{N}\backslash \left\{ 0\right\} ,$ and $\xi
_{r_{n}}\longrightarrow \xi _{0}\in S^{n-1},$ $u_{r_{n}}(\xi
_{r_{n}})\longrightarrow w(\xi _{0})$ and $\widetilde{k}(r_{n})=\frac{%
u_{r_{n}}(\xi _{r_{n}})}{v_{0}(r_{n})}$. But from \cite[Example 1.1]{FraPi},
the only solutions of this equation are the functions $\xi \longmapsto \ell
\left\vert \xi \right\vert ^{-\frac{N-p}{p}},$ $\ell >0.$ Then $w(\xi
_{0})=\ell \left\vert \xi _{0}\right\vert ^{-\frac{N-p}{p}}=\ell $ and $%
\widetilde{k}(r_{n})=u_{r_{n}}(\xi _{r_{n}})\longrightarrow k.$ Then $\ell
=k,$ thus $w$ is independant of the choice of $r_{n},$ then $%
\lim_{r\longrightarrow 0}u(r\xi )=k$ any $\xi \in \mathbb{R}^{N}\backslash %
\left[ 0\right] $, in particular for $\xi \in S^{n-1}.$ Then $%
\lim_{x\longrightarrow 0}\frac{u}{v_{0}}(x)=k.$ We get similar results in $%
\mathbb{R}^{N}\backslash \overline{B_{r_{0}}}$ when $\frac{N-p}{p}>\gamma $ .
\end{proof}

\section{Proof of the nonradial results\label{short}}

We begin by Theorem \ref{H5nonrad}, which is the simplest case, and the
proof is quite short.\medskip

\begin{proof}[Proof of Theorem \protect\ref{H5nonrad}]
Let $u$ be any solution in $B_{r_{0}}\backslash \left\{ 0\right\} $ in $%
B_{r_{0}}\backslash \left\{ 0\right\} .$ (resp. $\mathbb{R}^{N}\backslash 
\overline{B_{r_{0}}}$)$.$ From Theorem \ref{clef}, there exist two radial
solutions such that $v\leq u\leq w$ in $B_{\frac{r_{0}}{2}}\backslash
\left\{ 0\right\} $ (resp. in $\mathbb{R}^{N}\backslash \overline{B_{r_{0}}}$%
), (resp. $\mathbb{R}^{N}\backslash \left\{ 0\right\} $). From Theorem \ref%
{H5rad}$,$ where all the possible behaviours are described, there holds $%
\lim_{x\longrightarrow 0}\left\vert x\right\vert ^{\gamma }v(x)=a^{\ast
},\lim_{x\longrightarrow 0}\left\vert x\right\vert ^{\gamma }w(x)=a^{\ast }$
(resp $\lim_{x\longrightarrow \infty }\left\vert x\right\vert ^{\gamma
}v(x)=a^{\ast },\lim_{x\longrightarrow \infty }\left\vert x\right\vert
^{\gamma }w(x)=a^{\ast }$, resp. $\left\vert x\right\vert ^{\gamma
}v(x)=a^{\ast }=\left\vert x\right\vert ^{\gamma }w(x)).$ Hence we find $%
\lim_{x\longrightarrow 0}\left\vert x\right\vert ^{\gamma }u(x)=a^{\ast }$
(resp $\lim_{x\longrightarrow \infty }\left\vert x\right\vert ^{\gamma
}u(x)=a^{\ast }$, resp $\left\vert x\right\vert ^{\gamma }u(x)=a^{\ast }$%
).\bigskip
\end{proof}

Next we prove Theorem \ref{H1nonrad} (resp. \ref{H2nonrad}).\medskip

\begin{proof}[Proof of Theorem \protect\ref{H1nonrad}]
(i)\textbf{\ }Let $u$ be a solution in $B_{r_{0}}\backslash \left\{
0\right\} .$ From Theorem \ref{clef}, there exists radial functions such
that $v\leq u\leq w\leq cv$ in $B_{r_{0}}\backslash \left\{ 0\right\} .$
First suppose $\mu >\mu _{0}.$ From Theorem \ref{H1rad}, either $%
\lim_{r\longrightarrow 0}r^{\gamma }v=a^{\ast },$ then also $%
\lim_{r\longrightarrow 0}r^{\gamma }w=a^{\ast },$ then by squeezing $%
\lim_{x\longrightarrow 0}\left\vert x\right\vert ^{\gamma }u=a^{\ast }$; or $%
\lim_{r\longrightarrow 0}r^{S_{i}}v=k_{i}^{\prime }>0,$ with $i=1$ or $2$,
then also $\lim_{r\longrightarrow 0}r^{S_{i}}w=k_{i}^{\prime \prime }>0,$
implying that $u\asymp $ $\left\vert x\right\vert ^{-S_{i}},$ then $%
\lim_{x\longrightarrow 0}\left\vert x\right\vert ^{S_{i}}u=k_{i}>0$ from
Proposition \ref{exact}. For $\mu =\mu _{0}$ we still have a possibility
that $\lim_{x\rightarrow 0}\left\vert x\right\vert ^{\frac{N-p}{p}%
}(\left\vert \ln \left\vert x\right\vert \right\vert ^{-\frac{2}{p}%
})v(\left\vert x\right\vert )=\ell _{1}>0$, and $\lim_{x\rightarrow
0}\left\vert x\right\vert ^{\frac{N-p}{p}}\left\vert \ln \left\vert
x\right\vert \right\vert ^{-\frac{2}{p}}w(\left\vert x\right\vert )=\ell
_{2}>0$, then we get (\ref{autnew}) from Proposition \ref{excrit}.\medskip

(ii) Let $u$ be a solution in $\mathbb{R}^{N}\backslash \overline{B_{r_{0}}}%
. $ Then from Theorem \ref{H1rad}, $\lim_{r\longrightarrow \infty }r^{\gamma
}v=a^{\ast }=\lim_{r\longrightarrow \infty }r^{\gamma }w,$ thus $%
\lim_{\left\vert x\right\vert \longrightarrow \infty }\left\vert
x\right\vert ^{\gamma }u=a^{\ast }$.\medskip

(iii) Let $u$ be a solution in $\mathbb{R}^{N}\backslash \left\{ 0\right\} $%
. Then $v\leq u\leq w\leq cv$ in $\mathbb{R}^{N}\backslash \left\{ 0\right\} 
$. From Theorem \ref{H1rad}, all the behaviours of radial solutions in $%
\mathbb{R}^{N}\backslash \left\{ 0\right\} $ are still described. Either%
\textbf{\ }$\mu >\mu _{0}$ and $r^{\gamma }v\equiv r^{\gamma }v=a^{\ast },$
thus $\left\vert x\right\vert ^{\gamma }u\equiv a^{\ast }$. Or $\mu >\mu
_{0} $, and $\lim_{r\longrightarrow 0}r^{S_{1}}v=k_{i}^{\prime }>0$, and $%
\lim_{r\longrightarrow \infty }r^{\gamma }v=a^{\ast }$; then $%
\lim_{x\longrightarrow 0}\left\vert x\right\vert ^{S_{1}}u=k_{i}>0$ from
Proposition \ref{exact}, and $\lim_{\left\vert x\right\vert \longrightarrow
\infty }\left\vert x\right\vert ^{\gamma }u=a^{\ast }$. For given $k_{1}>0,$
there exists a unique function $u$ satisfying these two conditions. Indeed
if $\widetilde{u}$ is another such solution, then $(1+\varepsilon )%
\widetilde{u}$ is a supersolution, and near $0$ and $\infty ,$ it is greater
than $u,$ then $(1+\varepsilon )\widetilde{u}\geq u;$ then $\widetilde{u}%
\geq u,$ and $\widetilde{u}=u.$ By uniqueness $u$ is radial. Finally if $\mu
=\mu _{0},$ then $\lim_{x\rightarrow 0}\left\vert x\right\vert ^{\frac{N-p}{p%
}}\left\vert \ln \left\vert x\right\vert \right\vert ^{-\frac{2}{p}%
}u(x)=\ell >0$ from Proposition \ref{excrit}, and $\lim_{\left\vert
x\right\vert \longrightarrow \infty }\left\vert x\right\vert ^{\gamma
}u(x)=a^{\ast }.$ By comparison as above, it is unique and radial.\medskip
\end{proof}

The proofs of Theorem \ref{H2nonrad}, \ref{H3nonrad}, \ref{H4nonrad} are
analogous.

\begin{remark}
Compared to the proofs of \cite{FrVe} in the case $\mu =0$ and of \cite{CiFa}
in the case $p=2$, our proofs are much shorter: we do not need any
discussion on the cases where $\lim \sup_{\left\vert x\right\vert
\longrightarrow 0}\left\vert x\right\vert ^{S_{1}}u=\infty $ or $\lim
\inf_{\left\vert x\right\vert \longrightarrow 0}u=0$; and we do not require
a comparison with radial subsolutions or supersolutions which existence is
difficult to obtain, see \cite[Lemma 1.4]{FrVe} and \cite[Propositions
3.1,3.4]{Ci}.
\end{remark}

\section{Proofs of the radial results\label{App}}

\subsection{Fixed points of system (\protect\ref{SGV})}

We consider the system (\ref{SGV}), which takes the form

\begin{equation}
\left\{ 
\begin{array}{ccc}
G_{t} & = & F(G)-V, \\ 
V_{t} & = & (q+1-p)V(\gamma -\left\vert G\right\vert ^{\frac{2-p}{p-1}}G),%
\end{array}%
\right.  \label{GVF}
\end{equation}%
where we recall that $V>0$, $G=\left\vert S\right\vert ^{p-2}S$, and 
\begin{equation}
F(G)=\varphi (S)=(p-1)\left\vert G\right\vert ^{\frac{p}{p-1}}-(N-p)G-\mu ,
\label{Ff}
\end{equation}%
so we study the system for 
\begin{equation*}
(G,V)\in \mathbb{R\times }\left[ 0,\infty \right) .
\end{equation*}%
\medskip

$\bullet $ The system has \textbf{three possible fixed points}: when $\mu
\geq \mu _{0}$ we find two first points 
\begin{equation*}
\mathbf{A}_{1}=(G_{1},0)=(\left\vert S_{1}\right\vert ^{p-2}S_{1},0),\qquad 
\mathbf{A}_{2}=(G_{2},0)=(\left\vert S_{2}\right\vert ^{p-2}S_{2},0),
\end{equation*}%
eventually confounded when $\mu =\mu _{0},$ and eventually a third point in $%
\mathbb{R\times }\left( 0,\infty \right) $, 
\begin{equation*}
\mathbf{M}_{0}=(G_{0},V_{0})=(\left\vert \gamma \right\vert ^{p-2}\gamma
,(p-1)\left\vert \gamma \right\vert ^{p}-(N-p)\left\vert \gamma \right\vert
^{p-2}\gamma -\mu )
\end{equation*}%
under the condition $F(G_{0})>0,$ corresponding to the solution $u^{\ast }$
given by (\ref{condi}); if $\gamma =0,$ $\mu <0,$ then $\mathbf{M}%
_{0}=(0,\left\vert \mu \right\vert )$.$\medskip $

$\bullet $ In the sequel we use the \textbf{vanishing curves of the vector
field} in the phase-plane $\mathbb{R\times }\left( 0,\infty \right) $ 
\begin{equation*}
\mathcal{C=}\left\{ G_{t}=0\right\} =\left\{ V=F(G)\right\} ,
\end{equation*}%
\begin{equation*}
\left\{ V_{t}=0\right\} =\left\{ V=0\right\} \cup \mathcal{L},\mathcal{%
\qquad L}=\left\{ G=G_{0}=\left\vert \gamma \right\vert ^{p-2}\gamma
\right\} .
\end{equation*}%
Obviously $\mathcal{C}$ and $\mathcal{L}$ meet at $\mathbf{M}_{0}$ when it
exists. And $\left\{ V=0\right\} $ and $\mathcal{L}$ meet at point $%
L_{0}=(G_{0},0)=(\left\vert \gamma \right\vert ^{p-2}\gamma ,0).$ Moreover
when $\mu \geq \mu _{0}$, then $\mathcal{C}=\mathcal{C}_{1}\cup \mathcal{C}%
_{2},$ where $\mathcal{C}_{1}$ is the graph of a nonincreasing function $%
F_{1}$ such that $F_{1}(G_{1})=0$ and $\mathcal{C}_{2}$ is the graph of an
nondecreasing function $F_{2}$ such that $F_{2}(G_{2})=0$. When $\mu >\mu
_{0}$, the slope of $F_{1}$ at $\mathbf{A}_{1}$ (resp. of $F_{2}$ at $G_{2}$%
) is 
\begin{equation*}
F_{1}^{\prime }(G_{1})=p\left\vert G_{1}\right\vert ^{\frac{p-2}{p-1}%
}G_{1}-(N-p)>0\text{ \quad }(\text{resp. }F_{2}^{\prime }(G_{2})=pG_{2}^{%
\frac{1}{p-1}}-(N-p)<0).
\end{equation*}%
When $\mu <\mu _{0}$ (resp. $\mu =\mu _{0})$ the function $F$ has a positive
(resp. zero) minimum at $(\frac{N-p}{p})^{p-1}$.$\medskip $

\begin{remark}
\label{Vzero}Consider the eventual trajectories $\mathcal{T}$ located on the
axis $\left\{ V=0\right\} ,$ which are \textbf{nonadmissible} for our
purpose. We claim that the union of their adherence covers the axis, so that 
\textbf{any trajectory with a point in }$\mathbb{R}\times (0,\infty )$%
\textbf{\ stays in it}. Indeed in case $V=0,$ the system reduces to equation 
\begin{equation*}
G_{t}=(p-1)\left\vert G\right\vert ^{\frac{p}{p-1}}-(N-p)G-\mu =F(G),
\end{equation*}
corresponding to the solutions of the Hardy equation $\mathcal{L}_{p,\mu
}(u)=0,$ see (\ref{equi}). Thus in any interval $\mathcal{I}$ where $%
F(G)\neq 0,$ we can express $t$ as a function of $G$: for fixed $\widetilde{G%
}$ $\in \mathcal{I}$, we write $t(G)=t(\widetilde{G})+\int_{\widetilde{G}%
}^{G}\frac{dg}{F(g)}$. If $\mu <\mu _{0},$ then $\mathcal{I=}\mathbb{R}$,
and the integral converges at $\pm \infty ,$ then $G$ describes $\mathbb{R}$
as $t$ describes $(t(-\infty ),t(\infty ))$. When $\mu \geq \mu _{0},$ the
integral still converges as $G\rightarrow \pm \infty ,$ and diverges at $%
G_{1},G_{2}$; taking $\mathcal{I}=(-\infty ,G_{2})$ (resp. $(G_{1},\infty )$%
) $G$ describes $\mathcal{I}$ as $t$ describes $(t(-\infty ),\infty )$
(resp. $(-\infty ,t(\infty )$); if $G_{1}\neq G_{2},$ taking $\mathcal{I=}%
(G_{1},G_{2})$, $G$ describes $\mathcal{I}$ as $t$ describes $\mathbb{R}$.
\end{remark}

Next we give a general property of the trajectories of system (\ref{SGV}).

\begin{lemma}
\label{conver}For any solution $u$ of (\ref{onep}) defined near $r=0$ (resp.
near $r=\infty ),$ then $(G,V)$ converge to one of the fixed points $\mathbf{%
M}_{0},\mathbf{A}_{1},\mathbf{A}_{2}$ as $t\rightarrow -\infty $ (resp. $%
t\rightarrow \infty $).\bigskip
\end{lemma}

\begin{proof}
From the Osserman's estimate (\ref{Oss}), for such a solution, $V\ $is
bounded near $t=-\infty $ (resp. near $t=\infty $). Suppose that $G$ is
unbounded near $t=-\infty $ (resp. near $t=\infty )$. Either $G$ is monotone
near $t=-\infty $ (resp. near $t=\infty )$ and $G$ tends to $\pm \infty ,$
then $G_{t}\sim (p-1)\left\vert G\right\vert ^{\frac{p}{p-1}};$ by
integration, we deduce that $G\ $tends to $0,$ which is contradictory. Or
there exists a monotone sequence $t_{n}\rightarrow -\infty $ (resp. $%
t_{n}\rightarrow \infty $) such that $G_{t}(t_{n})=0,$ and $\left\vert
G(t_{n})\right\vert \rightarrow \infty $ then $V(t_{n})=F(G(t_{n}))%
\rightarrow \infty ,$ which is still contradictory. Then $(V,G)$ is bounded.
From the Poincar\'{e}-Bendixon Theorem, either it converges to a fixed
point, or it has a limit cycle around a fixed point. If it is $\mathbf{A}%
_{1} $ or $\mathbf{A}_{2},$ it converges to the point. If it is $\mathbf{M}%
_{0}=(G_{0},V_{0}),$ then $F(G_{0})>0$. Then either $\mu >\mu _{0}$ and then 
$\gamma \neq \frac{N-p}{p},$ or $\mu <\mu _{0}$. Suppose that there exists a
periodic trajectory $\mathcal{T}$ around $\mathbf{M}_{0}$. It is impossible
when $\gamma \neq \frac{N-p}{p}$: indeed the energy function $\mathcal{E}$
defined at (\ref{derigv}) satisfies%
\begin{equation*}
\mathcal{E}_{t}=-DV^{\frac{p}{p-1}}(\gamma -\left\vert G\right\vert ^{\frac{%
2-p}{p-1}}G)(\left\vert \gamma \right\vert ^{p-2}\gamma -G)
\end{equation*}%
with $D=N-p-p\gamma \neq 0$ , then $D\mathcal{E}_{t}<0,$ up to a finite
number of points where $G=\left\vert \gamma \right\vert ^{p-2}\gamma $. Then 
$(V,G)$ converges to $\mathbf{M}_{0}$. Next assume $\mu <\mu _{0}$ and $%
\gamma =\frac{N-p}{p}$. Consider a solution $t\in \mathbb{R}\longmapsto
(G(t),V(t))$ with period $\mathcal{P}$ on the trajectory $\mathcal{T}$.
There exist at least a value $t_{1}\in \left[ 0,\mathcal{P}\right] $ where $%
G(t_{1})=G_{0}$ and $V(t_{1})>V_{0}$; then at $P(t_{1})=(G_{0},V(t_{1}))$, $%
\mathcal{T}$ enters the region $\left\{ V>F(g),G<G_{0}\right\} $; this
region is positively invariant since the field on the curve $\mathcal{C}_{2}$
is directed by $(0,1)$ and the field on $\mathcal{L}$ by $(-1,0)$ (see
Figure 8); so $\mathcal{T}$ stays in it for $t>t_{1}$, and cannot turn
around $\mathbf{M}_{0}$, and we still get a contradiction.
\end{proof}

\subsubsection{Behaviour near the fixed point $\mathbf{M}_{0}$}

Next we precise the nature of the point $\mathbf{M}_{0}$. In that case some
new phenomena appear in the particular case $\theta +p=0$, equivalently $%
\gamma =0$.

\begin{lemma}
\label{fix} Suppose that $\mathbf{M}_{0}$ exists, and $\gamma \neq 0$. Then $%
\mathbf{M}_{0}$ is a saddle point: the eigenvalues $\lambda _{1}<0<\lambda
_{2}$ are the roots of the trinom 
\begin{equation*}
T(\lambda )=\lambda ^{2}-(p\gamma -N+p)\lambda -\frac{q+1-p}{p-1}\left\vert
\gamma \right\vert ^{2-p}V_{0}.
\end{equation*}%
Some corresponding eigenvectors are 
\begin{equation*}
\overrightarrow{w_{1}}=(1,p\gamma -N+p-\lambda _{1})=(1,\lambda _{2}),\qquad 
\overrightarrow{w_{2}}=(1,p\gamma -N+p-\lambda _{2})=(1,\lambda _{1}).
\end{equation*}%
There exist precisely two trajectories $\mathcal{T}_{1},\mathcal{T}_{2}$%
,converging to $\mathbf{M}_{0}$ as $t\rightarrow \infty $, directed by $%
\overrightarrow{w_{1}}$, with slope $\lambda _{2}:$ a trajectory $\mathcal{T}%
_{1}$ with $G<G_{0}$ (resp. $\mathcal{T}_{2}$ with $G>G_{0}$) as $%
t\rightarrow \infty ;$ and two trajectories $\mathcal{T}_{3},\mathcal{T}_{4}$%
,converging to $\mathbf{M}_{0}$ as $t\rightarrow -\infty $, directed by $%
\overrightarrow{w_{2}}$, with slope $\lambda _{1}:$a trajectory $\mathcal{T}%
_{3}$ such that $G<G_{0}$ (resp. $\mathcal{T}_{4}$ with $G>G_{0}$) as $%
t\rightarrow -\infty $.
\end{lemma}

\begin{proof}
We set $G=(G_{0}+\overline{G},V_{0}+\overline{V})$, then the linearized
system at $\mathbf{M}_{0}$ is given by 
\begin{equation}
\left\{ 
\begin{array}{ccc}
\overline{G}_{t} & = & (p\gamma -N+p)\overline{G}-\overline{V}, \\ 
\overline{V}_{t} & = & -\frac{q+1-p}{p-1}\left\vert \gamma \right\vert
^{2-p}V_{0}\overline{G},%
\end{array}%
\right.  \label{linea}
\end{equation}%
so the eigenvalues, given by 
\begin{equation*}
\det (%
\begin{array}{cc}
p\gamma -N+p-\lambda & -1 \\ 
-\frac{q+1-p}{p-1}\left\vert \gamma \right\vert ^{2-p}V_{0} & -\lambda%
\end{array}%
)=0,
\end{equation*}%
are the roots of the trinom $T(\lambda );$ and the product of the roots is
negative, thus $\mathbf{M}_{0}$ is a saddle point. The eigenvectors can be
computed easily.\bigskip
\end{proof}

Next we study the particular case $\gamma =0$, equivalently $\theta +p=0$.
Equation (\ref{pq}) takes the form (\ref{pqzero}), and the solution $u^{\ast
}$ exists when $\mu <0$, and it is constant: $u^{\ast }\equiv \left\vert \mu
\right\vert ^{\frac{1}{q+1-p}}$. Here the study depends on $p$, in
particular we find the existence of \textbf{locally constant solutions near }%
$0$\textbf{\ or near }$\infty $\textbf{\ when }$p>2$:

\begin{lemma}
\label{gammazero}Assume $\gamma =0$. When $\mu <0$, $\mathbf{M}%
_{0}=(0,\left\vert \mu \right\vert )$ is well defined.\medskip

\noindent (i) For $p=2$, Lemma \ref{fix} still applies without change.
Denoting by $\lambda _{1}<0<\lambda _{2}$ the roots of $\lambda
^{2}+(N-2)\lambda -(q-1)\left\vert \mu \right\vert =0$, the solutions
corresponding to trajectories $\mathcal{T}_{3}$ and $\mathcal{T}_{4}$
satisfy 
\begin{equation*}
u-\left\vert \mu \right\vert ^{\frac{1}{q-1}}\sim _{r\longrightarrow
0}C_{i}r^{\lambda _{2}},\quad i=3,4,\quad C_{3}>0>C_{4},
\end{equation*}%
those corresponding to trajectories $\mathcal{T}_{1}$ and $\mathcal{T}_{2}$
satisfy 
\begin{equation*}
u-\left\vert \mu \right\vert ^{\frac{1}{q-1}}\sim _{r\longrightarrow \infty
}C_{i}r^{\lambda _{1}},\quad i=1,2,\quad C_{2}>0>C_{1}.
\end{equation*}

\noindent (ii) For $p\neq 2$ there exist at least a trajectory $\mathcal{T}%
_{1}$ with $G<0$ (resp. trajectory $\mathcal{T}_{2}$ with $G>0)$ converging
to $\mathbf{M}_{0}$ as $t\rightarrow \infty ;$ and a trajectory$\mathcal{T}%
_{3}$ such that $G<0$ (resp. a trajectory $\mathcal{T}_{4}$ with $G>0)$
converging to $\mathbf{M}_{0}$ as $t\rightarrow -\infty $.

$\bullet $ For $p<2$, the solutions corresponding to trajectories of types $%
\mathcal{T}_{3}$ and $\mathcal{T}_{4}$ satisfy 
\begin{equation*}
u-\left\vert \mu \right\vert ^{\frac{1}{q+1-p}}\sim _{r\longrightarrow
0}C_{i}\left\vert \ln r\right\vert ^{-\frac{p-1}{2-p},}\quad i=3,4,\quad
C_{3}>0>C_{4};
\end{equation*}%
the solutions corresponding to trajectories of types $\mathcal{T}_{1}$ and $%
\mathcal{T}_{2}$ satisfy 
\begin{equation*}
u-\left\vert \mu \right\vert ^{\frac{1}{q+1-p}}\sim _{r\longrightarrow
\infty }C_{i}r^{-\frac{N-p}{p-1,}}\quad i=1,2,\quad C_{2}>0>C_{1}.
\end{equation*}

$\bullet $ For $p>2$, the solutions $u$ corresponding to trajectories of
types $\mathcal{T}_{3}$ and $\mathcal{T}_{4}$ are constant near $r=0$, there
holds $\rho >0$ such that%
\begin{equation*}
u(r)-\left\vert \mu \right\vert ^{\frac{1}{q+1-p}}\sim _{r\longrightarrow
\rho }C_{i}((r-\rho )^{+})^{\frac{p}{p-2}},\quad i=3,4,\quad C_{3}>0>C_{4}.
\end{equation*}%
The solutions $u$ corresponding to trajectories $\mathcal{T}_{1}$ and $%
\mathcal{T}_{2}$ are constant near $r=\infty $, and there holds $R>0$ such
that 
\begin{equation*}
u(r)-\left\vert \mu \right\vert ^{\frac{1}{q+1-p}}=C_{i}((R-r)^{+})^{\frac{p%
}{p-2}},\quad i=1,2,\quad C_{2}>0>C_{1}.
\end{equation*}
\end{lemma}

\begin{proof}
We are in the case ($\mathcal{H}_{2}$) with\textbf{\ }$\gamma =0<S_{2}<S_{1}$%
, $\mu _{0}<\mu <0$, or in the case \textbf{(}$\mathcal{H}_{5}$\textbf{) }$%
\mu <\mu _{0}$. We consider the regions in $\mathbb{R}\times (0,\infty )$%
\begin{eqnarray*}
\mathcal{J}_{1} &=&\left\{ 0<V<F(G),G<0\right\} ,\qquad \mathcal{J}%
_{2}=\left\{ V>F(G),G>0\right\} , \\
\mathcal{J}_{3} &=&\left\{ V>F(G),G<0\right\} ,\qquad \mathcal{J}%
_{4}=\left\{ 0<V<F(G),G>0\right\} .
\end{eqnarray*}

(i) Case $p=2$. The former linearization analysis still applies: the
eigenvalues are the roots $\lambda _{1}<0<\lambda _{2}$ of 
\begin{equation*}
T(\lambda )=\lambda ^{2}+(N-2)\lambda -(q-1)\left\vert \mu \right\vert
\end{equation*}%
so $\mathbf{M}_{0}$ is still a saddle point as above, defining four
trajectories $\mathcal{T}_{1}$,$\mathcal{T}_{2}$,$\mathcal{T}_{3}$,$\mathcal{%
T}_{4}$, respectively in regions $\mathcal{J}_{1},\mathcal{J}_{2},\mathcal{J}%
_{3},\mathcal{J}_{4}$. System (\ref{linea}) becomes%
\begin{equation*}
\left\{ 
\begin{array}{ccc}
\overline{G}_{t} & = & (p\gamma -N+p)\overline{G}-\overline{V}, \\ 
\overline{V}_{t} & = & -(q-1)\left\vert \mu \right\vert \overline{G}.%
\end{array}%
\right.
\end{equation*}%
By reduction to the diagonal form we deduce that $\lim_{t\longrightarrow
-\infty }\overline{V}$ $e^{\lambda _{2}t}=l\neq 0$ on $\mathcal{T}_{3},%
\mathcal{T}_{4}$ (resp. $\lim_{t\longrightarrow \infty }\overline{V}$ $%
e^{\lambda _{1}t}=\lambda \neq 0$ on $\mathcal{T}_{1},\mathcal{T}_{2}$).
Since $U(t)=(\left\vert \mu \right\vert +\overline{V})^{\frac{1}{q-1}%
}=(\left\vert \mu \right\vert +\overline{V})^{\frac{1}{q-1}}$, there holds $%
U(t)-\left\vert \mu \right\vert ^{\frac{1}{q-1}}\sim _{t\longrightarrow
-\infty }\frac{1}{q-1}\left\vert \mu \right\vert ^{\frac{2-q}{q-1}}\overline{%
V}\sim _{t\longrightarrow -\infty }C_{i}e^{\lambda _{2}t}$ on $\mathcal{T}%
_{3},\mathcal{T}_{4}$ (resp. $U(t)-\left\vert \mu \right\vert ^{\frac{1}{q-1}%
}\sim _{t\longrightarrow \infty }C_{i}e^{\lambda _{1}t}$ on $\mathcal{T}_{1}$%
,$\mathcal{T}_{2}$).\bigskip

(ii) Cases $p\neq 2$.\medskip

$\bullet $ Case $p<2$. Then the linearization is still valid, with $%
T(\lambda )=\lambda ^{2}+(N-p)\lambda $, that means $\lambda
_{1}=-(N-p)<0=\lambda _{2}$. Relatively to the eigenvalue $\lambda _{1}$,
there is an eigenvector $(1,0)$, there exists precisely two trajectories
with slope $0$, $\mathcal{T}_{2}$ in the region $\mathcal{J}_{2}$ and $%
\mathcal{T}_{1}$ in the region $\mathcal{J}_{1}$, ending at the point as $%
t\rightarrow \infty $. Relatively to the eigenvalue $0$, the central
manifold, of dimension $1$, is directed by $(1,-(N-p))$, we find at least
two trajectories on it, converging to $\mathbf{M}_{0}$ as $t\rightarrow
\infty $ or $t\rightarrow \infty $. Since $\lim_{t\longrightarrow \infty }%
\frac{V_{t}}{G_{t}}=-(N-p)$, then $(V,G)$ is necessarily in one of the
regions $\mathcal{J}_{3},\mathcal{J}_{4}$ ; as a consequence, the
convergence holds as $t\rightarrow -\infty $, so we find two trajectories $%
\mathcal{T}_{3},\mathcal{T}_{4}$. Moreover, consider for example $\mathcal{T}%
_{4};$ we know that $\lim_{t\longrightarrow -\infty }\frac{\overline{V}}{G}%
=-(N-p)$, then 
\begin{equation*}
\overline{V}_{t}=-(q+1-p)(\left\vert \mu \right\vert +\overline{V}%
)\left\vert G\right\vert ^{\frac{2-p}{p-1}}G)\sim _{t\longrightarrow -\infty
}-k\left\vert G\right\vert ^{\frac{2-p}{p-1}}G\sim k(N-p)^{-\frac{1}{p-1}}%
\overline{V}^{\frac{1}{p-1}},
\end{equation*}%
thus by integration, $\overline{V}\sim _{t\longrightarrow -\infty
}c\left\vert t\right\vert ^{-\frac{p-1}{2-p}}$, hence $u^{q+1-p}-\left\vert
\mu \right\vert \sim _{t\longrightarrow -\infty }c_{1}\left\vert \ln
r\right\vert ^{-\frac{p-1}{2-p}}$, so that $u-\left\vert \mu \right\vert ^{%
\frac{1}{q+1-p}}\sim _{r\longrightarrow 0}C_{4}\left\vert \ln r\right\vert
^{-\frac{p-1}{2-p}}$. We obtain a similar result for the trajectory $%
\mathcal{T}_{3}$. Next consider the trajectory $\mathcal{T}_{2};$ there
holds $G_{t}<0$,$V_{t}<0$ and $\lim_{t\longrightarrow \infty }\frac{%
\overline{V}}{G}=0$, then $G_{t}\sim _{t\longrightarrow \infty }-(N-p)G$,
thus $G=O(e^{(-(N-p)+\varepsilon )t})t$. Moreover, we still have $\overline{V%
}_{t}\sim _{t\longrightarrow -\infty }-k\left\vert G\right\vert ^{\frac{2-p}{%
p-1}}G$, then $\overline{V}_{t}=O(e^{\frac{-(N-p)+\varepsilon }{p-1}t})$,
thus $\overline{V}=O(e^{\frac{-(N-p)+\varepsilon }{p-1}t})$, and $\left\vert
G\right\vert ^{\frac{p}{p-1}}=O(e^{\frac{-(N-p)+\varepsilon }{p-1}pt})$;
then $G_{t}=-(N-p)G+O(e^{\frac{-(N-p)+\varepsilon }{p-1}t})$, which implies
that $\lim_{t\longrightarrow \infty }e^{(N-p)t}G=C>0;$ since $%
\lim_{t\longrightarrow \infty }U=\left\vert \mu \right\vert ^{\frac{1}{q+1-p}%
}$, we get $U_{t}\sim _{t\longrightarrow \infty }C^{\frac{1}{p-1}}e^{-\frac{%
(N-p)t}{p-1}}$, and $u-\left\vert \mu \right\vert ^{\frac{1}{q+1-p}}\sim
_{r\longrightarrow \infty }C_{2}r^{-\frac{N-p}{p-1}}$, where $C_{2}>0$.
Similarly for $\mathcal{T}_{1}$.$\medskip $

$\bullet $ Case $p>2$. Here we cannot linearize the system. We first
consider the region $\mathcal{J}_{4}$, which is positively invariant: all
the trajectories with one point in $\mathcal{J}_{4}$ converge to $\mathbf{A}%
_{2}$ as $t\rightarrow \infty $. Let $\mathcal{U}$ (resp. $\mathcal{V}$) be
the set of points $P$ of $\mathcal{J}_{4}$ such that the trajectory passing
by $P$ cuts the curve $\mathcal{C}_{2}$ (resp. the line $\mathcal{L}$ );
then $\mathcal{U}$ (resp. $\mathcal{V}$) is an open set in $\mathcal{J}_{4}$%
, since the intersections are transverse. Then $\mathcal{U}$ $\cup $ $%
\mathcal{V\neq J}_{4}$. Thus there exists at least a trajectory $\mathcal{T}%
_{4}$ in $\mathcal{J}_{4}$ converging to $\mathbf{M}_{0}$ as $t\rightarrow
\infty $. In the same way the region $\mathcal{J}_{3}$ is positivement
invariant, and there exists at least a trajectory $\mathcal{T}_{3}$ in $%
\mathcal{J}_{3}$ converging to $\mathbf{M}_{0}$ as $t\rightarrow -\infty $.
The region $\mathcal{J}_{1}$ (resp. $\mathcal{J}_{2}$) is negatively
invariant and as before it contains at least a trajectory $\mathcal{T}_{1}$
(resp. $\mathcal{T}_{2}$) converging to $\mathbf{M}_{0}$ as $t\rightarrow
\infty $.

Setting again $V=\left\vert \mu \right\vert +\overline{V}$, there holds 
\begin{equation*}
\left\{ 
\begin{array}{ccc}
G_{t} & = & (p-1)\left\vert G\right\vert ^{\frac{p}{p-1}}-(N-p)G-\overline{V}%
, \\ 
\overline{V}_{t} & = & -(q+1-p)(\left\vert \mu \right\vert +\overline{V}%
)\left\vert G\right\vert ^{\frac{2-p}{p-1}}G).%
\end{array}%
\right.
\end{equation*}%
By derivation, setting $k=(q+1-p)\left\vert \mu \right\vert $ 
\begin{equation*}
G_{tt}=(p\left\vert G\right\vert ^{\frac{2-p}{p-1}}G-(N-p))G_{t}-\overline{V}%
_{t},
\end{equation*}%
\begin{equation*}
G_{tt}+(N-p))G_{t}=\left\vert G\right\vert ^{\frac{2-p}{p-1}%
}G(pG_{t}+(q+1-p)(\left\vert \mu \right\vert +\overline{V}));
\end{equation*}%
near the fixed point, as $t\rightarrow \pm \infty $, $G$ and $\overline{V}$,
then also $G_{t}$ tends to $0$, then 
\begin{equation*}
G_{tt}+(N-p)G_{t}=k\left\vert G\right\vert ^{\frac{2-p}{p-1}%
}G(1+o(1)),\qquad k=(q+1-p)\left\vert \mu \right\vert .
\end{equation*}

First consider the trajectory $\mathcal{T}_{4}$ where $G>0$, $G_{t}>0$:%
\begin{equation*}
\frac{k}{2}G^{\frac{1}{p-1}}\leq G_{tt}+(N-p)G_{t}\leq 2kG^{\frac{1}{p-1}};
\end{equation*}%
by multiplication by $G_{t}$ we deduce that 
\begin{equation*}
(\frac{G_{t}^{2}}{2}-2k\frac{p-1}{p}G^{\frac{p}{p-1}})_{t}\leq 0.
\end{equation*}%
Then $\frac{G_{t}^{2}}{2}-2k\frac{p-1}{p}G^{\frac{p}{p-1}}$ is
nonincreasing, and tends to $0$ at $-\infty $, then $\frac{G_{t}^{2}}{2}\leq
2k\frac{p-1}{p}G^{\frac{p}{p-1}}$, thus $G^{-\frac{p}{2(p-1)}}G_{t}\leq c$,
hence $\frac{2(p-1)}{p-2}G^{\frac{p-2}{2(p-1)}}-ct$ is nonincreasing; this
is impossible if $t\rightarrow -\infty $. Denoting by $(T,\tau )$ the
maximal interval where the solution is positive, $T$ is finite, and as $%
t\rightarrow T$, then for $p>2$, $G^{\frac{p-2}{2(p-1)}}-c(t-T)\leq 0$. We
obtain a solution $u$ such that $u_{r}$ is zero for $r\leq e^{T}$, that
means $u$ is constant near the origin. Moreover $G_{t}\leq cG^{\frac{p}{%
2(p-1)}}=o(G^{\frac{1}{p-1}})$, then $G_{tt}\sim _{t\longrightarrow T}kG^{%
\frac{1}{p-1}}$, then for any $\varepsilon >0$, $(k-\varepsilon )G^{\frac{1}{%
p-1}}\leq G_{tt}\leq (k+\varepsilon )G^{\frac{1}{p-1}}$ for $t-T\leq
t_{\varepsilon }$, and by new integrations we get consecutively $G\sim
_{t\longrightarrow T}c_{k}(t-T)^{\frac{2(p-1)}{p-2}}$, 
\begin{equation*}
\frac{U_{t}}{U}\sim _{t\longrightarrow T}\left\vert \mu \right\vert ^{-\frac{%
1}{q+1-p}}U_{t}\sim _{t\longrightarrow T}-c_{k}(t-T)^{\frac{2}{p-2}},
\end{equation*}%
\begin{equation*}
U-\left\vert \mu \right\vert ^{\frac{1}{q+1-p}}\sim _{t\longrightarrow T}%
\frac{p-2}{p}c_{k}(t-T)^{\frac{p}{p-2}}\sim _{t\longrightarrow
T}C_{4}(r-e^{T})^{\frac{p}{p-2}},
\end{equation*}%
where $C_{4}<0$.. We get a similar result for the trajectory $\mathcal{T}%
_{3}:$ in this case $G<0$, $G_{t}<0$, we get the same conclusion by
considering $-G$.

Next consider $\mathcal{T}_{2}$, where $G>0$ and $G_{t}<0$. We find 
\begin{equation*}
2kG^{\frac{1}{p-1}}G_{t}\leq G_{t}G_{tt}+(N-p)G_{t}^{2}\leq \frac{k}{2}G^{%
\frac{1}{p-1}}G_{t};
\end{equation*}%
hence $\frac{G_{t}^{2}}{2}-\frac{k}{2}\frac{p-1}{p}G^{\frac{p}{p-1}}$is
nonincreasing and tends to $0$ as $t\rightarrow \infty $; thus $\frac{%
G_{t}^{2}}{2}\geq \frac{k}{2}\frac{p-1}{p}G^{\frac{p}{p-1}}$, consequently $%
G^{-\frac{p}{2(p-1)}}G_{t}\leq -c$, then $G^{\frac{p-2}{2(p-1)}}+ct$ is
nonincreasing, hence again we get a solution $u$ such that $u_{r}$ is zero
for $r\geq e^{T}$, that means $u$ is constant for large $r$. Then $%
u(r)-\left\vert \mu \right\vert ^{\frac{1}{q+1-p}}=C_{2}((R-r)^{+})^{\frac{p%
}{p-2}}$, with $C_{2}>0$. The result is similar for $\mathcal{T}_{1}$.
\end{proof}

\begin{remark}
\label{singu} Here we make a comment on the regularity of system (\ref{SGV}%
). Let us write it under the form%
\begin{equation*}
\left\{ 
\begin{array}{ccc}
G_{t} & = & f(G(t),V(t)), \\ 
V_{t} & = & g(G(t),V(t)).%
\end{array}%
\right.
\end{equation*}%
If $p\leq 2$, the functions $f$,$g$ are of class $C^{1}$, so the system has
no singular point. If $p>2$, the system is singular at $G=0$, that means $%
u^{\prime }=0$, since $g$ is only continuous at $G=0$. However, for any
given $\widetilde{V}>0$, $\widetilde{V}\neq -\mu $, $\gamma \neq 0$, the
point $(0,\widetilde{V})$ is not a fixed point, even if $\gamma =0$; then
there is only one trajectory passing by the point $(0,\widetilde{V})$.
Indeed consider the Cauchy conditions $G(t_{0})=0,V(t_{0})=\widetilde{V}$.
There holds $G_{t}(t_{0})=-\mu -\widetilde{V}\neq 0$. Then one can define $t$
as a $C^{1}$ function $t=\psi (G)$ near $t_{0}$, and (\ref{SGV}) is
equivalent to%
\begin{equation*}
\frac{d}{dG}(\psi ,V)=\mathcal{F}(G,V)=(\frac{1}{f(G,V)},\psi (G)g(G,V))
\end{equation*}%
where $\mathcal{F}$ is of class $C^{1}$ with respect to $V$, and continuous
with respect to $G$, so the Cauchy Theorem can be applied.\medskip

In conclusion \textbf{the unique singular point is }$(0,\left\vert \mu
\right\vert )$\textbf{\ when }$\mu <0$\textbf{\ and }$p>2$.\medskip

$\bullet $ When $\gamma =0$, it is the fixed point $\mathbf{M}_{0}$, studied
at Lemma \ref{gammazero}.\medskip

$\bullet $ When $\gamma \neq 0$, $p>2$, then $(0,\left\vert \mu \right\vert
) $ is not a fixed point. Consider a trajectory $\mathcal{T}_{\mu }$ passing
by this point, and a solution $t\longmapsto U(t)$ passing by this point at
time $t_{0}$ (there exist at least one, from the Peano theorem). In that
case $G(t_{0})=G_{t}(t_{0})=0$, but $f$ is of class $C^{1}$, then $%
G_{tt}(t_{0})$ exists so $G$ is $C^{2}$ and $G_{tt}(t_{0})=-V_{t}(t_{0})=-%
\gamma (q+1-p)\left\vert \mu \right\vert \neq 0$. Then $G(t)\sim
_{t\longrightarrow t_{0}}-c\gamma (t-t_{0})^{2}$, with $c=\frac{%
(q+1-p)\left\vert \mu \right\vert }{2}$ thus $G$\text{\text{ has the sign of 
}}$-\gamma ;$ and $\mathcal{T}_{\mu }$ is tangent to the axis $\left\{
G=0\right\} $ at this point We do not know if $\mathcal{T}_{\mu }$ is unique.
\end{remark}

\subsubsection{Behaviour near the fixed points $\mathbf{A}_{1}$,$\mathbf{A}%
_{2}$}

Next we study the fixed points $A_{i}:$

\begin{lemma}
\label{fixai}Suppose that $\mu \geq \mu _{0}$. Then the points $%
A_{i}=(G_{i},0)$ are well defined (confounded when $\mu =\mu _{0}$), and the
associated eigenvalues are 
\begin{equation}
\rho _{i}=pS_{i}-N+p,\qquad \eta _{i}=(q+1-p)(\gamma -S_{i}).  \label{val}
\end{equation}

As a consequence, when $\mu >\mu _{0}$, if $\gamma >S_{1}$, then $\mathbf{A}%
_{1}$ is a source, and $\mathbf{A}_{2}$ is a saddle point; if $\gamma <S_{2}$%
, $\mathbf{A}_{1}$ is a saddle point and $\mathbf{A}_{2}$ is a sink. If $%
S_{2}<\gamma <S_{1}$, $\mathbf{A}_{1}$ and $\mathbf{A}_{2}$ are saddle
points. In those cases, corresponding eigenvectors to $\rho _{i}$,$\eta _{i}$
are 
\begin{equation}
\overrightarrow{u_{i}}=(1,0),\qquad \overrightarrow{v_{i}}%
=(1,pS_{i}-N+p-\eta _{i}),  \label{vel}
\end{equation}%
and the slope of $\overrightarrow{v_{i}}$ is $m_{i}=pS_{i}-N+p-\eta
_{i}=(q+1)S_{i}-(N+\theta )$.
\end{lemma}

\begin{proof}
Setting $G=G_{i}+\overline{G}$, $(p-1)\left\vert G_{i}+\overline{G}%
\right\vert ^{\frac{p}{p-1}}\simeq (p-1)\left\vert G_{i}\right\vert ^{\frac{p%
}{p-1}}+p\left\vert G_{i}\right\vert ^{\frac{2-p}{p-1}}G_{i}\overline{G}$
the linearized system at $A_{i}$ (here available for any $p>1$, and any $\mu
\geq \mu _{0}$) is%
\begin{equation}
\left\{ 
\begin{array}{ccccc}
G_{t} & = & ((p\left\vert G_{i}\right\vert ^{\frac{2-p}{p-1}}G_{i}-(N-p))%
\overline{G}-V & = & (pS_{i}-N+p)\overline{G}-V, \\ 
V_{t} & = & (q+1-p)(\gamma -\left\vert G_{i}\right\vert ^{\frac{2-p}{p-1}%
}G_{i})V & = & (q+1-p)(\gamma -S_{i})V,%
\end{array}%
\right.  \label{line}
\end{equation}%
and the eigenvalues $\lambda $ are given by 
\begin{equation*}
\det \left( 
\begin{array}{cc}
pS_{i}-N+p-\lambda & -1 \\ 
0 & (q+1-p)(\gamma -S_{i})-\lambda%
\end{array}%
\right) =0,
\end{equation*}%
so we get (\ref{val}) and (\ref{vel}).
\end{proof}

\begin{remark}
A point of discussion, first encountered when $\mu =0$ in \cite{BiGa} is the
repartition of the trajectories in case of a source or a sink, that means
two eigenvalues have the same sign. First suppose that $\rho _{1}>0$ and $%
\eta _{1}>0$. We get 
\begin{equation*}
\rho _{1}-\eta _{1}=pS_{1}-N+p-(q+1-p)(\gamma -S_{1})=(q+1)S_{1}-(N+\theta
)=m_{1}.
\end{equation*}%
If $\eta _{1}>\rho _{1}>0$ and $m_{1}$ $\neq 0$ we know that there exists
only one trajectory tangent to the vector $\overrightarrow{v_{1}}$, with
negative slope $m_{1}$, and all the other are tangent to $\overrightarrow{%
u_{1}}=(1,0)$, that means to the axis $\left\{ V=0\right\} $. On the
contrary, if $\rho _{1}>\eta _{1}>0$, there is only one trajectory tangent
to $\overrightarrow{u_{1}}$, and it is not admissible. It means that all the
trajectories are tangent to $\overrightarrow{v_{1}}$, with positive slope.
It is the case when $q>\widetilde{\mathbf{q}_{1}}$, where 
\begin{equation*}
\widetilde{\mathbf{q}_{1}}=\frac{N+\theta }{S_{1}}-1=\frac{\theta +p}{S_{1}}+%
\frac{N-p}{S_{1}}-1=\mathbf{q}_{1}-p+\frac{N-p}{S_{1}}<\mathbf{q}_{1}.
\end{equation*}%
\bigskip In the same way, consider the case where $\rho _{2}<0$ and $\eta
_{2}<0$. Then $\rho _{2}-\eta _{2}=m_{2}$. If $\eta _{1}>\rho _{1}>0$, there
exists only one trajectory tangent to the vector $\overrightarrow{v_{1}}$,
with negative slope, and all the other are tangent to $\left\{ V=0\right\} $
On the contrary, if $\rho _{1}>\eta _{1}>0$, there is only one trajectory
tangent to $\left\{ V=0\right\} $, and it is not admissible. It means that
all the trajectories are tangent to $\overrightarrow{v_{1}}$, with positive
slope. It is the case when $(q+1)S_{1}>N+\theta $. If $\mu <0$ it is
equivalent to $q>\widetilde{\mathbf{q}_{2}}$, where%
\begin{equation*}
\widetilde{\mathbf{q}_{2}}=\mathbf{q}_{2}-p+\frac{N-p}{S_{2}}>\mathbf{q}_{2}.
\end{equation*}%
If $\mu >0$ it is only possible if $N+\theta <0$, and $q<\widetilde{\mathbf{q%
}_{2}}$.
\end{remark}

\subsection{Proofs and comments}

We first consider the case ($\mathcal{H}_{1}$). In that case $\gamma \geq
S_{1}>0$ so there always hold $\theta +p>0$.$\medskip $\pagebreak

\begin{figure}[!h]
\begin{center}
 \includegraphics[keepaspectratio, width=11cm]{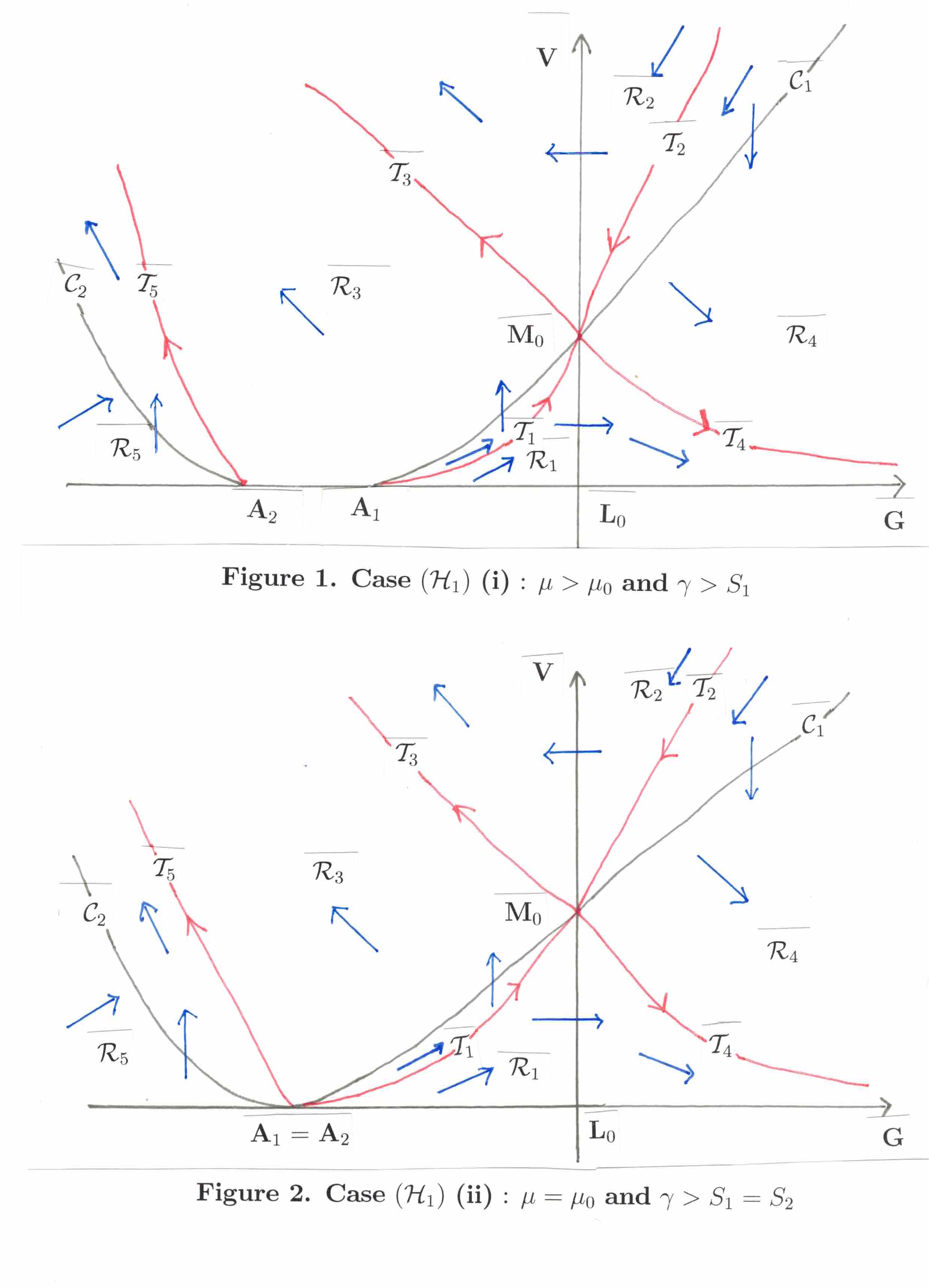}
 \end{center}
 \end{figure}

\begin{proof}[Proof of Theorem \protect\ref{H1rad}]
\textbf{Case (}$\mathcal{H}_{1}$\textbf{): }$\mu \geq \mu _{0}$\textbf{\ and 
}$\gamma \geq S_{1}.$ In that case the point $\mathbf{M}_{0}$ exists, with
the four trajectories $\mathcal{T}_{i},i=1,..,4$ associated.$\medskip $

(i) We first assume that $\mu >\mu _{0}$. We define the regions (see Figure
1) 
\begin{eqnarray*}
\mathcal{R}_{1} &\mathcal{=}&\left\{ V<F(G),G_{1}<G<G_{0}\right\} ,\qquad 
\mathcal{R}_{2}\mathcal{=}\left\{ V>F(G),G>G_{0}\right\} , \\
\mathcal{R}_{3} &=&\left\{ V>F(G),G<G_{0}\right\} ,\quad \mathcal{R}%
_{4}=\left\{ V<F(G),G>G_{0}\right\} ,\quad \mathcal{R}_{5}\mathcal{=}\left\{
V<F(G),G<G_{2}\right\} .
\end{eqnarray*}

$\bullet $ The region $\mathcal{R}_{1}$ is negatively invariant. Moreover
the slope of $F_{1}$ at $\mathbf{M}_{0}$ is 
\begin{equation*}
m=pG_{0}^{\frac{1}{p-1}}-N+p=p\gamma -N+p=\lambda _{1}+\lambda _{2}<\lambda
_{2},
\end{equation*}%
then $\mathcal{T}_{1}$ lies in $\mathcal{R}_{1}$ as $t\rightarrow \infty $,
then it stays in it, and necessarily converges to $\mathbf{A}_{1}$ (even in
the case $\rho _{1}=\eta _{1}$). Therefore there is a unique trajectory,
namely $\mathcal{T}_{1}$, joining $\mathbf{A}_{1}$ to $\mathbf{M}_{0}$, then 
$u$ satisfies (\ref{aa}).$\medskip $

$\bullet $ The trajectory $\mathcal{T}_{3}$ starts from $\mathbf{M}_{0}$ in
the region $\mathcal{R}_{3};$ and $\mathcal{R}_{3}$ is positively invariant:
indeed the vector field is entering $\mathcal{R}_{3}$ at any point or the
curve $\mathcal{C}_{2}$, except possibly at the singular point $%
(0,\left\vert \mu \right\vert )$ when $\mu <0$ and $p>2$. But no trajectory
with a point in $\mathcal{R}_{3}$ can pass by this point, since such
trajectories satisfy $G<0$, from then $\mathcal{T}_{3}$ stays in $\overline{%
\mathcal{R}_{3}}$. Consider any corresponding solution $u(r)=U(t)$ defined
on an maximal interval $(-\infty ,T)$. Suppose that $V$ is bounded$;$ then $%
F(G)$ is bounded, so $\left\vert G\right\vert $ is bounded, then $(G,V)$
converges to some $(\ell ,\Lambda )$, then $(G_{t},V_{t})$ has a finite
limit. It cannot happen that $(\ell ,\Lambda )=(0,-\mu )$, because we have
seen that the trajectories passing by this point are contained in the set $%
G<0$. Then $T=\infty $, thus $(G,V)$ converges to a fixed point; it is
impossible, since there is no other fixed point in $\overline{\mathcal{R}_{3}%
}$. Then $T$ is finite, and $\lim_{t\longrightarrow T}V=\infty ;$ if $%
\lambda =\lim_{t\longrightarrow T}G$ is finite, and $\lim_{t\longrightarrow
T}\frac{V_{t}}{V}=c>0$, which contradicts $\lim_{t=T}V=\infty $. Then $%
\lim_{t\longrightarrow T}G=-\infty $. It implies that the function $u$, such
that $u\sim _{r\longrightarrow 0}u^{\ast }$decreasing as $r\rightarrow 0$,
has a minimum point, and then is increasing to $\infty $ as $r\rightarrow
e^{T}$. Namely we obtain a solution $u$, satisfying (\ref{dd}), often called
large solution.$\medskip $

Moreover all the trajectories with one point in $\mathcal{R}_{3}$ present
the same type of behaviour, corresponding to large solutions.\medskip

$\bullet $ The trajectory $\mathcal{T}_{4}$ starts from $\mathbf{M}_{0}$ as $%
t\rightarrow -\infty $, in the region $\mathcal{R}_{4}$, where $%
V_{t}<0<G_{t} $ which is positively invariant, so it stays in it. Then $V$
is bounded. Here also it is impossible that $G$ to be bounded, since there
is no other fixed point in $\overline{\mathcal{W}}$. Consider any
corresponding solution $u(r)=U(t)$ defined on an maximal interval $(-\infty
,T)$.Then $\lim_{t\longrightarrow T}G=\infty ;$ and $\lim_{t\longrightarrow
T}V=l\geq 0$. Then $G_{t}\sim _{t\longrightarrow T}(p-1)G^{\frac{p}{p-1}}$,
thus $(G^{-\frac{1}{p-1}})_{t}\sim _{t\longrightarrow T}-1$, which is
impossible if $T=\infty $ since $\lim_{t\longrightarrow T}G^{-\frac{1}{p-1}%
}=0$. Then $T$ is finite and $G^{-\frac{1}{p-1}}=S=-\frac{U_{t}}{U}\sim
_{t\longrightarrow T}(T-t)^{-1};$ if $l>0$, $l\sim _{t\longrightarrow
T}e^{T(\theta +p)}U^{q+1-p}$, then $U$ has a positive limit, and this is
contradictory because $(T-t)^{-1}$ is not integrable. Then $%
\lim_{t\longrightarrow T}V=0$, that means $U$ is vanishing at $T$, and $u$
satisfies (\ref{cc}). All the trajectories with one point in the region $%
\mathcal{R}_{4}$ have the same type of behaviour, so they lead to solutions $%
U$ which are decreasing and vanish in finite time.\medskip

$\bullet $ The trajectory $\mathcal{T}_{2}$ ends to the point $\mathbf{M}%
_{0} $, in the region $\mathcal{R}_{2}$. Indeed we have seen that the slope
of the function $F_{1}$ at $\mathbf{M}_{0}$ is smaller than the slope of $%
\mathcal{T}_{2}$. This region is negatively invariant, so $\mathcal{T}_{2}$
stays in it. As in the case of region $\mathcal{R}_{3}$, we obtain that $V$
cannot be bounded, and that the trajectory is defined in a maximal interval $%
(T,\infty )$, such that $T$ is finite, and $\lim_{t=T}V=\infty
=\lim_{t\longrightarrow T}G$. The corresponding functions $u$ are decreasing
and satisfy $u\sim _{r\longrightarrow \infty }u^{\ast }$ and $%
\lim_{r\longrightarrow R}u=\infty $, where $R=e^{T}$, satisfying (\ref{gg}%
).\medskip

$\bullet $ There exists two types of trajectories with one point in $%
\mathcal{R}_{2}$:\medskip

(a) the trajectories with one point above $\mathcal{T}_{3}$ cross the line $%
\mathcal{L}$ and then pass into the region $\mathcal{R}_{3}$, where we still
have established the behaviour. They correspond to solutions defined on a
finite maximal interval $(R_{1},R_{2})$, such that $\lim_{r\longrightarrow
R_{1}}u=\infty =\lim_{r\longrightarrow R_{2}}u=\infty $, with a minimum
point inside, satisfying 
\begin{equation}
\lim_{r\rightarrow R_{1}}u=\infty ,\text{ \qquad\ }\lim_{r\rightarrow
R_{2}}u=\infty ;\text{ }  \label{hh}
\end{equation}

(b) the trajectories with one point under $\mathcal{T}_{3}$ cross the curve $%
\mathcal{C}_{2}$ and pass into the region $\mathcal{R}_{4}$ where we have
established the behaviour. They correspond to solutions defined on a finite
maximal interval $(R_{1},R_{2})$, such that $\lim_{r\longrightarrow
R_{1}}u=\infty $ and vanishing at $R_{2}$, satisfying 
\begin{equation}
\lim_{r\rightarrow R_{1}}u=\infty ,\qquad \lim_{r\rightarrow R_{2}}u=0.
\label{mm}
\end{equation}

$\bullet $ All the trajectories with one point in $\mathcal{R}_{1}$ above $%
\mathcal{T}_{1}$ cross the curve $C_{2}$ and enter $\mathcal{R}_{3}$, with
the same behaviour as $\mathcal{T}_{3}$. The corresponding solutions $u$
satisfy (\ref{dd}). All the trajectories with one point in $\mathcal{R}_{1}$
under $\mathcal{T}_{1}$ enter $\mathcal{R}_{4}$ and stay under $\mathcal{T}%
_{4}$, and present the same behaviour, and the corresponding solutions $u$
satisfy (\ref{ee}).$\medskip $

$\bullet $ The point $\mathbf{A}_{2}$ is a saddle point, its eigenvalues are 
$\rho _{2}<0<\eta _{2}$, in $\mathbb{R\times }\left[ 0,\infty \right) $
there is precisely one trajectory $\mathcal{T}_{5}$ starting from $\mathbf{A}%
_{2}$ and directed by $\overrightarrow{v_{2}}$, with a negative slope, and
two trajectories, ending at $\mathbf{A}_{2}$ and directed by $%
\overrightarrow{u_{2}}=(1,0)$ not admissible, since that are contained in
the set $\left\{ V=0\right\} $. The slope of $F_{2}$ at point $\mathbf{A}%
_{2} $ equal to is $M_{2}=p\left\vert G_{1}\right\vert ^{\frac{p-2}{p-1}%
}G_{1}-N+p $ is greater than the slope $m_{2}$ of $\overrightarrow{v_{2}}$,
thus $\Theta _{1}$ starts in the region $\left\{ (G,V)\mid
V>F(G),G<G_{2}\right\} \mathcal{\ }$contained in $\mathcal{R}_{3}$. The
corresponding solutions satisfy (\ref{ff}), and by the scaling (\ref{scal}),
for any $k>0$ there exists a unique solution $u$ satisfying (\ref{ff}%
).\medskip

$\bullet $The region $\mathcal{R}_{5}$ is negatively invariant. The
trajectories with one point $(g,v)\in C_{2}$ pass from $\mathcal{R}_{5}$ to $%
\mathcal{R}_{3}$ and stay under $\mathcal{T}_{5}$, and the corresponding
solutions $u$ are defined on an interval $(R_{1},R_{2})$ satisfy 
\begin{equation}
\lim_{r\rightarrow R_{1}}u=0,\qquad \lim_{r\rightarrow R_{2}}u=\infty .
\label{hm}
\end{equation}%
Moreover either $g<0$, and then $G<0$, thus $u$ is increasing, or $g>0$,
then $u$ has a maximum and a minimum point. Or eventually when $p>2$, $\mu
<0 $, there can exist (possibly mulptiple) trajectories tangent to the axis $%
G=0 $ at the point $(0,-\mu )$, such that $u$ is nondecreasing, with an
inflexion point.\bigskip

(ii) Next we suppose $\mu =\mu _{0}$, $\gamma >S_{1}$ (see Figure 2). Here a
great part of the analysis of Theorem \ref{H1rad} is still available. The
difference is that the point $\mathbf{A}_{1}$ is critical: from (\ref{val})
the eigenvalues are $\rho _{1}=pS_{1}-N+p=0$ and $\eta _{1}=(q+1-p)(\gamma
-S_{1})>0$ and the eigenvectors are $\overrightarrow{u_{1}}=(1,0)$ for $\rho
_{1}$ and $\overrightarrow{v_{1}}=(1,-\eta _{1})$, with a negative slope.$%
\medskip $

$\bullet $ There still exists a unique trajectory $\mathcal{T}_{5}$,
associated to $\eta _{1},$ such that the correponding solutions $u$ have the
behaviour (\ref{dd}).$\medskip $

$\bullet $ There exists a global trajectory joining $\mathbf{A}_{1}$ to $%
\mathbf{M}_{0}$, and it is still unique: indeed $\gamma \neq 0$, then $%
\mathbf{M}_{0}$ is a saddle point, so this trajectory is precisely $\mathcal{%
T}_{1}$ from Lemma \ref{fix}. This trajectory is on the central manifold of $%
\mathbf{A}_{1}$, it converges to $\mathbf{A}_{1}$ as $t\rightarrow -\infty $%
, and it tangent to the axis $V=0$.\medskip

For giving a more precise behaviour, we proceed as in Lemma \ref{critic}.
Setting $G=G_{1}+\overline{G}$, there holds $\lim_{t\longrightarrow -\infty }%
\frac{V}{\overline{G}}=0$ and%
\begin{equation*}
\left\{ 
\begin{array}{ccc}
\overline{G}_{t} & = & F_{1}(G_{1}+\overline{G})-V, \\ 
V_{t} & = & (q+1-p)(\gamma -(G_{1}+\overline{G})^{\frac{1}{p-1}})V.%
\end{array}%
\right.
\end{equation*}%
Now 
\begin{equation*}
F_{1}(G_{1}+\overline{G})=F_{1}(G_{2})+F_{1}^{\prime }(G_{2})\overline{G}+%
\frac{1}{2}F_{1}^{\prime \prime }(G_{1}+\tau \overline{G})\overline{G}^{2}=%
\frac{p}{2(p-1)}(G_{1}+\tau \overline{G})^{\frac{2-p}{p-1}}\overline{G}^{2}
\end{equation*}%
for some $\tau \in \left( 0,1\right) ,$ thus $F_{1}(G_{1}+\overline{G})\sim
_{t\longrightarrow -\infty }c\overline{G}^{2}$ with $c=\frac{1}{2}\frac{p}{%
p-1}S_{1}^{2-p},$ and 
\begin{equation*}
\overline{G}_{t}=c\overline{G}^{2}(1+O(\overline{G}))-V.
\end{equation*}%
There holds $V_{t}\sim _{t\longrightarrow -\infty }a_{1}V$ with $%
a_{1}=(q+1-p)(\gamma -S_{1})>0$. Thus $V=O(e^{a_{1}t/2})$ as $t\rightarrow
-\infty $. And $\overline{G}_{t}\leq 2c\overline{G}^{2}$ , hence $\overline{G%
}\geq \frac{C^{\prime }}{\left\vert t\right\vert }$for $t$ small enough,
thus $V=o(\overline{G}^{3})$. Therefore $\overline{G}_{t}\sim
_{t\longrightarrow -\infty }c\overline{G}^{2}$. We deduce that $\overline{G}%
(t)\sim _{t\longrightarrow -\infty }\frac{1}{c\left\vert t\right\vert }$; in
turn $\overline{G}_{t}=c\overline{G}^{2}(1+O(\frac{1}{\left\vert
t\right\vert }))=c\overline{G}^{2}+O(\frac{1}{\left\vert t\right\vert ^{3}})$%
\textbf{\ }and finally by integration $\overline{G}(t)=\frac{1}{c\left\vert
t\right\vert }+O(\frac{\ln \left\vert t\right\vert }{\left\vert t\right\vert
^{2}})$; then 
\begin{equation*}
S=-\frac{U_{t}}{U}=S_{1}+\frac{S_{1}^{2-p}}{(p-1)c}\frac{1}{\left\vert
t\right\vert }+O(\frac{\ln \left\vert t\right\vert }{\left\vert t\right\vert
^{2}})=S_{1}-\frac{2}{p}\frac{1}{t}+O(\frac{\ln \left\vert t\right\vert }{%
\left\vert t\right\vert ^{2}}),
\end{equation*}%
then thus $\ln (Ue^{S_{1}t}\left\vert t\right\vert ^{-\frac{2}{p}})=O(\frac{%
\ln \left\vert t\right\vert }{\left\vert t\right\vert ^{2}})$, and $\frac{%
\ln \left\vert t\right\vert }{\left\vert t\right\vert ^{2}}$ is integrable,
then $Ue^{S_{1}t}\left\vert t\right\vert ^{-\frac{2}{p}}$ admits some limit $%
\ell >0$, which is precisely (\ref{aaa}); by the scaling (\ref{scal}), using
the fact that $\gamma \neq S_{1}$, for any $\ell >0$\ there exists a unique
solution $u$\ satisfying (\ref{aaa}). Note that there is an infinity of
trajectories on the central manifold, corresponding to solutions $u$\
satisfying (\ref{ee}).
\end{proof}

\begin{remark}
\label{loc} In the description of the radial solutions, we have obtained all
the global and all the local solutions near $0$ or $\infty $, corresponding
to all the trajectories converging to a fixed point, from Lemma \ref{conver}%
. Moreover our decription of the phase plane is more complete: we have
described the other trajectories, corresponding to maximal solutions in an
interval $(R_{1},R_{2})$, with $0<R_{1}<R_{2}<0$, showing the existence of
solutions satisfying (\ref{hh}), (\ref{mm}) or (\ref{hm}).
\end{remark}

Next we study the case ($\mathcal{H}_{2}$), where $\mu \geq \mu _{0}\quad $%
and $\gamma <S_{2}$. Here it can happen that $\gamma <0$, that means $\theta
+p<0$. It appears in particular when $\mu >0$, since $S_{1}<0$.

\begin{remark}
\label{Kelvin} Our idea is to deduce the case ($\mathcal{H}_{2}$) from ($%
\mathcal{H}_{1}$). In case $p=2$ it follows from the Kelvin transform $%
u(x)=\left\vert x\right\vert ^{2-N}v(\frac{x}{\left\vert x\right\vert ^{2}})$%
. Indeed it maps equation (\ref{casp2}) into 
\begin{equation*}
-\Delta v+\mu \frac{v}{\left\vert x\right\vert ^{2}}+\left\vert x\right\vert
^{\widetilde{\theta }}u^{q}=0
\end{equation*}%
where $\widetilde{\theta }=(N-2)q-(N+2+\theta ),$ then $\gamma =\frac{\theta
+2}{q-1}$ is replaced by $\widetilde{\gamma }=\frac{\widetilde{\theta }}{q-1}%
=N-2-\gamma $. In the radial case the inversion $x\longmapsto \frac{x}{%
\left\vert x\right\vert ^{2}}$ corresponds to a change of $t$ into $-t,$ and 
$\gamma >S_{2}\Longleftrightarrow \widetilde{\gamma }<S_{1}$ since $%
S_{1}+S_{2}=N-2,$ then the equation in $u$ satisfies ($\mathcal{H}_{2}$) if
and only if the equation in $v$ satisfies ($\mathcal{H}_{1}$). When $p\neq 2$
we cannot use this argument, but we replace the Kelvin tranform by suitable
symmetry properties of the phase plane to reduce the study.\bigskip
\end{remark}

\pagebreak

\begin{figure}[!h]
\begin{center}
 \includegraphics[keepaspectratio, width=11cm]{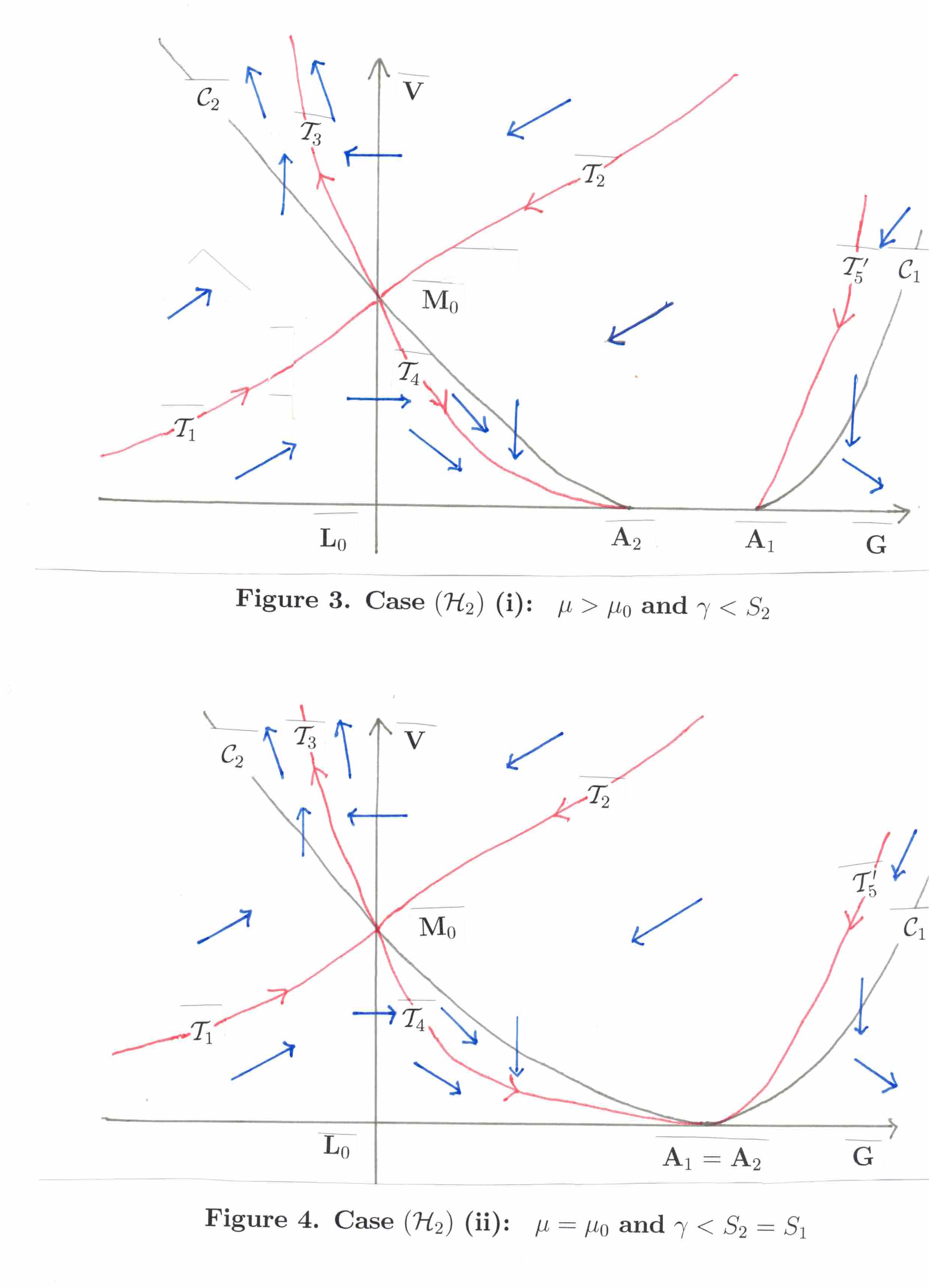}
 \end{center}
 \end{figure}

\begin{proof}[Proof of Theorem \protect\ref{H2rad}]
\textbf{Case (}$\mathcal{H}_{2}$\textbf{): }$\mu \geq \mu _{0}$\textbf{\ and 
}$\gamma <S_{2}$. Here the point $\mathbf{M}_{0}$ exists. \medskip

$\bullet $ When $\gamma \neq 0$ there exist precisely four trajectories $%
\mathcal{T}_{i}$, $i=1,..,4$ converging to $\mathbf{M}_{0}$ from Lemma \ref%
{fix} . The point $\mathbf{A}_{2}$ is a sink (if $\rho _{2}\neq \eta _{2}$), 
$\mathbf{A}_{1}$ is a saddle point when $\mu >\mu _{0}$. We note that 
\textbf{the phase plane has exactly the same shape} as the preceeding one,
after making $t\rightarrow -t$ and a transformation $G\longmapsto \varphi
(G) $ exchanging $G_{1}$ and $G_{2}$ (see Figure 3 for $\mu >\mu _{0}$ and
figure 4 for $\mu =\mu _{0}$). Thus as above we obtain the behaviour of the
four trajectories $\mathcal{T}_{i}$ and the existence of a unique trajectory 
$\mathcal{T}_{5}^{\prime }$ ending at $\mathbf{A}_{1}$ staying in the region 
$\left\{ (G,V)\mid V>F(G),G>G_{1}\right\} $. For simplicity we do not give
the detail of the proofs.\medskip

$\bullet $ The case $\gamma =0<S_{2}$ can happen when $\mu <0$. From Lemma %
\ref{gammazero}. When $p\leq 2$ there still exists precisely four
trajectories converging to $\mathbf{M}_{0}$ so we can conclude as above.
When $p>2,$ there exists at least a trajectory joining $\mathbf{A}_{1}=%
\mathbf{A}_{2}$ to $\mathbf{M}_{0}.$ of (\ref{pq}) satisfying This case is
delicate, because this trajectory could not be unique, since no
linearization is possible at $\mathbf{M}_{0},$ and the eigenvalues are $\rho
_{1}=0$ and $\rho _{2}=(q+1-p)(\gamma -S_{2})<0.$ Such trajectory correpond
to radial solutions $u$ of equation (\ref{pq}) satisfying (\ref{AAA}) with $%
\gamma =0,$ that is\label{att} 
\begin{equation*}
\lim_{r\rightarrow 0}u=\left\vert \mu \right\vert ,\qquad \lim_{r\rightarrow
\infty }r^{\frac{N-p}{p}}\left\vert \ln (r)\right\vert ^{\frac{2}{p}%
}u(r)=\ell >0.
\end{equation*}%
Here we show directly the uniqueness for given $\ell >0$, using a main
argument which will be used in the nonradial case: let $u$ and $\widetilde{u}
$ be two solutions with such a behaviour. then for any $\varepsilon >0,$ the
function $(1+\varepsilon )\widetilde{u}$ is a \textbf{supersolution }of (\ref%
{pq}) , greater than $u$ near $0$ and near $\infty .$ Then $(1+\varepsilon )%
\widetilde{u}\geq u,$ from Corollary \ref{DG}. As $\varepsilon
\longrightarrow 0$ we get $\widetilde{u}\geq u$ and then $\widetilde{u}=u.$ %
\label{dat}\medskip \medskip
\end{proof}

\begin{remark}
When $\gamma \neq 0$, the sign of $\gamma $\ is not apparent in our
conclusions. However, it appears in the behaviours of the different functions%
\textbf{\ }$u$.Recall that when $\mu >0$, then $S_{2}<0$, then $\gamma
<S_{2} $ implies $\gamma <0$. For example, consider the solutions satisfying
(\ref{AA}): when $0<\gamma <S_{2}$ , implying $\mu <0$, they are decreasing
from $\infty $ to $0$. When $\gamma <0<S_{2}$, they are increasing from $0$
to a maximum point, and then decreasing to \ $0$. When $\gamma =0<S_{2}$,
they are decreasing from $a^{\ast }$ to $0$.$\ $When $\gamma <S_{2}<0$ ,
implying $\mu >0$, they are increasing from $0$ to $\infty $. When $\gamma
<S_{2}=0$, implying $\mu =0$, they are increasing from $0$ to a
constant.\medskip \medskip
\end{remark}

\begin{remark}
Under the assumptions of Theorem \ref{H1rad}, As in the proof of Theorem \ref%
{H2rad}, there also exist solutions in an interval $(R_{1},R_{2})$, such
that respectively $\lim_{r\rightarrow R_{1}}u=\infty $,\ $\lim_{r\rightarrow
R_{2}}u=\infty $, or $\lim_{r\rightarrow R_{1}}u=0$, $\lim_{r\rightarrow
R_{2}}u=\infty $, or $\lim_{r\rightarrow R_{1}}u=\infty $, $%
\lim_{r\rightarrow R_{2}}u=0$.\medskip \medskip
\end{remark}

Next we consider the case ($\mathcal{H}_{3}$):\pagebreak

\begin{figure}[!h]
\begin{center}
 \includegraphics[keepaspectratio, width=11cm]{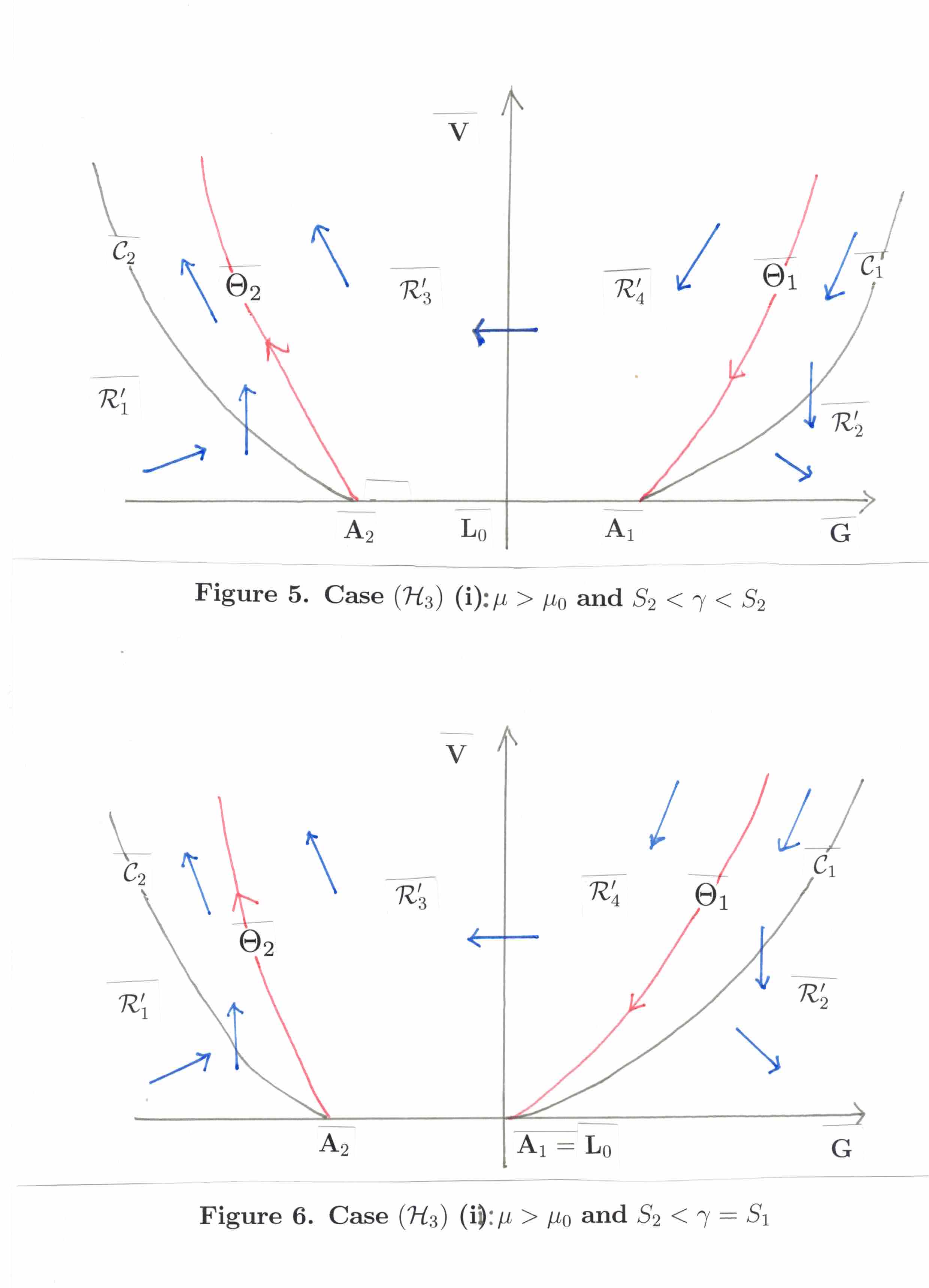}
 \end{center}
 \end{figure}
\medskip

\begin{proof}[Proof of Theorem \protect\ref{H3rad}]
\textbf{Case (}$\mathcal{H}_{3}$\textbf{): }$\mu >\mu _{0}$\textbf{\ and }$%
S_{2}\leq \gamma \leq S_{1}.$ Here $\mathbf{M}_{0}$ does not exist. We
consider the regions 
\begin{eqnarray*}
\mathcal{R}_{1}^{\prime } &\mathcal{=}&\left\{ V<F(G),G<G_{2}\right\}
,\qquad \mathcal{R}_{2}^{\prime }\mathcal{=}\left\{ V>F(G),G>G_{1}\right\} ,
\\
\mathcal{R}_{3}^{\prime } &\mathcal{=}&\left\{ V>F(G),G<G_{0}\right\}
,\qquad \mathcal{R}_{4}^{\prime }=\left\{ V<F(G),G>G_{0}\right\} .
\end{eqnarray*}

(i) We first assume that $S_{2}<\gamma <S_{1}$ (see Figure 5).\medskip

$\bullet $ The points $\mathbf{A}_{1}$ and $\mathbf{A}_{2}$ are saddle
points. For $i=1,2$, there exists a unique admissible trajectory, of
direction $\overrightarrow{v_{i}}$. the slope of $\overrightarrow{v_{2}}$,
eigenvector for $\eta _{2}>0$ is negative and the slope of $\overrightarrow{%
v_{1}}$ eigenvector for $\eta _{1}<0$ is positive. There is precisely one
trajectory $\Theta _{2}$ starting from $\mathbf{A}_{2}$ and one trajectory $%
\Theta _{1}$ ending in $\mathbf{A}_{1}$. And the slope of $\overrightarrow{%
v_{2}}$ is less than the slope of $F_{2}$ at $G_{2}$, and the slope of $%
\overrightarrow{v_{1}}$ is positive, greater than the slope of $F_{1}$ at $%
G_{1}$, so the two trajectories lie in the region 
\begin{equation*}
\left\{ V>F(G)\right\} =\mathcal{R}_{3}^{\prime }\cup \mathcal{R}%
_{4}^{\prime }\cup \mathcal{L},
\end{equation*}%
where $G_{t}<0$. The field on $\mathcal{C}_{1}$ is directed by $(0,-1)$ and
the field on $\mathcal{C}_{2}$ is directed by $(0,1)$. In $\mathcal{R}%
_{3}^{\prime }$,$\mathcal{R}_{4}^{\prime }$, there holds respectively $%
V_{t}>0$, and $V_{t}<0$, and $\mathcal{R}_{3}$ is positively invariant, and $%
\mathcal{R}_{4}^{\prime }\mathcal{\ }$is negatively invariant. Then $\Theta
_{2}$ stays in $\mathcal{R}_{3}^{\prime }$ and $\Theta _{1}$ in $\mathcal{R}%
_{4}^{\prime }$. As in case ($\mathcal{H}_{1}$), we get that the
corresponding solutions are local, respectively on $(0,R)$ and $(R^{\prime
},\infty )$, and satisfying (\ref{nn}) and (\ref{rr}).

$\bullet $ The solutions $u$ associated to other trajectories with one point
above $\Theta _{1}$ and $\Theta _{2}$ satisfy%
\begin{equation}
\lim_{r\rightarrow R_{1}}u=\infty ,\qquad \text{ }\lim_{r\rightarrow
R_{2}}u=\infty .  \label{tt}
\end{equation}

$\bullet $ The trajectories with a point in $\mathcal{R}_{3}^{\prime }$ and
under $\Theta _{2}$ have crossed the curve $\mathcal{C}_{2}$ and are issued
of $\mathcal{R}_{1}^{\prime }$, and the associated solutions $u$ satisfy 
\begin{equation}
\lim_{r\rightarrow R_{1}}u=0,\qquad \text{ }\lim_{r\rightarrow
R_{2}}u=\infty .  \label{yy}
\end{equation}

$\bullet $ The trajectories with one point under $\mathcal{C}_{1}$ cross $%
\mathcal{C}_{1}$ and then stay in $\mathcal{R}_{2}^{\prime }$, and the
solutions $u$ satisfy 
\begin{equation}
\lim_{r\rightarrow R_{1}}u=\infty ,\qquad \text{ }\lim_{r\rightarrow
R_{2}}u=0,  \label{vv}
\end{equation}

(ii) Next we assume $\mu >\mu _{0}$ and $\gamma =S_{1}$, hence $\gamma >0$
(see Figure 6). There still exists a unique trajectory $\Theta _{2}$
starting from $\mathbf{A}_{2}$ as above. At the point $\mathbf{A}_{1}$, the
eigenvalues are $\rho _{1}=pS_{2}-N+p>0$, and $\eta _{1}=(q+1-p)(\gamma
-S_{1})=0$. There still exists a unique trajectory, corresponding to $\rho
_{1}$, not admissible. There exists at least a trajectory $\mathcal{T}$ on
the central manifold, directed by the eigenvector $\overrightarrow{v_{1}}%
=(1,pS_{1}-N+p)$, which has a positive slope $m_{1};$ but a priori it can
converge to $\mathbf{A}_{1}$ as $t\rightarrow \infty $ or as $t\rightarrow
-\infty $. And this slope is equal to the slope of $F_{1}$, that is $%
M_{1}=p\left\vert G_{1}\right\vert ^{\frac{p-2}{p-1}}G_{1}-N+p$. Then $%
\mathcal{T}$ converges to $\mathbf{A}_{1}$ necessarily in the region $%
\mathcal{R}_{4}^{\prime }$ as $t\rightarrow \infty $. Setting $G=G_{1}+%
\overline{G}$, system (\ref{SGV}) can be written under the form%
\begin{equation*}
\left\{ 
\begin{array}{ccc}
\overline{G}_{t} & = & (pG_{1}^{\frac{1}{p-1}}-N+p)\overline{G}(1+o(1))-V,
\\ 
V_{t} & = & -\frac{q+1-p}{p-1}G_{1}^{\frac{2-p}{p-1}}\overline{G}V(1+o(1)),%
\end{array}%
\right.
\end{equation*}%
and $\frac{V}{\overline{G}}=m_{1}(1+o(1))$; then 
\begin{equation*}
V_{t}\sim _{t\longrightarrow \infty }-\frac{q+1-p}{p-1}\frac{G_{1}^{\frac{2-p%
}{p-1}}}{m_{1}}V^{2}=-V^{2}/b_{1},
\end{equation*}%
with $b_{1}=\frac{(p-1)(pS_{1}-N+p)S_{1}^{p-2}}{q+1-p}$. By integration we
get that $V=r^{\theta +p}u^{q+1-p}\sim _{t\longrightarrow \infty }\frac{b_{1}%
}{t}$, that means 
\begin{equation*}
u\thicksim _{r\longrightarrow \infty }r^{-S_{1}}(b_{1}/\ln r)^{\frac{1}{q+1-p%
}},
\end{equation*}%
so we get (\ref{rrr}) in any case. Our description of the solutions shows
that there is no global solution.\medskip

(iii) Next assume $\mu >\mu _{0}$ and $\gamma =S_{2}\neq 0$, then $\gamma >0$
if $\mu <0$, or $\gamma <0$ if $\mu >0$. Then $G_{1}$ is replaced by $G_{2}$
and $t$ by $-t$. Since $G_{2}\neq 0$, the result is similar and we get (\ref%
{nnn}).\medskip

(iv) Assume $\gamma =S_{2}=0$, that means $\mu =0$, and $p+\theta =0$. As we
mentioned in the proof of Lemma \ref{fixai}, the linearization (\ref{line})
at $\mathbf{A}_{2}=(0,0)$ is still valid with $\gamma =0$, even if $p>2$,
and the eigenvalues given by $\rho _{2}=pS_{2}-N+p<0$ and $\eta _{2}=0$. The
trajectory corresponding to $\rho _{2}$ is still nonadmissible. There exists
a trajectory $\mathcal{T}$ on the central manifold, directed by the
eigenvector $\overrightarrow{v_{2}}=(1,-(N-p))$, then $\frac{V}{G}\sim
-(N-p) $. And $\mathcal{T}$ converges to $\mathbf{A}_{2}$ necessarily in the
region $\mathcal{R}_{3}^{\prime }$ as $t\rightarrow -\infty ;$ moreover 
\begin{equation*}
V_{t}=-(q+1-p)V\left\vert G\right\vert ^{\frac{2-p}{p-1}}G=(q+1-p)V\left%
\vert G\right\vert ^{\frac{1}{p-1}}\sim _{t\longrightarrow -\infty
}(q+1-p)(N-p)^{-\frac{1}{p-1}}V^{\frac{p}{p-1}}.
\end{equation*}%
By integration, we obtain 
\begin{equation*}
V=u^{q+1-p}\sim _{t\longrightarrow -\infty }(-\frac{(q+1-p)(N-p)^{-\frac{1}{%
p-1}}}{p-1}t)^{1-p},
\end{equation*}%
thus (\ref{nno}) follows.\bigskip
\end{proof}

Finally we study the doubly critical case ($\mathcal{H}_{4}$), and the case (%
$\mathcal{H}_{5}$) where $\mu <\mu _{0}.$\pagebreak
\begin{figure}[!h]
\begin{center}
 \includegraphics[keepaspectratio, width=11cm]{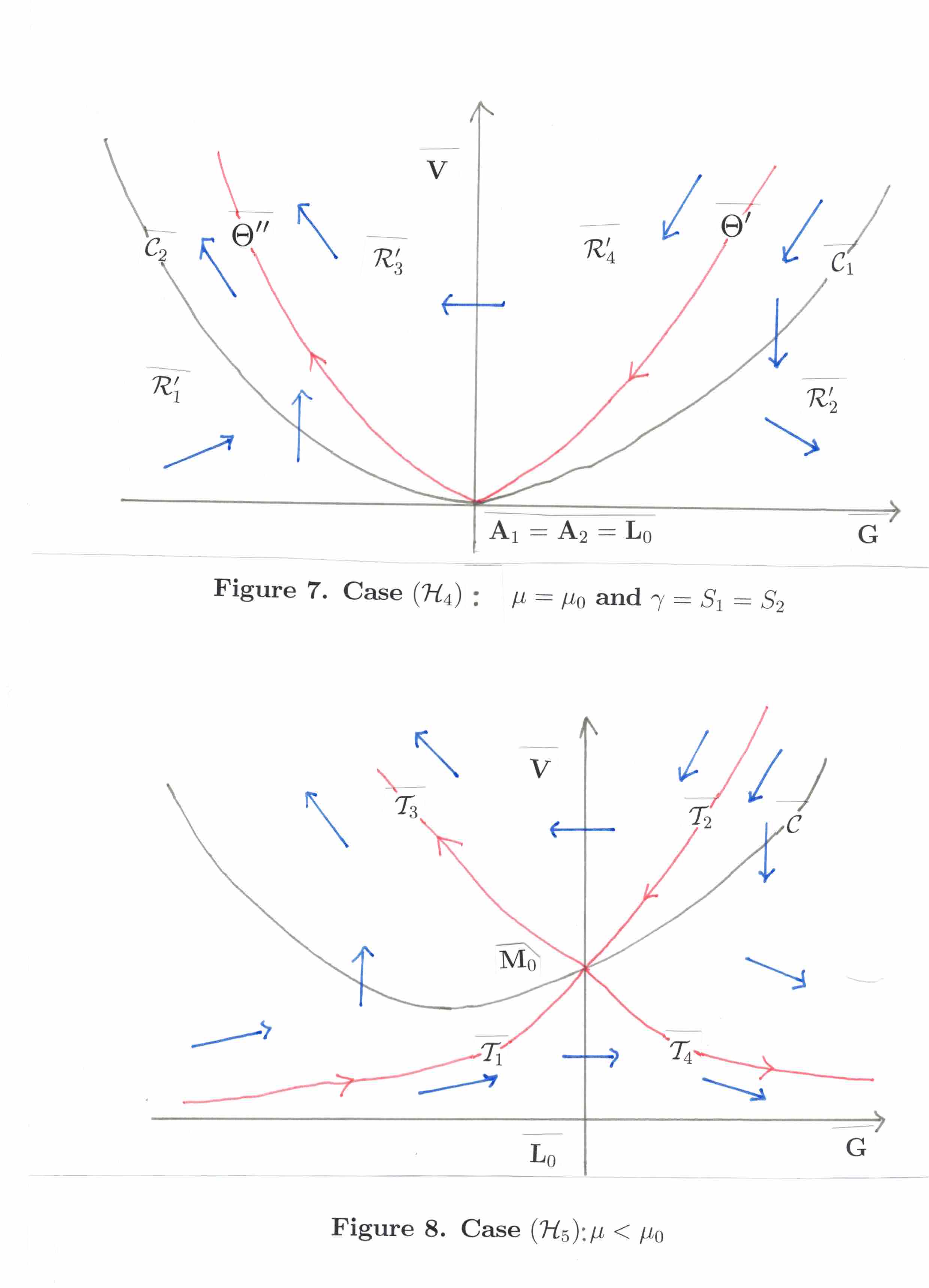}
 \end{center}
 \end{figure}

\begin{proof}[Proof of Theorem \protect\ref{H4rad}]
\textbf{Case (}$\mathcal{H}_{4}$\textbf{): }$\mu =\mu _{0}$\textbf{\ and }$%
\gamma =\frac{N-p}{p}.$ Here $\mathbf{M}_{0}$ does not exist, and $\mathbf{A}%
_{1}=\mathbf{A}_{2}=\mathbf{L}_{0}=((\frac{N-p}{p})^{p-1},0)$ (see Figure 7).%
$\medskip $

$\bullet $ The region 
\begin{equation*}
\mathcal{R}_{4}^{\prime }=\left\{ (G,V)\mid V>F(G),G>G_{0}\right\}
\end{equation*}%
is negatively invariant, then any trajectory defined on a maximal interval $%
(T_{1},T_{2})$ (finite or not) with one point in this region stays in it as $%
t\rightarrow T_{1}$, and necessarily $T_{1}$ is finite. Let $\mathcal{U}$
(resp. $\mathcal{V}$) be the set of points $P$ of $\mathcal{R}_{4}^{\prime }$
such that the trajectory passing by $P$ cuts the curve $\mathcal{C}_{2}$
(resp. the line $\mathcal{L}$ ); then $\mathcal{U}$ (resp. $\mathcal{V}$) is
an open set in $\mathcal{R}_{4}^{\prime }$, since the intersections are
transverse. Then $\mathcal{U}$ $\cup $ $\mathcal{V\neq R}_{4}^{\prime }$.
Then there exists at least one trajectory $\Theta ^{\prime }$ starting in $%
\mathcal{R}_{4}^{\prime }$ and converging to $\mathbf{L}_{0}=(G_{0},0)$ as $%
t\rightarrow \infty $. Similarly there exists at least one trajectory $%
\Theta ^{\prime \prime }$ in region 
\begin{equation*}
\mathcal{R}_{3}^{\prime }=\left\{ (G,V)\mid V>F(G),G<G_{0}\right\}
\end{equation*}%
converging to $\mathbf{L}_{0}=(G_{0},0)$ as $t\rightarrow -\infty $.$%
\medskip $

$\bullet $ Here for more precision we still use the energy function defined
at (\ref{fv}) with $\gamma =\frac{N-p}{p}$, equivalently $q=\mathbf{q}_{s}:$ 
\begin{equation*}
\mathcal{E}=V^{\frac{p}{q+1-p}}(\frac{F(G)}{p}-\frac{V}{q+1})=V^{\frac{p}{%
q+1-p}}(\frac{G_{t}+V}{p}-\frac{V}{q+1})=\frac{1}{p}V^{\frac{p}{q+1-p}%
}(G_{t}+\frac{(q+1-p)V}{q+1})
\end{equation*}%
satisfies $\mathcal{E}_{t}=0$, since then $D=N-p-p\gamma =0$, so that $%
\mathcal{E}$ is constant. Consider any trajectory converging to $\mathbf{A}%
_{1}=\mathbf{A}_{2}=((\frac{N-p}{p})^{p-1},0)$ as $t\rightarrow -\infty $
(or as $t\rightarrow \infty )$, we get that $\mathcal{E}$ tends to $0$. So
on such a trajectory, $\mathcal{E\equiv }0$. This gives a \textbf{first
integral} satisfied by such solution: there holds 
\begin{equation*}
\left\{ 
\begin{array}{ccc}
G_{t} & = & -\frac{(q+1-p)V}{q+1}, \\ 
V_{t} & = & (q+1-p)V(\gamma -\left\vert G\right\vert ^{\frac{2-p}{p-1}}G),%
\end{array}%
\right.
\end{equation*}%
then $G_{t}<0$, and $V_{t}+(q+1)(\gamma -\left\vert G\right\vert ^{\frac{2-p%
}{p-1}}G)G_{t}=0$. By integration, 
\begin{equation*}
\frac{V}{q+1}+\gamma G-\frac{p-1}{p}\left\vert G\right\vert ^{\frac{p}{p-1}%
}=C,
\end{equation*}%
\begin{equation*}
G_{t}+\frac{(q+1-p)V}{q+1}=0=G_{t}+(q+1-p)(-\gamma G+\frac{p-1}{p}\left\vert
G\right\vert ^{\frac{p}{p-1}}+C;
\end{equation*}%
and $C=(q+1-p)(\frac{N-p}{p}G_{0}-\frac{p-1}{p}G_{0}^{\frac{p}{p-1}})=\frac{1%
}{p}(q+1-p)\gamma ^{p}$, so that finally 
\begin{equation}
G_{t}+\frac{q+1-p}{p}F(G)=0.  \label{gt}
\end{equation}%
In the case $p=2$ we find 
\begin{equation*}
G_{t}+\frac{q-1}{2}(G-\frac{N-2}{2})^{2}=0,
\end{equation*}%
then $G=\frac{N-2}{2}+\frac{2}{(q-1)(t+C)}$, thus $\frac{V_{t}}{V}=-\frac{2}{%
t+C}$, implying $V=\frac{C_{1}}{(t+C)^{2}}$, and 
\begin{equation*}
\frac{(q-1)V}{q+1}=\frac{q-1}{q+1}\frac{C_{1}}{(t+C)^{2}}=-G_{t}=\frac{2}{q-1%
}\frac{1}{(t+C)^{2}},
\end{equation*}%
then $C_{1}=\frac{2(q+1)}{(q-1)^{2}}$and $V^{\frac{1}{q-1}}=r^{\frac{N-2}{2}%
}u$, so $u$ satisfies (\ref{co}).\medskip

\noindent In the general case $p>1$, we can solve the equation (\ref{gt}) by
quadrature, and get 
\begin{equation*}
-\frac{q+1-p}{p}(t+C)=\int_{G_{0}}^{G}\frac{dg}{F(g)}.
\end{equation*}%
Setting $G=G_{0}+\overline{G}$, there holds $F(G)\sim _{G\longrightarrow
G_{0}}c\overline{G}^{2}$, with $c=\frac{pG_{0}^{\frac{2-p}{p-1}}}{2}$; then $%
\overline{G}\sim -\frac{p}{c(q+1-p)}\frac{1}{t};$ and 
\begin{equation*}
\frac{(q+1-p)V}{q+1}=-G_{t}=\frac{q+1-p}{p}F(G)\sim _{t\longrightarrow \pm
\infty }\frac{1}{c}\frac{p}{q+1-p}\frac{1}{t^{2}},
\end{equation*}%
so that $u(r)=r^{-\frac{N-p}{p}}V^{\frac{1}{q+1-p}}$satisfies (\ref{ao}) or (%
\ref{bo}).\bigskip
\end{proof}

\begin{proof}[Proof of Theorem \protect\ref{H5rad}]
\textbf{Case (}$\mathcal{H}_{5}$\textbf{): }$\mu <\mu _{0}.$ Here the point $%
\mathbf{M}_{0}$ still exists (independently of the value of $\gamma $) and
it is the unique fixed point (see Figure 8). Any trajectory of global
solution of the system must join some fixed points, from Lemma \ref{conver},
then it is unique, reduced to the point $\mathbf{M}_{0}$, so $u^{\star }$ is
the unique global solution. There exists trajectories $\mathcal{T}_{i}$ ($%
i=1,..,4$) as before, and the corresponding solutions $u$ satisfy
respectively (\ref{t1}),(\ref{t2}),(\ref{t3}),(\ref{t4}). And all the
trajectories converging to $\mathbf{M}_{0}$ present one of these
types.\bigskip
\end{proof}

\textbf{Acknowledgment} $\quad $H. Chen is supported by NSFC of China, No.
12071189, 12361043, by Jiangxi Province Science Fund No. 20232ACB201001.

\end{document}